\documentclass[11pt,reqno]{amsart}
\usepackage[letterpaper,textheight=9.25in,textwidth=7.1in,centering]{geometry}
\usepackage[T1]{fontenc}
\usepackage{placeins}
\usepackage{lmodern,microtype,amsmath,amssymb,mathtools,mathrsfs}
\microtypesetup{expansion=false}
\usepackage{amsthm,amscd}
\usepackage{hyperref}
\hypersetup{hidelinks,
 pdftitle={Ternary polynomial relations and periods of cusp forms},
 pdfauthor={Vasily Ilin and Ken Ono}}
\usepackage{array,enumitem,needspace,booktabs}
\numberwithin{equation}{section}
\setlist[enumerate]{leftmargin=*,itemsep=3pt,topsep=4pt}

\newtheorem{theorem}{Theorem}[section]
\newtheorem{proposition}[theorem]{Proposition}
\newtheorem{lemma}[theorem]{Lemma}
\newtheorem{corollary}[theorem]{Corollary}
\theoremstyle{definition}

\newtheorem*{example}{Example}
\theoremstyle{remark}

\theoremstyle{plain}

\newcommand{\QQ}{\mathbb Q}
\newcommand{\ZZ}{\mathbb Z}
\newcommand{\FF}{\mathbb F}
\newcommand{\CC}{\mathbb C}
\newcommand{\RR}{\mathbb R}
\newcommand{\St}{\operatorname{St}}
\newcommand{\Sym}{\operatorname{Sym}}
\newcommand{\Ind}{\operatorname{Ind}}
\newcommand{\Hom}{\operatorname{Hom}}
\newcommand{\Ext}{\operatorname{Ext}}
\newcommand{\Res}{\operatorname{Res}}
\newcommand{\sgnchar}{\operatorname{sgn}}
\newcommand{\vol}{\operatorname{vol}}

\newcommand{\GL}{\operatorname{GL}}
\newcommand{\SL}{\operatorname{SL}}

\title[Ternary polynomial relations and periods of cusp forms]{Ternary polynomial relations and periods of cusp forms}

\author{Vasily Ilin}
\address{Axiom Math, 124 University Avenue, Palo Alto, CA 94301}
\email{vasily@axiommath.ai}
\author{Ken Ono}
\address{Axiom Math, 124 University Avenue, Palo Alto, CA 94301}
\email{ken@axiommath.ai}

\date{}
\newcommand{\Z}{\ZZ}
\newcommand{\cC}{\mathcal C}\newcommand{\cE}{\mathcal E}
\newcommand{\bn}{\overline N}\newcommand{\cp}{C_+}\newcommand{\vect}{\mathbf x}
\newcommand{\SO}{\mathrm{SO}}

\begin{document}

\begin{abstract}
Bachmann and van Ittersum conjectured that every polynomial
$f\in\QQ[x,y,z]$ satisfying
\[
  f(x,y,z)=f(y,z,x),\qquad
  f(x,y,z)+f(-y,x-y,z)+f(y-x,-x,z)=0
\]
is a sum of polynomials in two variables. We prove this conjecture
over every field of characteristic zero and determine all homogeneous
solutions. Moreover, restriction to $z=0$ identifies the
degree $n$ solution space with the space of binary period polynomials
of degree $n$ normalized to vanish on both coordinate axes.
The solution corresponding to such a polynomial $q$ is
\[
  f(x,y,z)=q(x,y)+q(y,z)+q(z,x).
\]
There are no nonzero homogeneous solutions of odd degree.
For even degrees $n$, the symmetric and alternating
solution spaces are each isomorphic over $\CC$ to the space of weight $n+2$
cusp forms on $\SL_2(\ZZ)$.  The proof connects the polynomial relations with the cohomology
of arithmetic groups through the rational Tits building for $\GL_3$.
Cusp forms enter both the description of the solutions
and the proof that there are no others. The main theorems and the essential
argument are formalized unconditionally in Lean~4 by AxiomProver.
\end{abstract}

\maketitle

\section{Introduction and statement of results}
Classical period polynomials associated with modular forms satisfy
linear relations in two variables. Bachmann and van Ittersum
considered a three-variable system built from one of these relations
\cite[p.~481]{ABBI26}. Let $F$ be a field of characteristic zero,
and suppose that $f\in F[x,y,z]$ satisfies
\begin{equation}\label{eq:intro-cyclic}
  f(x,y,z)=f(y,z,x)
\end{equation}
and
\begin{equation}\label{eq:intro-norm}
  f(x,y,z)+f(-y,x-y,z)+f(y-x,-x,z)=0.
\end{equation}
The first identity imposes cyclic symmetry. The second is the
binary period relation in $x,y$, with $z$ fixed. Its substitution
$(x,y)\mapsto(-y,x-y)$ has order three. We denote the sum over its
three iterates by $Nf$, the \emph{three-term norm}. Cyclic symmetry
gives the same relation in the other two pairs of variables.

Bachmann and van Ittersum conjectured that every rational solution
is a sum of polynomials in two variables. Equivalently, no monomial
of a solution involves all three variables. We prove this conjecture
over every field of characteristic zero, and we show that
every homogeneous solution of odd degree vanishes.

\begin{theorem}\label{thm:main}
Let $F$ be a field of characteristic zero. Every polynomial
$f\in F[x,y,z]$ satisfying \eqref{eq:intro-cyclic} and
\eqref{eq:intro-norm} can be written as
\begin{equation}\label{eq:intro-decomposition}
 f(x,y,z)=P(x,y)+Q(y,z)+R(z,x),\qquad P,Q,R\in F[X,Y].
\end{equation}
Equivalently, $f$ satisfies $\partial_x\partial_y\partial_z f=0$.
Every homogeneous solution of odd degree is zero.
\end{theorem}

In fact, restriction to $z=0$ determines every solution. Write
$F[x,y,z]_n$ and $F[X,Y]_n$ for the homogeneous degree $n$ spaces,
and $W_n(F)$ for the degree $n$ solutions. Interchanging $x,y$ preserves
the relations, so cyclic symmetry gives the decomposition into symmetric
and alternating subspaces
\[
  W_n(F)
    =W_n^{\mathrm{sym}}(F)\oplus W_n^{\mathrm{alt}}(F).
\]
The reconstruction uses the cyclic sum of a binary polynomial $q$:
\[
  \Phi(q)(x,y,z)=q(x,y)+q(y,z)+q(z,x).
\]

\begin{theorem}\label{thm:symmetric-classification}
Let $n\ge0$ be even. A symmetric polynomial $f\in F[x,y,z]_n$ satisfies
\eqref{eq:intro-norm} if and only if
\begin{equation}\label{eq:main-classification}
 f=\Phi(q),
\end{equation}
where $q\in F[X,Y]_n$ is symmetric, odd in each variable, and satisfies
\begin{equation}\label{eq:binary-period-intro}
 q(X,Y)+q(-Y,X-Y)+q(Y-X,-X)=0.
\end{equation}
The binary polynomial is uniquely $q(X,Y)=f(X,Y,0)$.
\end{theorem}

To describe both symmetry types, we let $\mathscr P_n^0(F)$
be the space of degree $n$ homogeneous binary polynomials
satisfying
\begin{equation}\label{eq:normalized-period-space}
\begin{gathered}
  q(X,0)=q(0,Y)=0,\qquad q(X,Y)+q(-Y,X)=0,\\
  q(X,Y)+q(-Y,X-Y)+q(Y-X,-X)=0.
\end{gathered}
\end{equation}
These are the \emph{normalized binary period polynomials}
\cite[Section~2]{PasolPopa}. For positive even $n$, subtracting
$q(1,0)(X^n-Y^n)$ from a polynomial satisfying the two period relations
makes it vanish on both axes without changing those relations.
Write $\mathscr P_n^{0,\mathrm{ev}}(F)$ and
$\mathscr P_n^{0,\mathrm{odd}}(F)$ for the subspaces even and odd
in $X$, respectively.

\begin{corollary}\label{cor:complete-periods}
For every $n\ge0$, restriction is an $F$-linear isomorphism
\begin{equation}\label{eq:period-restriction}
 W_n(F)\longrightarrow\mathscr P_n^0(F),\qquad f\longmapsto f(X,Y,0).
\end{equation}
Its inverse is $q\mapsto\Phi(q)$. It identifies
$W_n^{\mathrm{sym}}(F)$ with $\mathscr P_n^{0,\mathrm{odd}}(F)$ and
$W_n^{\mathrm{alt}}(F)$ with $\mathscr P_n^{0,\mathrm{ev}}(F)$.
\end{corollary}

The two normalized binary period spaces have a classical
interpretation in terms of cusp forms. Let
$\mathfrak H=\{\tau\in\CC:\Im\tau>0\}$ be the upper half-plane,
and let $S_k=S_k(\SL_2(\ZZ))$ denote the complex vector space
of holomorphic cusp forms of weight $k$ for the full modular group.
For even $n>0$ and $h\in S_{n+2}$, define its period polynomial by
\begin{equation}\label{eq:analytic-period}
  r_h(X,Y)=\int_0^{i\infty}h(\tau)(\tau Y-X)^n\,d\tau.
\end{equation}
The integral converges by cuspidality. Taking the odd part in $X$,
or normalizing the even part as above, gives maps from $S_{n+2}$
to $\mathscr P_n^{0,\mathrm{odd}}(\CC)$ and
$\mathscr P_n^{0,\mathrm{ev}}(\CC)$, respectively.
Eichler--Shimura theory states that both maps are complex linear
isomorphisms \cite{KohnenZagier,PasolPopa,ZagierPeriods}.

\begin{corollary}\label{cor:dimensions}
For every even $n>0$, the period maps induce complex linear isomorphisms
\[
 W_n^{\mathrm{sym}}(\CC)\cong S_{n+2},\qquad
 W_n^{\mathrm{alt}}(\CC)\cong S_{n+2}.
\]
For every characteristic zero field $F$ and every $n\ge0$, we have
\begin{equation}\label{eq:dimensions}
 \dim_F W_n^{\mathrm{sym}}(F)=\dim_F W_n^{\mathrm{alt}}(F)
   =\dim_{\CC}S_{n+2}.
\end{equation}
All these spaces vanish when $n=0$ or $n$ is odd. Consequently, the generating function of the dimensions is
\begin{equation}\label{eq:W-Hilbert}
 \sum_{n\ge0}\dim_F W_n(F)t^n=\frac{2t^{10}}{(1-t^4)(1-t^6)}.
\end{equation}
\end{corollary}

The classification also determines a universal divisor. Set
\begin{equation}\label{eq:universal-divisor}
 D(x,y,z)=(x+y+z)(x+y)(y+z)(z+x).
\end{equation}
\begin{corollary}\label{cor:common-divisor}
Over every characteristic zero field, $D$ divides every solution of
\eqref{eq:intro-cyclic}--\eqref{eq:intro-norm}. Up to a nonzero scalar,
it is the greatest common divisor of all nonzero solutions.
\end{corollary}

\begin{example}
The first nonzero solutions occur in degree ten. The two binary
polynomials in the original problem \cite[p.~481]{ABBI26} can be
written as
\begin{align*}
  q_{\mathrm s}(X,Y)
    &=XY(X^2-Y^2)^2(X^2-4Y^2)(4X^2-Y^2),\\
  q_{\mathrm a}(X,Y)
    &=X^2Y^2(X^2-Y^2)^3.
\end{align*}
Both satisfy \eqref{eq:normalized-period-space}.
The first is symmetric and odd in each variable, and the second is
alternating and even in each variable. Their cyclic sums
$f_{\mathrm s}=\Phi(q_{\mathrm s})$ and
$f_{\mathrm a}=\Phi(q_{\mathrm a})$ form a basis of $W_{10}(F)$.
Over $\CC$, the binary polynomials $q_{\mathrm s}$ and
$q_{\mathrm a}$ are nonzero scalar multiples of the odd and
normalized even period polynomials, respectively, of Ramanujan's
weight twelve discriminant cusp form $\Delta(\tau)$.
\end{example}

\subsection*{Method of proof}
The difficulty is not to construct solutions from binary periods, but to
show that there are no others. The norm relation mixes two-variable
monomials with those involving all three variables. For example,
\begin{equation}\label{eq:support-mixing-example}
  y^2z\longmapsto(x-y)^2z=x^2z-2xyz+y^2z.
\end{equation}
Thus, we cannot discard the monomials that we wish to
exclude.

The geometric part of the proof takes place on the rational
Tits building for $\GL_3$. Roughly speaking, this building
organizes the rational subspaces of $\QQ^3$ and the ways in
which they arise from changes of basis. The functional equations
will become compatibility relations among certain distinguished
cycles in this building, and the original polynomial $f$ arises by evaluating the resulting functional on one
particular cycle.

We now describe this picture concretely. The substitutions in the
functional equations are invertible linear changes of variables with
integer coefficients, so we are naturally led to ordered bases of
$\QQ^3$, represented by primitive integral vectors, and to the rational
subspaces that they span. The rational Tits building for $\GL_3$ is the graph whose
vertices are the rational lines and planes in $\QQ^3$, with an edge
joining a line to a plane precisely when the line is contained in the
plane.

In rank three the building is one-dimensional, hence a graph, and
its apartments can be described explicitly. An ordered basis
$(v_1,v_2,v_3)$ determines the three lines
\[
    L_i=\QQ v_i
\]
and the three planes
\[
    P_{ij}=\QQ v_i+\QQ v_j.
\]
Each line $L_i$ lies in the two planes obtained by adjoining either
of the other two basis vectors; among these six subspaces, these are
the only incidences. Their incidence graph is therefore the hexagon
\[
    L_1\longrightarrow P_{12}\longrightarrow L_2
    \longrightarrow P_{23}\longrightarrow L_3
    \longrightarrow P_{31}\longrightarrow L_1.
\]
Such an oriented hexagon is called an \emph{apartment}. Changing the
basis changes the apartment, and the three-term polynomial relation
will correspond to a relation among these apartment cycles.

The congruence condition defining the restricted graph is preserved
by a natural subgroup of $\GL_3(\ZZ)$. Let
\[
   K=\{g\in\GL_3(\ZZ):(1,1,1)g\equiv(1,1,1)\pmod 3\}.
\]
Thus $K$ consists of the integral changes of basis that preserve the
coordinate sum functional modulo three, and it acts on the restricted
building.

Proposition~\ref{thm:good-apartment} gives the crucial
translation from the original polynomial problem to this geometry.
A solution determines a $K$-equivariant polynomial valued functional
on the cycles of the restricted graph, and conversely the defining
relations of such a functional are exactly the polynomial relations
above. Its value on the apartment associated with the standard basis
$(e_1,e_2,e_3)$ is the original polynomial $f$.

The problem is therefore converted into an extension problem.
The functional is initially defined only on the restricted graph,
whereas the geometric support argument applies to the full Tits
building. Proposition~\ref{thm:boundary-evaluation} shows that if the
functional extends to the full building, then its value on the
standard apartment is a sum of polynomials in two variables.
It remains to construct such an extension in each symmetry and parity
case, or to prove that the corresponding solution vanishes.
Sections~\ref{sec:symmetric} and \ref{sec:odd-alt} carry this out.

The first task is the common support argument.  To carry it out, we
pass from the Tits building to an arithmetic quotient of the space of
positive definite quadratic forms in three variables, taken modulo
positive scaling.  Its Borel--Serre compactification has boundary
faces indexed by rational lines and planes, with intersections indexed
by flags $L\subset P$.  Thus the same incidence geometry that appears
in the building reappears at the boundary of the arithmetic quotient.

A functional on the full building gives a top-dimensional homology
class on this quotient.  We prove that such a class comes from the
boundary.  By Poincar\'e--Lefschetz duality, this reduces to the
vanishing of the corresponding inner cohomology.  An explicit positive
lower bound for the relevant Laplacian proves this vanishing in the
degrees that occur here; the remaining small degrees are handled
directly.  It therefore remains to understand how a boundary class
evaluates on the coordinate apartment.  The contributions from the
line and plane faces are controlled by elementary unipotent shears,
while the contributions at their intersections are controlled by the
coboundary relation joining the two adjacent face cocycles.  Both types
of contribution have two-variable support.  This proves the support
statement for any functional that extends to the full building.

We next turn to the extension problem.  For symmetric solutions,
finite differences produce an explicit polynomial $J_f$.  Its
vanishing is exactly the compatibility condition needed to extract
the successive binary periods used to construct the extension.
We interpret $J_f$ as a boundary coordinate of a degree two
cohomology class attached to the restricted-graph functional.
The class already vanishes on the boundary faces corresponding to
retained vertices, so only the omitted faces remain to be controlled.

For this purpose we construct closed differential forms from classical
cusp forms and sum their translates over the arithmetic quotient.
The sums converge in the range needed here.  We make the same
construction with dual coefficients.  Applying Stokes's theorem to
the original and dual forms gives two boundary identities; comparing
them shows that the plane restrictions determine the line restrictions.
In even degree a sign symmetry forces the remaining plane restriction
to vanish.  In odd degree an explicit nonzero integral, together with
the nonvanishing of outer critical modular $L$-values, shows that a
nonzero omitted-plane restriction would produce a nonzero retained-line
restriction.  This contradicts the vanishing already known there.
Hence $J_f$ vanishes in both parities.

Once $J_f=0$, binary period extraction constructs the symmetric
extension.  In odd degree one candidate term may initially have a
simple pole.  Its residue is a binary period polynomial, and the space
of residues that can arise is stable under the Hecke operators.
Liftability forces one extreme period coefficient to vanish on this
space.  If the residue space were nonzero, Hecke stability would
produce an eigenperiod for which the same coefficient is a nonzero
modular $L$-value.  Thus the residue space is zero, the pole disappears,
and the extension is polynomial.

The alternating cases are shorter.  In even degree, signs in the
stabilizers of the omitted vertices and edges force the extension
obstruction to vanish.  In odd degree we use a separate integral
argument modulo four: an alternating integer polynomial of odd degree
whose norm is divisible by four must have all coefficients even.
Scaling a nonzero rational solution to a primitive integer polynomial
would contradict this statement.  Restriction to $z=0$ and cyclic
summation then complete the classification.

Sections~\ref{sec:polynomial}--\ref{sec:ordinary} establish the polynomial
tools, restricted presentation, and support theorem.
Sections~\ref{sec:symmetric}--\ref{sec:odd-alt} treat the symmetry types,
and Sections~\ref{sec:assembly}--\ref{sec:period-corollaries} finish the
proof and its corollaries. Appendices~\ref{app:topology}--\ref{mf:app:kernel} prove the restricted
presentation, the finite-difference and residue--Hecke comparisons, the
analytic boundary comparison, and the reduced norm-kernel description.

The external mathematical inputs are the augmented-frame theorems of
Church--Putman \cite[Theorems~B and C]{ChurchPutman}, the well-rounded
retract \cite{AshRetract}, and Borel--Serre duality and compactification
\cite{BorelSerre}. The analytic inputs are the spectral-gap estimate and
Kuga's formula \cite[Lemma~4.1 and (3.6.1)]{BergeronVenkatesh}, the
nilmanifold comparison with unipotent coefficients
\cite[Theorem~1.1]{Kasuya}, and the Eichler--Shimura and Haberland
identities \cite[Theorems~2.1 and 3.3]{PasolPopa}. The residue--Hecke
comparison uses Hecke-equivariant Shapiro \cite[Theorem~3.1]{AshDoud}.
The applications, coefficient conventions, and additional comparisons
needed here are given in the corresponding sections and appendices.

\section{Polynomial preliminaries}\label{sec:polynomial}

We separate solutions by degree and symmetry, establish their vanishing
on several planes, and extract binary periods when an explicit
finite-difference obstruction vanishes. We also fix the cusp form
conventions and prove the $L$-value nonvanishing used later.

Unless stated otherwise, coefficients are rational. Section~\ref{sec:assembly}
extends scalars to characteristic zero fields. Put $V_n=\QQ[x,y,z]_n$ and
use $\QQ[X,Y]_n$ for binary forms of degree $n$. Abbreviate $W_n=W_n(\QQ)$
and $\mathscr P_n^0=\mathscr P_n^0(\QQ)$.

Write $\GL_d(R)$ for the invertible $d\times d$ matrices over a
commutative ring $R$, and $\SL_d(R)$ for its determinant one subgroup.
For a vector space $V$, write $\Sym^n(V)$ for its $n$th symmetric power
and $V^*$ for its linear dual over the ground field.

The substitutions in the defining equations give linear operators
on $V_n$:
\begin{equation}\label{eq:polynomial-operators}
\begin{aligned}
  (Sf)(x,y,z)&=f(y,x,z),
      &&\text{exchange of the first two variables},\\
  (Cf)(x,y,z)&=f(y,z,x),
      &&\text{cyclic permutation},\\
  (Af)(x,y,z)&=f(-y,x-y,z),
      &&\text{the order three substitution}.
\end{aligned}
\end{equation}
Products denote composition, with the rightmost operator applied
first. Direct substitution gives
\[
  (A^2f)(x,y,z)=f(y-x,-x,z),\qquad A^3=I,
\]
where $I$ is the identity operator. Thus the three-term norm
introduced above is $N=I+A+A^2$, and the defining equations become
\begin{equation}
  Cf=f\qquad\text{and}\qquad Nf=0.
\end{equation}

Let
\[
  U_n=\QQ[x,y]_n+\QQ[y,z]_n+\QQ[z,x]_n.
\]
This is the span of monomials missing a variable. The support assertion is
\[
  W_n(\QQ)=\{f\in V_n:Cf=f,\ Nf=0\}\subseteq U_n.
\]

On cyclic polynomials, the $\pm1$ eigenspaces of $S$ are the symmetric
and alternating subspaces. The next lemma shows that both projections
preserve solutions and settles degrees at most three.

\begin{lemma}\label{lem:sector-decomp}
Homogeneous components of a solution are solutions. Its symmetric and
alternating parts $(f+Sf)/2$ and $(f-Sf)/2$ are also solutions. There are no
nonzero solutions of degree at most three.
\end{lemma}
\begin{proof}
Each operator in \eqref{eq:polynomial-operators} preserves degree. Direct
substitution gives $SAS=A^2$ and $SCS=C^{-1}$. Hence $S$ preserves both
conditions $Nf=0$ and $Cf=f$. A three-cycle and a transposition generate the
full permutation group, which proves the two symmetry assertions. In degrees zero and one, the cyclic spaces are spanned by $1$ and
$x+y+z$. Their norms are $3$ and $3z$, respectively. In degree two, the
basis $x^2+y^2+z^2,xy+yz+zx$ has norms
$4(x^2-xy+y^2)+3z^2,\quad x^2-xy+y^2$.
These are linearly independent. The kernels in degrees zero, one, and two are therefore zero.

In degree three, the symmetric basis
$p_3=x^3+y^3+z^3$, $m_{21}=x^2y+x^2z+xy^2+xz^2+y^2z+yz^2$, and $xyz$ has norms
\[
 3z^3,\qquad -(x+y)(2x-y)(x-2y)+4z(x^2-xy+y^2),
 \qquad z(x^2-xy+y^2),
\]
which are linearly independent. The alternating basis
$\Delta=(x-y)(x-z)(y-z)$ has norm $3xy(x-y)\ne0$.
\end{proof}

Period extraction requires the following divisibility. Put
\[
  w=x+y+z,\qquad c_+=(x+y)(x+z)(y+z).
\]
The factors of $c_+$ define the \emph{opposite pair planes}
$x+y=0$, $x+z=0$, and $y+z=0$. The additional factor $w$ in even
degree will reduce period extraction to odd degree.
Write $N_{ij}$ for the norm in coordinates $i,j$, fixing the remaining
coordinate. Thus $N_{12}=N$, and symmetry gives $N_{13}f=N_{23}f=0$.

\begin{lemma}\label{lem:sym-boundary-factors}
A symmetric homogeneous solution $f$ is divisible by $c_+$. In positive even
degree it is divisible by $wc_+$.
\end{lemma}
\begin{proof}
Write $f\in V_n$, put $\nu=(-1)^n$, and let $g(s,t)=f(s,-s,t)$.
Expansion and symmetry give
\begin{align*}
 &(N_{23}f)(-s,-t,s)-\nu(N_{23}f)(-s-t,-t,s)\\
 &\quad \quad -(N_{23}f)(-s-t,-s,-s)+(N_{12}f)(-s-t,0,s)
   =\nu\{g(s,t)-g(s+t,t)\}.
\end{align*}
Therefore, $g$ is translation-invariant in $s$ over $\QQ(t)$, hence constant in $s$.
The norm at $(0,0,t)$ makes it zero. Symmetry gives the three factors of $c_+$.
For even degree, put $h(s,t)=f(s,t,-s-t)$. The identity
\[
 (N_{12}f)(-s-t,t,s)-(N_{23}f)(-s-2t,-t,s)
     =h(s,t)-h(s+t,t)
\]
makes $h$ independent of $s$. Since $h(s,t)=h(t,s)$, it is constant.
Positive homogeneity gives $h=0$. The distinct linear factors are
relatively prime, which proves the product divisibility.
\end{proof}

\subsection{Binary periods and discrete antiderivatives}

We use both binary period spaces in the classification, but only the
symmetric space in period extraction.

We write $\epsilon=1$ for the symmetric case and $\epsilon=-1$
for the alternating case. A ternary polynomial $f$ has
\emph{symmetry $\epsilon$} if it is cyclic and $Sf=\epsilon f$.
For binary polynomials, define the three-term norm by
\[
  N_2q(X,Y)
    =q(X,Y)+q(-Y,X-Y)+q(Y-X,-X).
\]
For $d\ge0$, let $\mathcal B_d^\epsilon$ be the space of
polynomials $q\in\QQ[X,Y]_d$ satisfying
\begin{equation}\label{eq:signed-binary}
\begin{gathered}
  q(Y,X)=\epsilon q(X,Y),\qquad
  q(-X,Y)=q(X,-Y)=-\epsilon q(X,Y),\\
  q(X,0)=q(0,Y)=0,\qquad N_2q=0.
\end{gathered}
\end{equation}
We put $\mathcal B_d^\epsilon=0$ for $d<0$ and abbreviate
$\mathcal B_d=\mathcal B_d^+$ and
$\mathcal B_d^0=\mathcal B_d^-$.
In the notation of the introduction, these are
\[
  \mathcal B_d=\mathscr P_d^{0,\mathrm{odd}}(\QQ),
  \qquad
  \mathcal B_d^0=\mathscr P_d^{0,\mathrm{ev}}(\QQ).
\]
The parity conditions imply $q(-X,-Y)=q(X,Y)$, so both spaces
vanish in odd total degree. For $\epsilon=1$, the axis conditions
follow from oddness in each variable. For $\epsilon=-1$, they
impose the normalization discussed in the introduction.

Using the symmetry and parity conditions in
\eqref{eq:signed-binary}, we can rewrite $N_2q=0$ as
\begin{equation}\label{eq:binary-addition}
  q(X,Y)-q(X-Y,Y)+\epsilon q(X,X-Y)=0.
\end{equation}
For $q\in\mathcal B_d^\epsilon$ with $d>0$, define its
\emph{polynomial lift} from the cyclic sum $\Phi(q)$ by
\begin{equation}\label{eq:polynomial-lift}
  L(q)(x,y,z)
    =\frac{\Phi(q)(x,y,z)}{x+y+z}
    =\frac{q(x,y)+q(y,z)+q(z,x)}{x+y+z}.
\end{equation}
To see that this quotient is a polynomial, substitute
$(X,Y)=(x+y,y)$ into \eqref{eq:binary-addition} and use
$q(y,x)=\epsilon q(x,y)$ to obtain
\[
  q(x,y)=q(x+y,y)+q(x,x+y).
\]
The parity conditions then give
$\Phi(q)(x,y,-x-y)=0$, so $w=x+y+z$ divides $\Phi(q)$.
Thus $L(q)$ is homogeneous of degree $d-1$ and has symmetry
$\epsilon$. We set $L(0)=0$ and write
$L(\mathcal B_d^\epsilon)$ for the image of this construction.
We use these lifts in period extraction, without assuming that they
satisfy the ternary norm equation.

We next introduce the operator that relates finite differences
to derivatives. Let
\begin{equation}\label{eq:Bernoulli}
  B(T)=\frac{T}{e^T-1}
      =\sum_{r\ge0}\frac{B_r}{r!}T^r,
  \qquad B_1=-\frac12,
\end{equation}
be the generating series for the Bernoulli numbers.
For a parameter $t$ independent of $X$, let
$(\tau_tg)(X)=g(X-t)$.
Taylor's formula gives $\tau_t=e^{-t\partial_X}$, and hence
\[
  B(-t\partial_X)(\tau_t-I)=-t\partial_X.
\]
Thus the Bernoulli operator converts a finite difference into a
derivative. The series acts finitely on polynomials and is invertible,
since $B(0)=1$.

We now define the polynomial obstruction used in the symmetric
argument. Let $X,Y,Z$ be independent auxiliary variables,
distinct from the source variables $x,y,z$.
For a symmetric polynomial $f$, set
\begin{equation}\label{eq:Jdef}
\begin{aligned}
  \psi_1&=f(Z-2Y,-X,X-Y),\\
  \psi_2&=\psi_1-f(Z-X,X-2Y,-Y)
                 +f(Z-2Y,Y-X,X-2Y),\\
  J_f&=ZB(-Y\partial_X)\psi_1
          -YB(-Z\partial_X)\psi_2.
\end{aligned}
\end{equation}
This is the \emph{Bernoulli obstruction}. Its substitutions come from
the graph construction in Section~\ref{sec:symmetric}, where
Proposition~\ref{thm:J-vanishes} proves $J_f=0$ for symmetric solutions.
Every source triple has coordinate sum $Z-3Y$, independent of $X$.
Since the Bernoulli operators differentiate only in $X$, every symmetric
polynomial $v$ satisfies
\begin{equation}\label{eq:J-w-mult}
  J_{w^rv}=(Z-3Y)^rJ_v
  \qquad (r\ge0).
\end{equation}
In particular, if $v=w^2u$, then $J_v=0$ if and only if $J_u=0$.

\subsection{Extracting symmetric binary periods}
\label{app:period-extraction}

The next lemma describes the odd degree symmetric kernel of $J$.
It does not assume the norm equation $Nv=0$.

\begin{lemma}\label{lem:period-extraction}
Let $m>0$ be odd and let $v\in V_m$ be symmetric. Then we have
\begin{equation}\label{eq:common-kernel}
 J_v=0\quad\Longleftrightarrow\quad
 v\in\bigoplus_{j\ge0}w^{2j}L(\mathcal B_{m+1-2j}).
\end{equation}
Every lift on the right is polynomial, and the expression is unique.
\end{lemma}
\begin{proof}
We first read the first order data on $w=0$, then reconstruct a binary
period, and finally divide the remainder by $w^2$.

\smallskip\noindent\emph{The first order data.}
First, direct substitution gives $J_v(X,0,Z)=Z v(Z,-X,X)$. Thus $J_v=0$ implies vanishing
on each opposite pair plane, and hence $c_+\mid v$.
Write $\partial_i$ for differentiation in the $i$th argument of $v$, and put
\[
 h(X,Y)=v(-Y,X,Y-X),\quad
 H(X,Y)=(\partial_1+\partial_2)v(-Y,X,Y-X),\quad
 p(X,Y)=-\int_0^X H(s,Y)\,ds.
\]
Define
\[
 C_v(X,Y)=p(X-Y,Y)-p(X,Y)-h(X,Y),\qquad
 \mathcal D=-Y\partial_X,\qquad \tau=e^{\mathcal D}.
\]
The operator $\tau$ translates $X$ to $X-Y$. We claim that
\begin{equation}\label{eq:normal-jet}
 J_v(X,Y,3Y)=0,\qquad
 \left.\partial_ZJ_v\right|_{Z=3Y}
     =\tau^2 B(3\mathcal D)B(\mathcal D)C_v.
\end{equation}
Here all power series are finite operators on the fixed polynomial degree.
To verify the claim, let $v_i=(\partial_i v)(-Y,X,Y-X)$ and set
$\psi_i^0=\psi_i|_{Z=3Y}$,
$\psi_i^1=(\partial_Z\psi_i)|_{Z=3Y}$. Odd homogeneity and symmetry give
\[
 \psi_1^0=-h,\quad \psi_2^0=-(1+\tau+\tau^2)h,\qquad
 \psi_1^1=v_1,\quad \psi_2^1=(1+\tau)v_1-\tau^2v_3.
\]
The chain rule also gives $h_X=v_2-v_3$, $h_Y=-v_1+v_3$, and $p_X=-(v_1+v_2)$.
Differentiate the defining expression for $J_v$, including its Bernoulli
operator. The required cancellations are exactly
\begin{align*}
 B(3t)(1+e^t+e^{2t})&=3B(t),\\
 -B(t)+tB'(3t)(1+e^t+e^{2t})+te^{2t}B(3t)
     &=-e^{2t}B(3t)B(t),\\
 B(\mathcal D)(\tau-1)&=\mathcal D.
\end{align*}
The first two identities follow by substituting $B(t)=t/(e^t-1)$ and
clearing denominators. The third is its defining finite-difference identity.
This proves \eqref{eq:normal-jet}. The operator on its right is invertible,
since each factor has constant term one. Consequently $J_v=0$ gives $C_v=0$.

\smallskip\noindent\emph{Reconstruction of the binary period.}
Symmetry and homogeneity give
\begin{equation}\label{eq:H-symmetry}
 h_X=H(X,Y)-H(Y-X,Y),\qquad H(-Y,-X)=H(X,Y).
\end{equation}
Differentiate $C_v=0$ in $X$ and use the first identity to obtain
$H(-X,Y)=H(X,Y)$. Since $H$ has even total degree, it is separately even
and symmetric. Thus $p$ is odd in $X$ and even in $Y$, and
$p_X(X,Y)=p_X(Y,-X)$. Define
\begin{equation}\label{eq:q-extract}
 q(X,Y)=-\frac{Xp(Y,-X)+Yp(X,Y)}{m+1}.
\end{equation}
Euler's identity for the degree $m$ polynomial $p$ gives $q_Y=-p$.
The displayed formula and the parities of $p$ show that $q$ is symmetric,
odd in each variable, and zero on both axes. Since $c_+\mid v$, we have
$h(0,Y)=0$. Evaluating $C_v=0$ at $X=0$ therefore gives $p(Y,Y)=0$,
and hence $q(Y,Y)=0$.

Set
\[
 D_q=q(X,Y)-q(X-Y,Y)+q(X,X-Y).
\]
Using $q_Y=-p$, symmetry of $q$, and oddness of $p$ in its first
variable, we obtain
\[
 (\partial_X+\partial_Y)D_q
 =-p(Y,X)-p(X,Y)+p(X-Y,Y)-p(X-Y,X)
 =h(X,Y)+h(Y,X)=0.
\]
Here $C_v=0$ gives the second equality. A homogeneous polynomial of degree $m+1$
annihilated by this derivative is a multiple of $(X-Y)^{m+1}$.
At $Y=0$ it equals $q(X,X)=0$. Thus $D_q$ vanishes, which proves
$q\in\mathcal B_{m+1}$.

\smallskip\noindent\emph{The lift and descent in degree.}
We check directly that every lift has zero obstruction. For
$q\in\mathcal B_d$, put $t=Z-2Y$ and
$r(X,Y,Z)=q(X,Y)-q(X,t)$. If $\tau_u$ translates $X$ to $X-u$, the two
frame polynomials of $\Phi(q)$ are
\[
 ((1-\tau_Y)r,\ (1-\tau_Z)r).
\]
This follows by expanding the three cyclic sums and using
\eqref{eq:binary-addition}. Since
$B(-u\partial_X)(1-\tau_u)=u\partial_X$, the two terms of
$J_{\Phi(q)}$ cancel. The identity $\Phi(q)=wL(q)$ and
\eqref{eq:J-w-mult} then give $J_{L(q)}=0$.
Cancellation on each opposite pair plane also gives $c_+\mid L(q)$.

The polynomial extracted from $L(q)$ is the original $q$. Indeed,
differentiate $\Phi(q)=wL(q)$ once in each of the first two variables
and restrict to $w=0$. The resulting first-plane data give
$p_{L(q)}=-q_Y$, which recovers \eqref{eq:q-extract}.
For a symmetric positive degree homogeneous polynomial, $p_v=0$ is
equivalent to $w^2\mid v$: its three pair sums of first derivatives
then vanish on $w=0$, so twice each individual derivative vanishes there.
Euler's identity gives $w\mid v$, and differentiation gives a second
factor of $w$. The converse is immediate.
It follows that $v-L(q)=w^2u$. The quotient $u$ is symmetric and satisfies
$J_u=0$ by \eqref{eq:J-w-mult}. Repeat in degree $m-2$. The final remainder
has degree below three and is divisible by $c_+$, so it is zero.
The extracted first-plane data at each step also prove uniqueness.
\end{proof}

The next form of the extraction supplies the two seeds for the symmetric
extension. Its only negative exponent occurs in odd degree and accounts
for the possible simple pole.
\begin{corollary}\label{cor:period-towers}
If a symmetric solution $f\in V_n$ satisfies $J_f=0$, then uniquely
\begin{equation}\label{eq:sym-tower}
 f=\sum_{j\ge0}w^{e_j}\Phi(q_j),\qquad
 \begin{cases}
 e_j=2j,\ q_j\in\mathcal B_{n-2j},& n\text{ even},\\
 e_j=2j-1,\ q_j\in\mathcal B_{n+1-2j},& n\text{ odd}.
 \end{cases}
\end{equation}
The sum is finite. In odd degree, its first term is interpreted as the
polynomial lift $L(q_0)$.
\end{corollary}
\begin{proof}
For $n=0$, the assertion follows from Lemma~\ref{lem:sector-decomp}.
Assume $n>0$. In even degree, Lemma~\ref{lem:sym-boundary-factors} gives $f=wv$ with
$v$ symmetric of odd degree. Equation~\eqref{eq:J-w-mult} gives $J_v=0$.
Apply Lemma~\ref{lem:period-extraction} and multiply its expansion by $w$.
In odd degree apply that lemma directly to $f$.
\end{proof}

\subsection{Classical cusp forms and outer critical values}
\label{sec:classical-conventions}

We fix the conventions and $L$-value nonvanishing needed for the
boundary comparison and pole removal.

Let $\GL_2^+(\RR)$ denote the group of real $2\times2$ matrices
with positive determinant. For a holomorphic function $h$ on
$\mathfrak H$, an integer $k$, and
$A=\left(\begin{smallmatrix}a&b\\c&d\end{smallmatrix}\right)
\in\GL_2^+(\RR)$, define
\[
  (h|_k A)(\tau)
    =(\det A)^{k/2}(c\tau+d)^{-k}
       h\left(\frac{a\tau+b}{c\tau+d}\right).
\]
For a finite index subgroup $\Lambda\subset\SL_2(\ZZ)$, a weight $k$
cusp form is holomorphic on $\mathfrak H$, satisfies $h|_k\gamma=h$
for $\gamma\in\Lambda$, and has a Fourier expansion with only positive
powers at every cusp. This gives exponential decay in each cusp
coordinate and the bound
\[
  (\Im\tau)^{k/2}|h(\tau)|=O(1)
  \qquad(\tau\in\mathfrak H).
\]

We now specialize to the full-level space
$S_k=S_k(\SL_2(\ZZ))$.
For a positive integer $m$, define
\[
\begin{aligned}
  \mathcal H_m
    &=\left\{
       \begin{pmatrix}a&b\\0&d\end{pmatrix}:
       a,b,d\in\ZZ,\ a,d>0,\ ad=m,\ 0\le b<d
       \right\},\\
  T_mh
    &=m^{k/2-1}\sum_{A\in\mathcal H_m}h|_k A.
\end{aligned}
\]
For a prime $p$, this becomes
\[
  (T_ph)(\tau)
    =p^{k-1}h(p\tau)
       +\frac1p\sum_{b=0}^{p-1}
          h\left(\frac{\tau+b}{p}\right).
\]
Suppose that $h$ is a simultaneous eigenform for these operators,
normalized by
\[
  h(\tau)=\sum_{m\ge1}a_h(m)e^{2\pi i m\tau},
  \qquad a_h(1)=1.
\]
Then $T_mh=a_h(m)h$, and the Hecke relations give
\[
  L(h,s)=\sum_{m\ge1}\frac{a_h(m)}{m^s}
        =\prod_p L_p(h,s),
  \qquad
  L_p(h,s)=\bigl(1-a_h(p)p^{-s}+p^{k-1-2s}\bigr)^{-1}
\]
in the common half-plane of absolute convergence.
Write $L^{(3)}(h,s)=\prod_{p\ne3}L_p(h,s)$, omitting the factor at $3$.
The completed function $\Lambda(h,s)=(2\pi)^{-s}\Gamma(s)L(h,s)$,
where $\Gamma$ is Euler's gamma function, has entire continuation by
its Mellin integral. These conventions agree with
\eqref{eq:analytic-period}
\cite[Sections~5.1--5.3 and 5.9--5.10]{DiamondShurman}.
The next lemma concerns the last two critical integers in $1,\ldots,k-1$.

\begin{lemma}\label{lem:outer-values}
If $h$ is a normalized level one cusp eigenform of weight $k>6$, then
$L(h,k-2)$ and $L(h,k-1)$ are nonzero. The same holds after omitting any
finite set of Euler factors.
\end{lemma}
\begin{proof}
The Fourier coefficient integral at height $y=1/m$ and the bound
$y^{k/2}|h(\tau)|\ll1$ give $a_h(m)=O(m^{k/2})$.
Thus the Dirichlet series converges absolutely for $\Re s>k/2+1$.
To check each local denominator as well, the Hecke recurrence gives
\[
 \sum_{r\ge0}a_h(p^r)t^r=(1-a_h(p)t+p^{k-1}t^2)^{-1}.
\]
The coefficient bound makes this holomorphic for $|t|<p^{-k/2}$.
Consequently, if the denominator is $(1-\alpha_pt)(1-\beta_pt)$, then
$|\alpha_p|,|\beta_p|\le p^{k/2}$. Each local factor is therefore finite
and nonzero for $\Re s>k/2$. In the smaller half-plane
$\Re s>k/2+1$, the sum over primes of the absolute deviations of the
local denominators from one converges. Their product is nonzero there.
Since $k-2>k/2+1$ when $k>6$, both asserted values lie in that half-plane.
Removing finitely many nonzero factors preserves nonvanishing.
\end{proof}

\section{Arithmetic classes and polynomial support}
\label{sec:ordinary}\label{sec:restricted}

Proposition~\ref{thm:good-apartment} identifies solutions with
polynomial valued functionals on a restricted incidence graph.
Proposition~\ref{thm:boundary-evaluation} proves two-variable support
for functionals on the full graph. The next two sections construct
the required extensions.

\subsection{Cohomology and the restricted incidence graph}

Unless otherwise specified, coefficient modules are rational vector
spaces with linear group actions. For a left $H$-module $M$, let $C^r(H,M)$ be the functions
$H^r\to M$ with the group coboundary, and let $H^r(H,M)$ be its
cohomology. We use right modules for group homology $H_r(H,M)$.

The invariant and coinvariant spaces are
\[
  M^H=\{m\in M:hm=m\text{ for every }h\in H\},
  \qquad
  M_H=M/\langle hm-m:h\in H,\ m\in M\rangle.
\]
In particular, $H^0(H,M)$ is the invariant space $M^H$.
In degree one, a cocycle is a \emph{crossed homomorphism},
a map $b:H\to M$ satisfying
\[
  b(gh)=b(g)+g\,b(h).
\]
Two such maps represent the same cohomology class if their
difference has the form $g\mapsto gm-m$ for some $m\in M$.
In degree two, a cocycle is a function $z:H\times H\to M$
satisfying
\[
  g\,z(h,k)-z(gh,k)+z(g,hk)-z(g,h)=0.
\]
A two-coboundary has the form
$g\,b(h)-b(gh)+b(g)$ for a function $b:H\to M$.
These are the formulas used for the connecting class in
Section~\ref{sec:symmetric}
\cite[Chapters~III--IV]{Brown}.

Write $\Hom_H(V,M)$ for linear maps satisfying $\varphi(hv)=h\varphi(v)$,
and $\Ext_H^1(V,M)$ for equivalence classes of extensions
$0\to M\to E\to V\to0$.
For spaces, $H_r$ and $H^r$ denote singular homology and cohomology
with the specified local system. A tilde denotes reduced homology.
On a simply connected cover the coefficients are constant, with
transport under deck transformations given by the group action.

For $1\to U\to H\to L\to1$, we use the Hochschild--Serre spectral sequence
\begin{equation}\label{eq:HS-orientation}
  E_2^{p,q}=H^p(L,H^q(U,M))
  \ \Longrightarrow\ H^{p+q}(H,M).
\end{equation}
It computes $U$-cohomology first, followed by $L$-cohomology
\cite[Chapter~VII]{Brown}.

Our maximal parabolic groups stabilize rational lines or planes and
are block upper triangular in adapted coordinates. The subgroup with
identity diagonal blocks is the \emph{unipotent radical}, and the diagonal
blocks give the \emph{Levi quotient}. For an arithmetic subgroup,
$U$ is a lattice in the radical and $L$ its Levi quotient. A radical
element is a \emph{unipotent shear}
\[
  u(B)=
  \begin{pmatrix}
    I_r&B\\
    0&I_s
  \end{pmatrix},
  \qquad B\in M_{r\times s}(\ZZ),
\]
where $I_r,I_s$ are identity matrices. It acts on the two coordinate
blocks by $(x,y)\mapsto(x+By,y)$.

The arithmetic Levi groups occurring here have free subgroups
of finite index, so their rational cohomology vanishes above
degree one \cite[Chapter~I]{SerreTrees}.
Finite quotients are handled by averaging over $\QQ$.
Consequently, only the columns $p=0,1$ of
\eqref{eq:HS-orientation} can contribute in these applications.

We now specify the arithmetic group and its polynomial coefficients.
Let $\FF_3=\ZZ/3\ZZ$, write $e_i$ for the standard coordinate
vectors and $E_{ij}$ for the matrix units, and put
$U_{ij}=I+E_{ij}$ for $i\ne j$.
Set
\[
  G=\GL_3(\ZZ),\qquad
  \ell=(1,1,1)\in(\FF_3^3)^*,\qquad
  K=\{g\in G:\ell g=\ell\pmod3\}.
\]
Thus $K$ fixes the covector $\ell$ modulo three, not merely its line.

For $t=0,1$, our coefficient modules are
\[
  E_{n,t}
    =\Sym^n((\QQ^3)^*)\otimes\det^{-t}.
\]
Their underlying vector space is $V_n=\QQ[x,y,z]_n$.
The determinant twist changes the action, not the polynomials:
\[
  g\cdot f(\mathbf x)
    =\det(g)^{-t}f(g^{-1}\mathbf x),
  \qquad
  \mathbf x=(x,y,z)^{\mathsf t}.
\]
The choices $t=0$ and $t=1$ will correspond to the alternating
and symmetric cases, respectively.

We distinguish this left action from direct substitution and
from the right action used in homology:
\[
  (f\mid g)(\mathbf x)=f(g\mathbf x),\qquad
  f\mid_t g=g^{-1}\cdot f=\det(g)^t(f\mid g).
\]
The unindexed bar $\mid$ always denotes direct substitution,
without a determinant twist. We use the left action in cohomology
and equivariant $\Hom$, and the right action $\mid_t$ in homology.
The left action also induces a right action on the dual module:
\[
  (\Psi\cdot g)(v)=\Psi(g\cdot v).
\]
For a rational line or plane $W\subset\QQ^3$, consider its integral
lattice $W\cap\ZZ^3$. We call $W$ \emph{good} if this lattice
contains a vector whose coordinate sum is nonzero modulo three,
and \emph{bad} otherwise. Equivalently, $\ell$ is nonzero on the
image of $W\cap\ZZ^3$ in $\FF_3^3$.
For example, $\QQ e_1$ is good, whereas $\QQ(e_1-e_2)$ is bad.
The condition is imposed separately on each subspace:
the good plane $\QQ e_1+\QQ e_2$ contains the bad line
$\QQ(e_1-e_2)$.

Let $\mathcal T$ be the graph whose vertices are the rational
lines and planes in $\QQ^3$, with an edge joining a line to a
plane precisely when the line lies in the plane. This is the
rational Tits building for $\GL_3$, a spherical building of
type $A_2$ \cite{AbramenkoBrown}.
Let $\mathcal T^\uparrow$ be the subgraph consisting of the good
vertices and all edges between them. Define
\[
  \St=\widetilde H_1(\mathcal T;\QQ),
  \qquad
  \St^\uparrow=\widetilde H_1(\mathcal T^\uparrow;\QQ).
\]
The first is the \emph{Steinberg module}. For a graph, first homology
consists of finite rational edge chains with zero boundary. Thus
$\St^\uparrow\subset\St$, and both are $K$-modules.

An ordered basis $(v_1,v_2,v_3)$ of $\QQ^3$ determines three lines
$L_i=\QQ v_i$ and three planes $P_{ij}=\QQ v_i+\QQ v_j$.
The cycle
\[
  L_1\longrightarrow P_{12}\longrightarrow L_2
  \longrightarrow P_{23}\longrightarrow L_3
  \longrightarrow P_{31}\longrightarrow L_1
\]
is the \emph{apartment} determined by this ordered basis.
Permuting the basis vectors changes its orientation by the sign
of the permutation. The apartment of the standard basis
$(e_1,e_2,e_3)$ is called the \emph{identity apartment}. All its
vertices are good.

The next proposition identifies solutions by evaluation on the identity
apartment. Its presentation uses the matrices
\[
  A=
  \begin{pmatrix}
    0&-1&0\\
    1&-1&0\\
    0&0&1
  \end{pmatrix},
  \qquad
  C=
  \begin{pmatrix}
    0&1&0\\
    0&0&1\\
    1&0&0
  \end{pmatrix},
\]
and let $S_0$ interchange the first two coordinate vectors.
Here $A$ and $C$ are the matrix versions of the substitutions
in \eqref{eq:polynomial-operators}, with the coefficient actions
fixed above. All three matrices lie in $K$.

The group algebra $\QQ K$ records rational linear combinations of
translates of the identity apartment. The relations $1+A+A^2$, $C-1$,
and $S_0+1$ encode the norm, cyclic invariance, and the transposition
sign. The determinant twist converts the last sign into the required
polynomial symmetry.

The analogous presentation for the full incidence graph is due
to Bykovskii. Church--Putman give a topological proof using
augmented frames \cite[Theorem~B]{ChurchPutman}.
The congruence-restricted presentation requires the additional
argument in Appendix~\ref{app:topology}.

\begin{proposition}\label{thm:good-apartment}
The graph $\mathcal T^\uparrow$ is connected, and we have
\begin{equation}\label{eq:good-apartment-presentation}
 \St^\uparrow\cong
 \frac{\QQ K}{\QQ K(1+A+A^2)+\QQ K(C-1)+\QQ K(S_0+1)}.
\end{equation}
Evaluation on the identity apartment gives isomorphisms
\[
 \Hom_K(\St^\uparrow,E_{n,0})\cong W_n^{\mathrm{alt}}(\QQ),
 \qquad
 \Hom_K(\St^\uparrow,E_{n,1})\cong W_n^{\mathrm{sym}}(\QQ).
\]
\end{proposition}
\begin{proof}
A partial augmented frame is a subset of an integral basis, possibly
together with one additional primitive line represented by a signed sum of
two basis vectors. Let $Y$ be the full good subcomplex of the complex of
partial augmented frames of Church--Putman
\cite[Theorem C]{ChurchPutman}. It is connected
and has $H_1(Y;\QQ)=H_2(Y;\QQ)=0$. Appendix~\ref{app:topology}
proves the required restriction argument. If $Y^0$ consists of simplices spanning
proper subspaces, pass to the poset of its nonempty simplices, ordered
by inclusion. The span map takes this poset to the poset of good proper
rational subspaces. The inverse image of the lower interval at a line is
a point. At a plane it is the face poset of its good filled Farey complex.
These spaces are contractible by Appendix~\ref{app:topology}. Quillen's
Theorem A~\cite{Quillen}, applied to these posets, gives a homotopy equivalence
of their order complexes and hence a
$K$-equivariant homology identification $H_1(Y^0;\QQ)=\St^\uparrow$.
The relative complex of $(Y,Y^0)$ starts in degree two. Its generators are
good unimodular frames, and its relations are good augmented frames.

Every good primitive line has a unique primitive integral representative of $\ell$-value $1$
modulo three. Normalized good bases are therefore a single $K$-orbit.
In a good augmented core, reduction modulo three forces all signs in its
three-vector relation to agree. The normalized vectors sum to zero.
For $r=-e_1-e_2$, its boundary, modulo proper-span faces, is
$\partial[r,e_1,e_2,e_3]=[I]+[A]+[A^2]$.
Column permutations give $C-1$ and $S_0+1$. The relative homology sequence
now gives \eqref{eq:good-apartment-presentation}, and also connectedness
of $Y^0$. The determinant twist reverses the transposition condition on
the polynomial. Finally $A^{-1}=A^2$, so the left-action norm is the same
polynomial norm.
\end{proof}

\subsection{Ordinary cycles}

We call full-building classes \emph{ordinary cycles}. Their finite
coefficient model makes apartment evaluation explicit.

Let $\mathscr X$ be the space of positive definite ternary real quadratic
forms modulo positive scaling. Determinant one representatives give
\[
  \mathscr X\cong\SL_3(\RR)/\SO(3),\qquad
  g\SO(3)\longleftrightarrow Q=(gg^{\mathsf t})^{-1},
\]
where $\SO(3)$ is the group of real orthogonal matrices of
determinant one. The space $\mathscr X$ has real dimension five.
For $g\in G$, we use the left action
\[
  g\cdot Q=g^{-\mathsf t}Qg^{-1}.
\]
The corresponding right action is $Q\mapsto g^{\mathsf t}Qg$.
The analytic appendix uses the Gram matrix $gg^{\mathsf t}=Q^{-1}$ rather
than $Q$, so the distinction between these actions matters.

Since $K$ stabilizes $\ell$ modulo three, its coset space is finite:
reduction gives the bijection
\[
  K\backslash G\xrightarrow{\sim}
  \mathcal S=\FF_3^3\setminus\{0\},\qquad
  Kg\longmapsto\ell g\pmod3.
\]
Thus the induced module consists of functions on the $26$ nonzero
row covectors:
\[
  F:\mathcal S\longrightarrow V_n.
\]
For $g\in G$, the right action is
\[
  (F\mid_t g)(r)
    =\det(g)^t\bigl(F(rg^{-1})\mid g\bigr).
\]
Covectors are multiplied modulo three, and the polynomial bar is untwisted.

To determine which lists represent homology classes, we use
Ash's well-rounded retract \cite{AshRetract}. This is a
contractible, three-dimensional cell complex inside $\mathscr X$.
Its points are the classes of forms whose shortest nonzero
integral vectors span $\RR^3$. For this construction we choose
minimum-one representatives, rather than determinant one
representatives:
\[
  \min_{0\ne v\in\ZZ^3}v^{\mathsf t}Qv=1.
\]
Its equivariant cellular complex computes rational group homology
\cite[Chapter~VII, Section~7]{Brown}. Since there are no cells above
degree three, top cycles represent $H_3(K,E_{n,t})$. The retract is
three-dimensional, whereas the arithmetic quotient is five-dimensional.
Poincar\'e--Lefschetz duality on its Borel--Serre compactification
relates degree three homology to degree two cohomology.

The standard top-dimensional cell is stabilized by the signed
permutation matrices. For such a matrix $h=D\pi$, where $D$
is diagonal with entries $\pm1$ and $\pi$ is a permutation matrix,
define
\[
  \chi_0(h)=\sgnchar(\pi),\qquad
  \chi_1(h)=\prod_iD_{ii}.
\]
The full apartment presentation \cite[Theorem~B]{ChurchPutman} gives the
following dual model, agreeing with Allison--Ash--Conrad
\cite{AllisonAshConrad}. Identity apartment evaluation is $F(\ell)$.

\begin{lemma}\label{prop:26-model}
The group $H_3(K,E_{n,t})$, with the right coefficient action defined above,
is represented by precisely the functions $F:\mathcal S\to V_n$ such that
\begin{align}
 &F(rh)=\chi_t(h)F(r)\mid h,\label{eq:fiber-signed}\\
 &F(r)+F(rA^2)\mid A+F(rA)\mid A^2=0,\label{eq:fiber-norm}
\end{align}
for every $r\in\mathcal S$ and every signed permutation matrix $h$.
Every bar in \eqref{eq:fiber-signed}--\eqref{eq:fiber-norm} is untwisted.
The determinant factor has already been included in $\chi_t$.
The coordinate $F(\ell)$ is the value on the identity apartment of the
corresponding ordinary Steinberg functional.
\end{lemma}
\begin{proof}
Let $E=E_{n,t}$. Rational Borel--Serre duality gives
\[
 H^3(K,E^*)\cong(\St\otimes E^*)_K,
 \qquad H_3(K,E)\cong\Hom_K(\St,E),
\]
where the second isomorphism is obtained by dualizing the first,
with the coefficient conventions fixed above. The virtual cohomological
dimension is three, and the orientation character is trivial
\cite{BorelSerre}. Induction over the finite coset space
$K\backslash G$ identifies the latter space with $G$-equivariant
maps from $\St$ to the induced coefficient module
\cite[Chapter~III, Section~5]{Brown}.

The full apartment presentation \cite[Theorem~B]{ChurchPutman}
has signed-permutation relations and the relation $1+A+A^2$.
The first says that $F\mid_t h=\sgnchar(\pi)F$, which is
\eqref{eq:fiber-signed}. The second says that
$F+F\mid_t A+F\mid_t A^2=0$, which is \eqref{eq:fiber-norm}
after evaluating at $r$. These relations are sufficient by the
presentation theorem. Under induction, the identity coset is indexed
by $\ell$, so evaluation there recovers the original apartment value.
Lemma~\ref{lem:apartment-boundary-pairing} fixes the geometric
normalization of this evaluation.
\end{proof}

The cellular description also applies to arbitrary rational
coefficient modules: averaging over finite cell stabilizers is
exact, and the retract has no cells above degree three.
In particular, for the full group $G$ with determinant-twisted
polynomial or rational function coefficients, a top-cycle
coefficient is symmetric, odd in each variable, and annihilated
by $N$. We use this generality in Section~\ref{sec:hecke-pole},
where the coefficient module is enlarged to allow simple poles.

\subsection{Apartment evaluation and boundary duality}

For an ordinary class $z$, the polynomial of interest is $F_z(\ell)$.
We express this evaluation as a boundary pairing.

The Borel--Serre bordification $\overline{\mathscr X}$ adds faces indexed
by rational parabolic subgroups \cite{BorelSerre}. Maximal faces correspond
to lines and planes, and their intersections to flags $L\subset P$.
Its boundary is homotopy equivalent to the rational Tits building, so
apartments have geometric representatives there.

To handle torsion, pass to a neat normal congruence subgroup: for each
of its elements, the subgroup of $\CC^\times$ generated by the eigenvalues
contains no nontrivial root of unity. Such a subgroup exists, and its
quotient of $\overline{\mathscr X}$ is a compact manifold with corners.
After the geometric constructions, take coinvariants in homology and
invariants in cohomology under the finite covering group. Rational
averaging is exact. We write $\overline{\mathscr X}_G$ and
$\partial\overline{\mathscr X}_G$ for these equivariant quotients,
without assuming that the quotient by $G$ itself is a manifold.

Let
\[
  M=\Ind_K^G E_{n,t}
\]
with the preceding coefficient actions, and let $\mathcal M$ be its local
system. Shapiro's lemma gives
\[
  H_3(\overline{\mathscr X}_G;\mathcal M)
    \cong H_3(K,E_{n,t}).
\]
The top-cell coefficient $F_z\in M$ of Lemma~\ref{prop:26-model}
is determined by $z$, since the retract has no degree four cells.
Its value $F_z(\ell)$ records the identity apartment.

To express evaluation as a scalar pairing, choose
$\varphi\in M^*$. The determinant one diagonal forms
\[
  \operatorname{diag}(a,b,(ab)^{-1}),
  \qquad a,b>0,
\]
form a two-dimensional flat in $\mathscr X$.
Its compactified image defines the \emph{coordinate modular
symbol} \cite{AshRudolph}. With constant coefficient $\varphi$
on the universal flat, it gives a relative homology class
\[
  a_\varphi\in
  H_2(\overline{\mathscr X}_G,
      \partial\overline{\mathscr X}_G;\mathcal M^*).
\]
Under the homotopy equivalence between the universal boundary
and the rational Tits building, the boundary of the flat
represents the coordinate apartment. We orient the flat and
the averaged standard top-cell generator so that their local
intersection number is $+1$. With this normalization, the next
lemma compares evaluation of a boundary class with its boundary
intersection pairing. We use Poincar\'e--Lefschetz duality with
the indicated local coefficients
\cite[Sections~3.3 and 3.H]{Hatcher}.

\begin{lemma}\label{lem:apartment-boundary-pairing}
Let $j$ be the inclusion of the boundary, and let
$\operatorname{PD}_{\partial}:H_3(\partial\overline{\mathscr X}_G;\mathcal M)
\to H^1(\partial\overline{\mathscr X}_G;\mathcal M)$
be boundary Poincar\'e duality. The diagram
\begin{equation}\label{eq:apartment-boundary-diagram}
\begin{CD}
 H_3(\partial\overline{\mathscr X}_G;\mathcal M)
       @>{j_*}>> H_3(\overline{\mathscr X}_G;\mathcal M)\\
 @V{\operatorname{PD}_{\partial}}VV
       @VV{z\mapsto\varphi(F_z)}V\\
 H^1(\partial\overline{\mathscr X}_G;\mathcal M)
       @>{\omega\mapsto\langle\omega,\partial a_\varphi\rangle}>>\QQ
\end{CD}
\end{equation}
commutes. In particular, for $\lambda\in V_n^*$ and
$\varphi_\lambda(F)=\lambda(F(\ell))$, we have
\begin{equation}\label{eq:apartment-boundary-evaluation}
 \lambda\bigl(F_{j_*c}(\ell)\bigr)
    =\bigl\langle\operatorname{PD}_{\partial}c,
                     \partial a_{\varphi_\lambda}\bigr\rangle.
\end{equation}
This is compatible with induction of coefficients and with
passage to a finite index subgroup.
\end{lemma}
\begin{proof}
Intersection with the flat reads the top-cell coefficient. Moving the
intersection into a boundary collar gives the bottom arrow.

For the first, normalize a form in the standard open top cell by requiring
that $e_1,e_2,e_3$ have length one. Near the identity matrix its Gram matrix
is
\[
 \begin{pmatrix}1&r_{12}&r_{13}\\r_{12}&1&r_{23}\\r_{13}&r_{23}&1\end{pmatrix}.
\]
The diagonal flat meets this cell at the identity, transversely. It meets no
other open top cell: a diagonal form is well rounded only when all three
diagonal entries agree, and at that form the shortest primitive vectors
are precisely $\pm e_i$. The off-diagonal coordinates orient the cell,
and the two diagonal ratios orient the flat. A permutation has its permutation
sign on the cell, while a coordinate sign change has sign $+1$, exactly as
in the top-cell model. The local intersection therefore evaluates the
coefficient by $\varphi$.

Translate this calculation by an integral basis matrix. The cellular
orientation and the dual coefficient transport are respectively the signed
permutation character and $\varphi\cdot g$. Thus the local intersections
assemble with the same relations as the modular symbol pairing
\cite[Section 3, (3--1)]{AllisonAshConrad}. With induced coefficients,
unfolding a chain over $K\backslash G$ labels the translated basis by its
covector $\ell g$. The functional $\varphi_\lambda$ selects the identity
coset, not the sum of all covectors. Hence the intersection of $z$ with
$a_\varphi$ is $\varphi(F_z)$.

For the second comparison, smooth the neat cover's corners and move $c$
slightly inward along a collar, transverse to the relative flat.
The intersections are those of $c$ with the flat's boundary. With the
boundary orientation induced by this collar, the cap-product definition
gives $\langle\operatorname{PD}_{\partial}c,\partial a_\varphi\rangle$.
The identity holds on chains modulo boundaries by naturality of
Poincar\'e--Lefschetz duality \cite{BorelSerre}.

Both constructions commute with coefficient transport and subdivision.
Coset unfolding is the cellular Shapiro map, proving compatibility with
induction. Taking finite group coinvariants in homology and invariants
in cohomology is exact and preserves the pairing, giving descent to the
stated level.
\end{proof}
\subsection{The support principle}

The pairing in Lemma~\ref{lem:apartment-boundary-pairing} now gives support
for both coefficient systems. The remaining task will be extension.

\begin{proposition}\label{thm:boundary-evaluation}\label{prop:top-is-boundary}
If $n>0$, every $K$-equivariant map $\St\to E_{n,t}$ takes the identity
apartment into $U_n$, for $t=0,1$.
\end{proposition}
\begin{proof}
\smallskip\noindent\emph{Passage to the boundary.}
Choose a neat normal subgroup $\Gamma'\subset K\cap\SL_3(\ZZ)$ and put
$X_{\Gamma'}=\Gamma'\backslash\mathscr X$.
Proposition~\ref{ar:prop:coercivity}, proved in
Appendix~\ref{app:arithmetic-direct}, gives
\[
 H_c^2(X_{\Gamma'},E_{n,t}\otimes\CC)
       \longrightarrow H^2(X_{\Gamma'},E_{n,t}\otimes\CC)
       \quad\text{equal to zero}.
\]
Extension to $\CC$ is faithful, so the rational map is zero as well.

The five-dimensional symmetric space has trivial orientation character:
on symmetric matrices $Q\mapsto g^{\mathsf t}Qg$ has determinant
$(\det g)^4>0$, and quotienting by positive scaling removes an oriented
radial direction. Poincar\'e--Lefschetz duality identifies the
absolute-to-relative map in degree three homology with the zero map
above. Hence boundary homology surjects onto ordinary top homology.
Put $\Delta=K/\Gamma'$. 
Rational Hochschild--Serre and averaging give
\[
 H_3(K,E_{n,t})\cong H_3(\Gamma',E_{n,t})_{\Delta}.
\]
The same finite-cover descent applies to boundary homology and commutes with
its map to ordinary homology, so exact averaging preserves the
surjection at level $K$. It remains to evaluate an arbitrary boundary
class.

\smallskip\noindent\emph{Face cocycles and radical invariance.}
For evaluation, use $M=\Ind_K^G E_{n,t}$ and
$\partial=\partial\overline{\mathscr X}_G$ as in
Lemma~\ref{lem:apartment-boundary-pairing}. Shapiro's finite covector factor
retains every level three cusp. The two maximal parabolic types for $G$
are represented by the upper integral parabolics $P$ and $Q$ of block types
$(2,1)$ and $(1,2)$, with intersection the Borel group $B_0$.

Take collar neighborhoods of the two unions of closed maximal faces on
the neat cover. Their intersections are Borel strata, with no triple
intersection because a flag has only a line and a plane. Induction and
Shapiro give the Mayer--Vietoris sequence
\begin{equation}\label{eq:MV-boundary}
  H^0(P,M)\oplus H^0(Q,M)\to H^0(B_0,M)\to H^1(\partial,M)
  \to H^1(P,M)\oplus H^1(Q,M)\to H^1(B_0,M).
\end{equation}
To keep the corner signs explicit, use the mapping cone
\[
\begin{gathered}
 \mathcal C^r=C^r(P,M)\oplus C^r(Q,M)\oplus C^{r-1}(B_0,M),\\
 d(a,b,c)=(da,db,\Res a-\Res b-dc).
\end{gathered}
\]
The dual chain differential is
$\partial(x,y,z)=(\partial x+i_Pz,\partial y-i_Qz,-\partial z)$, where $i_P,i_Q$ are the two inclusions. The evaluation pairing is the sum
of the three component pairings. These differentials are adjoint. In
particular a degree-one cocycle is $(A_P,A_Q,u)$ with
$du=\Res A_P-\Res A_Q$.
For $R=P$ or $Q$, let $N_R\cong\ZZ^2$ be its unipotent radical and
$L_R$ its Levi quotient. The
five-term sequence shows
\begin{equation}\label{eq:inflated}
 H^1(R,M)=H^1(L_R,M^{N_R}).
\end{equation}
This means that face cocycles may be chosen to factor through the Levi,
with values fixed by the radical. We verify the necessary vanishing.
Decompose the covectors into $N_R$-orbits and apply Shapiro to their
stabilizer lattices. Their unipotent coefficient systems satisfy the
nilmanifold comparison \cite[Theorem~1.1]{Kasuya}. Lie and group cohomology
therefore agree, naturally for the Levi action, as explained in
Appendix~\ref{ar:sec:constantterms}. We may compute after complexification,
which detects the required rational vanishing.

We now compute the Lie cohomology explicitly. For the root operators $z\partial_x,z\partial_y$, the cocycle equation
for a pair $(a,b)$ is $a_y=b_x$. Such a pair is the gradient of a homogeneous
potential of degree $n+1$. Exact pairs are the gradients of potentials
divisible by $z$, so the quotient is the binary symmetric power of degree
$n+1$. For $x\partial_z,y\partial_z$, subtract an exact pair to remove
its common $z$-dependent part. The remaining binary pairs are taken modulo
$(xg,yg)$, and $(a,b)\mapsto xb-ya$ identifies that quotient with binary
forms of degree $n+1$. These descriptions cover the standard and dual
coefficient conventions.

It remains to take Levi invariants. In either case $H^1(N_R,E_{n,t})$ restricts to $\Sym^{n+1}(\QQ^2)$ or its dual on the
derived subgroup $\mathrm{SL}_2$ of the Levi. The corresponding
level three subgroup fixes every covector. Its two opposite root
unipotents have no common invariant in this positive binary degree.
Consequently $H^1(N_R,M)^{L_R}$ vanishes, which proves \eqref{eq:inflated}.

If $v\in M^{N_R}$, the lattice fixing $\ell$ contains $3N_R$ and is
Zariski dense in the algebraic radical. Therefore $v(\ell)$ is invariant
under its root derivatives and belongs to $U_n$. The same conclusion holds
for invariants of the common root subgroup of two adjacent maximal radicals.

\smallskip\noindent\emph{Evaluation on the coordinate apartment.}
Let $c$ be a boundary top class and $z_c=F_{j_*c}$ its induced top-cell
coefficient. Let $U_n^\perp=\{\lambda\in V_n^*: \lambda|_{U_n}=0\}$.
To prove $z_c(\ell)\in U_n$, take $\lambda\in U_n^\perp$ and
put $\Phi_\lambda(v)=\lambda(v(\ell))$.
Lemma~\ref{lem:apartment-boundary-pairing} gives
\[
 \lambda(z_c(\ell))=\langle\omega,\partial a_{\Phi_\lambda}\rangle,
 \qquad \omega=\operatorname{PD}_{\partial}c\in H^1(\partial,M).
\]
The boundary apartment has successive vertices
$L_1,\ P_{12},\ L_2,\ P_{23},\ L_3,\ P_{31},\ L_1$,
where $L_i=\QQ e_i$ and $P_{ij}=\QQ e_i+\QQ e_j$.
We evaluate its face and corner terms in the mapping-cone convention above.

By \eqref{eq:inflated}, a mapping-cone representative of $\omega$ is
$(A_P,A_Q,u)$, where the face cocycles are inflated and
$du=\Res A_P-\Res A_Q$. Here are the coefficient transports on the six
edges. For an ordered permutation $\sigma$, let $g_\sigma e_i=e_{\sigma(i)}$.
Starting from the flag $L_1\subset P_{12}$, the permutations are
\[
 (1,2,3),\quad(2,1,3),\quad(2,3,1),\quad
 (3,2,1),\quad(3,1,2),\quad(1,3,2),
\]
with incidence signs $+,-,+,-,+,-$, respectively. The transported dual
coefficient is $\Phi_\lambda\cdot g_\sigma$. The left action inside this
formula includes the determinant twist. The listed signs are the separate
cellular orientation signs. Every $g_\sigma$ fixes $\ell$ and preserves
$U_n$. Each transported coefficient therefore annihilates the appropriate
radical invariant space, and all face pairings are zero. Further transport
within a maximal parabolic preserves that space.

For a corner, the common root group is a translate of
$N_0=\{I+aE_{13}:a\in\ZZ\}$. Its coefficient functional annihilates
$M^{N_0}$, and finite-dimensional duality gives
$(M^{N_0})^\perp =\operatorname{span}\{\Psi\cdot n-\Psi:n\in N_0,\ \Psi\in M^*\}$.
Thus the corner zero-chain can be written
$r=\sum_i(\Psi_i\cdot n_i-\Psi_i)$, with $n_i\in N_0$. In the right-module
bar complex, $\partial(\Psi[n])=\Psi\cdot n-\Psi$, so
$s=\sum_i\Psi_i[n_i]$ has $\partial s=r$. Both face cocycles are inflated
from their Levis and vanish on this common radical. Adjointness gives
\[
 \langle u,r\rangle
 =\sum_i\Psi_i(n_i\cdot u-u)
 =\langle du,s\rangle
 =\langle\Res A_P-\Res A_Q,s\rangle=0.
\]
Hence, the corner pairing vanishes by the coboundary equation for $u$
(no invariance of $u$ is required).
We obtain $\lambda(z_c(\ell))=0$ for every annihilator $\lambda$, so
$z_c(\ell)\in U_n$. Lemma~\ref{prop:26-model} completes the proof.
\end{proof}

\section{Symmetric solutions}
\label{sec:symmetric}\label{sec:eisenstein}

Let $f\in V_n$ be a symmetric solution with $n>0$. Constants are
settled by Lemma~\ref{lem:sector-decomp}. We prove $J_f=0$, use
Corollary~\ref{cor:period-towers} to construct an extension, and remove
its possible odd degree pole by Hecke operators. The support theorem
then applies.

\subsection{An obstruction with boundary restrictions}

We realize $J_f$ as a boundary coordinate of a degree two class
attached to the restricted-graph functional.

Conjugate the level covector $\ell=(1,1,1)$ to $e_3^*=(0,0,1)$ by
\[
  R_3=
  \begin{pmatrix}
    1&0&0\\
    0&1&0\\
    -1&-1&1
  \end{pmatrix},
  \qquad
  B_3=R_3^{-1},\qquad
  K^\dagger=B_3KB_3^{-1},\qquad
  \mathcal T^\dagger=B_3\mathcal T^\uparrow.
\]
Since $e_3^*B_3=\ell$, the group $K^\dagger$ fixes $e_3^*$ modulo three.
A subspace is now good when its integral lattice contains a vector with
last coordinate nonzero modulo three. Write $E_n=E_{n,1}$. Transport
of the graph and coefficients gives a $K^\dagger$-equivariant functional
\[
  F:H_1(\mathcal T^\dagger;\QQ)\longrightarrow E_n.
\]
We fix its polynomial normalization. If $V$ is an ordered integral basis
with good column lines, let $\mathcal A(V)$ be its oriented apartment.
Choose column signs so that their last coordinates are one modulo three,
and call the resulting matrix $V_c$. Then the value of $F$ on this apartment is
\begin{equation}\label{eq:frame-evaluation}
  F(\mathcal A(V))
    =\det(V_c)f(V_c^{-1}\mathbf X),
  \qquad \mathbf X=(X,Y,Z)^{\mathsf t}.
\end{equation}
Indeed, $V_cR_3\in K^\dagger$, and the reference apartment
$\mathcal A(B_3)$ has value $f(R_3\mathbf X)$, so equivariance gives
this formula.

Choose a good vertex $v$ and oriented paths $\tau(g)$ from $v$ to $gv$
for $g\in K^\dagger$. Their endpoints make
\begin{equation}\label{eq:zF}
  z_F(g,h)
    =F\bigl(\tau(g)+g\tau(h)-\tau(gh)\bigr)
\end{equation}
well defined. The path terms cancel in the group coboundary, and changing
paths changes $z_F$ by a coboundary. Thus
$a_F=[z_F]\in H^2(K^\dagger,E_n)$ is the connecting class
\cite[Chapter~III]{Brown}, equivalently the pushout along $F$ of
\[
  0\longrightarrow H_1(\mathcal T^\dagger;\QQ)
   \longrightarrow C_1(\mathcal T^\dagger;\QQ)
   \longrightarrow C_0(\mathcal T^\dagger;\QQ)
   \longrightarrow\QQ\longrightarrow0,
\]
where $C_i$ is the rational cellular chain group.
If a subgroup fixes a good vertex $w_0$, choose a path $s$ from $v$ to
$w_0$ and take $\tau(g)=s-gs$. The cycle in \eqref{eq:zF} is then zero.
Hence $a_F$ vanishes on every good-vertex stabilizer.

Consider the bad line $\QQ e_1$. Write $P_L$ for its stabilizer in
$K^\dagger$, $U=\langle U_{12},U_{13}\rangle\cong\ZZ^2$ for its radical,
and $L=P_L/U$ for its Levi quotient. The generators act on $E_n$ by
$e^{-Y\partial_X}$ and $e^{-Z\partial_X}$. A one-cocycle $b$ with values
$b_1,b_2$ therefore satisfies
\[
  (e^{-Y\partial_X}-I)b_2
    =(e^{-Z\partial_X}-I)b_1.
\]
Define its Bernoulli coordinate by
\begin{equation}\label{eq:radical-coordinate}
  \alpha=B(-Y\partial_X)b_1,\qquad
  \beta=B(-Z\partial_X)b_2,\qquad
  \mathcal J(b)=Z\alpha-Y\beta.
\end{equation}
The finite-difference identity gives
$Y\partial_X\beta=Z\partial_X\alpha$, so $\partial_X\mathcal J(b)=0$.
This induces an isomorphism
\[
  H^1(U,E_n)\xrightarrow{\ \sim\ }\QQ[Y,Z]_{n+1}.
\]
Indeed, a coboundary has
$\alpha=-Y\partial_Xe$ and $\beta=-Z\partial_Xe$, so its coordinate is
zero. Conversely, $Z\alpha=Y\beta$ gives $\alpha=Yh$, $\beta=Zh$ because
$Y,Z$ are relatively prime. Choose $e\in E_n$ with $\partial_Xe=-h$.
Invertibility of the Bernoulli operators makes $b$ its coboundary.
For surjectivity, given $J\in\QQ[Y,Z]_d$, $d=n+1$, take
\[
  b_1=\alpha=\frac{J_Z}{d},\qquad
  b_2=\beta=-\frac{J_Y}{d}.
\]
These are independent of $X$, hence satisfy the cocycle equation and are
fixed by the Bernoulli operators. Euler's identity gives $\mathcal J(b)=J$.

For $1\to U\to P_L\to L\to1$, the \emph{mixed} Hochschild--Serre term is
$H^1(L,H^1(U,E_n))$, with one degree from each group. The coordinate above
expresses its coefficients as binary polynomials. A binary cusp is an
orbit of a rational line. A class is \emph{cuspidal} when its restriction
to every cusp stabilizer is zero. The next lemma identifies the mixed
restriction of $a_F$ and shows that its vanishing forces the literal
polynomial $J_f$ to vanish.

\begin{lemma}\label{lem:bad-line-coordinate}
The restriction of $a_F$ to the bad-line stabilizer has only the mixed
Hochschild--Serre component. This component is cuspidal on the binary Levi.
If it is zero, then $J_f=0$ as a polynomial. In the coordinate
\eqref{eq:radical-coordinate}, its cocycle is the negative of the one
whose specified order three generator has value $J_f$.
\end{lemma}

\begin{proof}
\smallskip\noindent\emph{The radical coordinate.}
Choose the base vertex $v=G_0$ and the path $s:G_0\to\QQ e_3\to H_0$, where
$G_0=\langle e_1,e_3\rangle$ and $H_0=\langle e_2,e_3\rangle$.
The bad-line Levi fixes $H_0$ and its radical $U$ fixes $G_0$.
For Levi elements $l$, use $\tau(l)=s-ls$ and $\tau(ul)=u\tau(l)$.
In particular $\tau(u)$ is zero. A binary Levi matrix $l$ acts on the last two
coordinates through $\operatorname{diag}(1,l)$. Put
\[
 A_0=\begin{pmatrix}1&-1\\3&-2\end{pmatrix},\quad
 B_0=\begin{pmatrix}-1&0\\-3&1\end{pmatrix},\quad
 R_0=\begin{pmatrix}-1&0\\0&1\end{pmatrix}.
\]
When acting on $\QQ^3$, these binary matrices are embedded as
$\operatorname{diag}(1,g)$ in the bad-line Levi. Both $B_0,R_0$ fix $s$. The radical cocycle
$\phi_l(u)=F((u-1)\tau(l))$ therefore has zero values for these two Levi
elements. We make the path cancellation explicit. Put
\[
 a=e_3,\quad b=(1,1,2),\quad c=(0,1,2),\quad
 d=(1,0,1),\quad e=(2,1,2),\quad G'=\langle e_1,c\rangle.
\]
Vectors denote their rational lines. Since $A_0e_3=-c$
and $A_0H_0=H_0$, the path $\tau(A_0)=s-A_0s$ is
\[
 G_0\longrightarrow a\longrightarrow H_0\longrightarrow c
 \longrightarrow G'.
\]
For independent vectors $v,w$, write
$\mathfrak e(v,w)=[\QQ v,\langle v,w\rangle]+[\langle v,w\rangle,\QQ w]$.
The apartment convention is
$\mathcal A(v,w,t)=\mathfrak e(v,w)+\mathfrak e(w,t)+\mathfrak e(t,v)$.
Cancelling repeated edges in the two translated paths gives
\[
\begin{aligned}
 (U_{12}-1)\tau(A_0)&=\mathfrak e(a,b)+\mathfrak e(b,c)+\mathfrak e(c,a),\\
 (U_{13}-1)\tau(A_0)&=\mathfrak e(a,d)+\mathfrak e(d,e)
                  +\mathfrak e(e,c)+\mathfrak e(c,a).
\end{aligned}
\]
The lines $b,c,e$ lie in one plane, so
$\mathfrak e(e,b)+\mathfrak e(b,c)=\mathfrak e(e,c)$. Consequently the two translated paths reduce to apartments:
\begin{align*}
 (U_{12}-1)\tau(A_0)&=\mathcal A(e_3,(1,1,2),(0,1,2)),\\
 (U_{13}-1)\tau(A_0)&=\mathcal A(e_3,(1,1,2),(0,1,2))
 +\mathcal A(e_3,(1,0,1),(2,1,2))
       +\mathcal A(e_3,(2,1,2),(1,1,2)).
\end{align*}
The column-sign lists are $(1,-1,-1),(1,1,-1),(1,-1,-1)$.
By \eqref{eq:frame-evaluation}, the preceding apartment identities give
\[
 \phi_{A_0}(U_{12})=\psi_1,
 \qquad
 \phi_{A_0}(U_{13})=\psi_2,
\]
where $\psi_1,\psi_2$ are exactly the symmetric frame polynomials
in \eqref{eq:Jdef}. Applying the radical coordinate
\eqref{eq:radical-coordinate}, we obtain
\[
 \mathcal J(\phi_{A_0})
   =ZB(-Y\partial_X)\psi_1-YB(-Z\partial_X)\psi_2
   =J_f.
\]
Thus $J_f$ is the literal radical coordinate and is independent of $X$.
We write it as $J_f(Y,Z)$ for the rest of this proof.

The mixed bar edge of a two-cocycle is
$e(z)(l)(u)=z(l,l^{-1}ul)-z(u,l)$.
Our choices $\tau(u)=0$ and $\tau(ul)=u\tau(l)$ give
\[
\begin{aligned}
 z_F(l,l^{-1}ul)&=F(\tau(l)-u\tau(l))=-\phi_l(u) \qquad {\text {\rm and}}\qquad
 z_F(u,l)&=F(u\tau(l)-\tau(ul))=0.
\end{aligned}
\]
Thus the mixed edge is $-\phi_l(u)$ and $z_F|_{U\times U}=0$.
We retain this sign, though its reversal would not affect vanishing.
The Levi has rational cohomological dimension one, so no further
filtration term occurs in this kernel: the complete restriction lies in
$H^1(L,H^1(U,E_n))$.
Under $\operatorname{diag}(t,L)$ the coordinate
\eqref{eq:radical-coordinate} transforms by $J(v)\mapsto t^{-2}J(L^{-1}v)$.
Indeed inverse conjugation on the radical contributes $t^{-1}L$, the
coefficient action contributes $t^{-1}\det(L)^{-1}$, and the contraction
$(Z,-Y)$ contributes $\det L$. In particular the independent integral
line sign acts trivially.

\smallskip\noindent\emph{The binary quotient and its cusps.}

At each cusp we take translation coinvariants, then reflection invariants.

Write $C_m$ for the cyclic group of order $m$ and $\mathfrak S_3$ for
the permutation group on three letters. The remaining arithmetic group is
$\Gamma_2=\{g\in\mathrm{GL}_2(\ZZ):(0,1)g=(0,1)\pmod3\}$, and it has a presentation
\begin{equation}\label{eq:Gamma-presentation}
 \Gamma_2=\langle A_0,B_0\rangle*\langle R_0\rangle\cong\mathfrak S_3*C_2.
\end{equation}
To verify completeness, the determinant one subgroup maps injectively to
an index four subgroup of
$\mathrm{PSL}_2(\ZZ)=\SL_2(\ZZ)/\{\pm I_2\}\cong C_2*C_3$.
On its four cosets the two generators act with cycle types $(2,2)$ and
$(3,1)$. The Bass--Serre quotient has four edges and four vertices, one
$C_3$ vertex stabilizer, and all other stabilizers trivial. Schreier
reduction uses the standard matrices
$\mathsf S=\left(\begin{smallmatrix}0&-1\\1&0\end{smallmatrix}\right)$ and
$\mathsf T=\left(\begin{smallmatrix}1&1\\0&1\end{smallmatrix}\right)$.
The representatives $1,\mathsf S,\mathsf S\mathsf T,\mathsf S\mathsf T^2$ give
$\langle A_0,U_0\mid A_0^3=1\rangle$, where
$U_0=\left(\begin{smallmatrix}1&0\\3&1\end{smallmatrix}\right)$ and $A_0=U_0\mathsf T^{-1}$.
Matrix multiplication gives
$R_0U_0R_0=U_0^{-1}$ and
$R_0A_0R_0=U_0^{-1}A_0^{-1}U_0$.
Set $B_0=U_0R_0$ to obtain \eqref{eq:Gamma-presentation}, with no extra
mixed relation.

Transport to $(a,b)=(-Y,Z-2Y)$ and write
$\widehat J(a,b)=J_f(-a,b-2a)$. The actions become
$\mathcal A(a,b)=(-b,a-b)$, $\mathcal B(a,b)=(-a,b-a)$, and
$\mathcal R(a,b)=(-a,b-4a)$.
The positive cocycle $\mathcal J(\phi_l)$ has values $\widehat J,0,0$
on $A_0,B_0,R_0$. Their finite group relations give
$(1+\mathcal A+\mathcal A^2)\widehat J=0$ and
$\widehat J(a,b)=-\widehat J(b,a)$.
For $u=(2\widehat J+\mathcal A\widehat J)/3$, one has
$(1-\mathcal A)u=\widehat J$ and $\mathcal Bu=u$.
Thus in $V=\QQ[a,b]_{n+1}$ the actual mixed class has quotient coordinate
\begin{equation}\label{eq:jF-quotient}
 j_F=-[u]\in V/(V^{\langle\mathcal A,\mathcal B\rangle}+V^{\mathcal R}).
\end{equation}
The minus sign comes from the mixed bar edge above.

Let $\Gamma_1(3)=\Gamma_2\cap\SL_2(\ZZ)$. A binary cusp is an orbit of
a rational line. There are exactly two cusp orbits, represented by the coordinate lines,
with widths one and three. Indeed a primitive column $(p,q)^t$ with
$3\mid q$ can be signed and completed to a matrix in $\Gamma_1(3)$ as its
first column. If $3\nmid q$, it can be so completed as its second column,
adjusting a Bezout complement by a multiple of that column.
These two cases are invariant and exhaustive.
Write $d=n+1$. At the lower cusp the translation is $Z\mapsto Z-3Y$,
whose degree $d$ coinvariant is the $Z^d$ coefficient. Its reflection fixes
that coefficient and reverses the translation, so its full cusp $H^1$
vanishes. At the upper cusp the translation is $Y\mapsto Y-Z$, with
coinvariant the $Y^d$ coefficient. The reflection acts on $H^1$ by
$(-1)^{d+1}$. Thus this restriction is automatically zero for even $d$.

For odd $d$, Lemma~\ref{lem:sym-boundary-factors} gives $w\mid f$.
All three frames have sum $Z-3Y$, so $J_f(Y,3Y)=0$, or
$\widehat J(a,-a)=0$. For the upper translation
$\mathcal A^2\mathcal B\mathcal R$, its cocycle value is
$\widehat J+\mathcal A\widehat J$. On $Z=0$, that is $b=2a$, the cocycle
relations give
$(\widehat J+\mathcal A\widehat J)(a,2a)=-\widehat J(a,-a)=0$.
Both cusp restrictions therefore vanish in every parity.

\smallskip\noindent\emph{Detection of the polynomial.}
Finally suppose $j_F=0$. Write the positive cocycle as $(1-l)m$.
Its zero values at $B_0,R_0$ make $m$ invariant under their product $U_0$,
hence $m=cY^d$. For odd $d$ the reflection makes $c=0$.
For even $d$, its $A_0$ value is
$J_f(Y,Z)=c(Y^d-(-2Y+Z)^d)$.
But $J_f(0,Z)=Zf(Z,-X,X)$ vanishes by the opposite pair vanishing.
Thus $c=0$ here as well. This proves the detection assertion.
\end{proof}

\subsection{The symmetric obstruction vanishes}

We apply the boundary comparison of Proposition~\ref{ar:prop:graph},
proved in Appendix~\ref{app:arithmetic-direct}, to the connecting class
$a_F$.

\begin{proposition}\label{thm:J-vanishes}
Every symmetric homogeneous solution satisfies $J_f=0$.
\end{proposition}
\begin{proof}
The degrees at most three are zero by Lemma~\ref{lem:sector-decomp}.
For $n\ge4$, let $a_F$ be the connecting class, and complexify its
coefficient system. Vanishing after complexification implies rational
vanishing. The class restricts to zero at every good vertex. Its bad-line
restriction is mixed and cuspidal by Lemma~\ref{lem:bad-line-coordinate}.
Its restriction at each flag is therefore zero: a flag contains either a
good line, where the whole restriction vanishes, or a bad line, where it
is a binary cusp restriction of that mixed class.

Let $P_0$ be the stabilizer of the bad plane $\langle e_1,e_2\rangle$ in
$K^\dagger$. Its radical is $N_P=\langle U_{13},U_{23}\rangle$, and its
Levi quotient is $L_P\cong\GL_2(\ZZ)$. The same gradient calculation used
in the support proof gives
\[
 H^1(N_P,E_n)=\det^{-1}\Sym^{n+1}((\QQ^2)^*),\qquad
 H^2(N_P,E_n)=\det^{-2}\Sym^n((\QQ^2)^*).
\]
The two opposite binary unipotents have no common invariant in positive
degree, so the second module has no $L_P$-invariants. The Levi has rational
cohomological dimension one, leaving only the mixed restriction. At a
binary cusp the mixed edge map is injective for the same dimension reason.
Flag vanishing therefore makes the plane class cuspidal.

Let $\xi_P$ and $\xi_Q$ be the collections of plane and line restrictions.
Let $T$ be the map taking plane cusp data to the line constant term of the
convergent cusp form sum constructed in Appendix~\ref{app:arithmetic-direct}.
Proposition~\ref{ar:prop:graph} gives $\xi_Q=T\xi_P$. It remains to show
that $\xi_P=0$.

In even degree the bad-plane
binary module has odd degree $n+1$, so its central $-I_2$ kills its
cohomology. Since all good-plane restrictions vanish, $\xi_P=0$ and hence
$\xi_Q=0$.

In odd degree the bad-plane datum is a level one cusp form, as
explained in Appendix~\ref{ar:sec:finite}. On each classical Hecke
eigenform its image at the good-line cusp is the nonzero scalar
\eqref{ar:eq:globalentry}, in the seed normalizations of
\eqref{ar:eq:seed}. Thus the good-line map is injective on the bad-plane space, whereas
the good-line restriction of $a_F$ is zero.
Again both $\xi_P$ and $\xi_Q$ vanish.

Lemma~\ref{lem:bad-line-coordinate} now gives $J_f=0$ as a polynomial,
not merely modulo coboundaries.
\end{proof}

\subsection{Extending the binary data}\label{sec:binary-extension}

Return to $\ell=(1,1,1)$ and the determinant-twisted model of
Lemma~\ref{prop:26-model}. The three signed-permutation orbits on
$\mathcal S$ are distinguished by support size. Thus extension requires
two seeds besides $f$. Corollary~\ref{cor:period-towers} will supply them.
We first allow rational functions to isolate the possible odd degree pole.

\begin{lemma}\label{lem:three-seed}
A signed-permutation equivariant function $F$ as in
\eqref{eq:fiber-signed} is uniquely determined by its seeds
$P=F(1,0,0)$, $Q=F(1,1,0)$, and $f=F(1,1,1)$.
Such an $F$ satisfies \eqref{eq:fiber-norm} if and only if the seeds satisfy
\begin{gather*}
 f\text{ symmetric},\quad Q(a,b,c)=Q(b,a,c)=-Q(a,b,-c),\\
 P(a,b,c)=P(a,c,b)=-P(a,-b,c)=-P(a,b,-c),\\
 N_{12}f=N_{12}Q=N_{23}P=0,
\end{gather*}
and
\begin{align}
 Q(a,b,c+a)-Q(a,b,c)&=f(a+c,-c,b),\label{eq:mixed1}\\
 P(a,c+a,b)-P(a,c,b)&=Q(-c,a+c,b).\label{eq:mixed2}
\end{align}
Conversely, seeds with these properties extend uniquely by
\eqref{eq:fiber-signed} to a function satisfying both fiber equations.
The statement holds for polynomial or rational function coefficients.
\end{lemma}
\begin{proof}
The signed-permutation orbits have support sizes one, two, and three. Their
stabilizers give the seed symmetries. The norm equations at
\[
 (0,0,1),\quad(1,1,0),\quad(1,1,1),\quad(1,0,1),\quad(1,0,0)
\]
give the five displayed equations. To see that they exhaust the fibers,
write $R(r)$ for the left side of \eqref{eq:fiber-norm}, and put
$D_{12}=\operatorname{diag}(-1,-1,1)$ and
$D_3=\operatorname{diag}(1,1,-1)$. Direct reindexing gives
\[
 R(rA)=R(r)\mid A,\qquad
 R(rD)=\chi_1(D)R(r)\mid D\quad(D=D_{12},D_3).
\]
The five representative orbits under $\langle A,D_{12},D_3\rangle$ have
sizes $2,2,4,12,6$, respectively. They partition the $26$ covectors.
\end{proof}

For the finite expansion $f=\sum_j w^{e_j}\Phi(q_j)$ in
\eqref{eq:sym-tower}, define the candidate seeds
\begin{equation}\label{eq:PQ-towers}
\begin{aligned}
  Q(a,b,c)
    &=\sum_j(a+b)^{e_j}
        \bigl(q_j(a,c)+q_j(b,c)\bigr),\\
  P(a,b,c)
    &=\sum_j a^{e_j}q_j(b,c).
\end{aligned}
\end{equation}

In even degree both seeds are polynomial. In odd degree only $e_0=-1$
is negative. Oddness of $q_0$ in its first variable makes
$q_0(a,c)+q_0(b,c)$ divisible by $a+b$, so $Q$ is still polynomial.
Only $P$ can have a pole, with term $q_0(b,c)/a$. We remove it in
Section~\ref{sec:hecke-pole}.

The $q_j$ give the required seed symmetries, and
\eqref{eq:binary-addition} gives \eqref{eq:mixed1}--\eqref{eq:mixed2}
term by term. Also $N_2q_j=0$ gives $N_{23}P=0$. Since $N_{12}f=0$,
only $N_{12}Q=0$ remains.

\begin{lemma}\label{lem:first-lift-norm}
Let $f\in\QQ[x,y,z]$ be symmetric with $Nf=0$. If
$Q\in\QQ[a,b,c]$ is symmetric in $a,b$, odd in $c$, and satisfies
\eqref{eq:mixed1}, then $N_{12}Q=0$.
\end{lemma}
\begin{proof}
Let $\Delta_s$ translate $c$ by $s$ and subtract the original value. Put
\[
 R_1=f(a+c,-c,b),\quad R_2=f(c-b,-c,a-b),\quad
 R_3=f(b-a+c,-c,-a).
\]
The norm and symmetry give
\begin{equation}\label{eq:common-denominator}
 \Delta_{-b}\Delta_{b-a}R_1+
 \Delta_a\Delta_{b-a}R_2+\Delta_a\Delta_{-b}R_3=0.
\end{equation}
Appendix~\ref{app:finite-difference} gives the twelve-norm expansion of
this identity. By \eqref{eq:mixed1}, it becomes
\[
 \Delta_a\Delta_{-b}\Delta_{b-a}(N_{12}Q)=0.
\]
Over $\QQ(a,b)$, each difference equals its nonzero step times $\partial_c$
times an invertible finite power series in $\partial_c$. Hence
$\partial_c^3N_{12}Q=0$, and oddness leaves $N_{12}Q=c\kappa(a,b)$.

Evaluate at $c=a$. With
$g_1=f(a,0,b)$, $g_2=f(-b,0,a-b)$, $g_3=f(b-a,0,-a)$, and
$h=f(a-b,-a,a-b)$, the mixed equation and symmetries give the three values
$g_1,g_2-h,-g_3$. Since
$(Nf)(a,b,0)=g_1+g_2+g_3$ and $(Nf)(b-a,0,-a)=2g_3+h$, we get
\[
 a\kappa=g_1+g_2-h-g_3
        =(g_1+g_2+g_3)-(2g_3+h)=0.
\]
Thus, we have $\kappa=0$.
\end{proof}

In even degree, Lemmas~\ref{lem:three-seed}
and~\ref{lem:first-lift-norm} give a polynomial extension with
identity apartment value $f$, so Proposition~\ref{thm:boundary-evaluation}
gives $f\in U_n$. The odd degree case remains polynomial except for
the possible term $q_0(b,c)/a$.

\subsection{A pole detected by Hecke operators}
\label{sec:hecke-pole}

To eliminate $q_0(b,c)$ using Hecke operators, we enlarge the coefficient
module to allow simple poles along every rational hyperplane. This makes
it stable under rational changes of variables.

Define
\[
  \mathcal M_n
    =V_n+
      \sum_{0\ne L\in(\QQ^3)^*}\frac{V_{n+1}}{L}.
\]
The sums are finite, and each $p/L$ has homogeneous degree $n$.
This is a vector space enlargement, not a localization allowing products
of linear denominators.

For a rational hyperplane $H=\{L=0\}$, define
\[
  \mathcal R_{n,H}
    =(V_{n+1}/LV_n)\otimes_\QQ(\QQ L)^*,
  \qquad
  \mathcal R_n=\bigoplus_H\mathcal R_{n,H}.
\]
The quotient records the numerator on $H$, and the dual factor removes
the choice of its equation. With $L^\vee(L)=1$, set
\[
  \operatorname{res}_H(p/L)=[p]\otimes L^\vee.
\]
Replacing $(p,L)$ by $(cp,cL)$ multiplies the factors by $c,c^{-1}$,
so the residue is unchanged. The direct sum over rational hyperplanes
through the origin has finite support.

Taking residues gives a $\GL_3(\QQ)$-equivariant exact sequence
\begin{equation}\label{eq:residue-exact}
  0\longrightarrow V_n
   \longrightarrow\mathcal M_n
   \xrightarrow{\operatorname{res}}\mathcal R_n
   \longrightarrow0.
\end{equation}
At a generic point of a hyperplane, terms with other pole hyperplanes
are regular, so residues are independent. Choosing numerators lifts every
residue. After combining proportional denominators, zero residue means
the numerator is divisible by its denominator, proving exactness.

This surjectivity is on coefficients, not on homology. The next lemma
identifies top residue classes with periods and constrains those that lift.
Write $[ac^n]q$ for the coefficient of $ac^n$.

\begin{lemma}\label{lem:residue-period}
The map sending $q\in\mathcal B_{n+1}$ to the residue cycle represented by
\[
 F_q(a,b,c)=\frac{q(b,c)}a+\frac{q(a,c)}b+\frac{q(a,b)}c
\] identifies $\mathcal B_{n+1}$ with $H_3(G,\mathcal R_n\otimes\det)$.
Its inverse reads the literal residue at $a=0$, namely $q(b,c)$.
Let $\mathcal L_n$ be the image of $H_3(G,\mathcal M_n\otimes\det)$ in this
space. If $n>0$, then $[ac^n]q(a,c)=0$ for every $q\in\mathcal L_n$.
Moreover $\mathcal L_n$ is stable under the usual binary Hecke operators.
\end{lemma}
\begin{proof}
A residue cycle is symmetric and separately odd, with finite hyperplane
support. Choose a noncoordinate supported primitive normal of maximal
height $|v_1|+|v_2|+|v_3|$, and normalize it by signed permutation to
$v=(\alpha,\beta,\gamma)$ with $\alpha,\beta>0$, $\gamma\ge0$.
The normals $vA=(\beta,-\alpha-\beta,\gamma)$ and
$vA^2=(-\alpha-\beta,\alpha,\gamma)$ remain primitive and have larger
height, so neither is supported. The norm relation at the hyperplane
normal to $v$ would then contain only its nonzero residue, a contradiction.
Thus only coordinate poles occur. Symmetry gives the representative $F_q$
in the statement, with $q$ symmetric and separately odd. Writing $d=a-b$,
direct substitution gives
\begin{equation}\label{eq:N-Fq}
 N_{12}F_q=K_q+\frac{N_2q(a,b)}c,\qquad
 K_q=\frac{a^2q(a,c)-b^2q(b,c)-d^2q(d,c)}{abd}.
\end{equation}
The numerator vanishes at $a=0$, $b=0$, and $a=b$. The three linear
factors $a,b,a-b$ are pairwise relatively prime, so their product divides
the numerator and $K_q$ is polynomial.
Consequently $F_q$ is norm-zero modulo polynomial coefficients if and
only if $N_2q=0$. Together with its symmetry and separate oddness, this
proves the converse as well as the residue identification.

It remains to determine the liftable residue subspace $\mathcal L_n$.
If this residue lifts to a top cycle, write the lift as $F_q+P$.
The polynomial remainder $P$ is symmetric and separately odd, with
$NP=-K_q$. To constrain $q$, put $\lambda=[ac^n]q(a,c)$ and compare the
$c^n$ coefficients. The numerator of $K_q$ contributes
\[
 \lambda\{a^3-b^3-(a-b)^3\}=3\lambda ab(a-b),
\]
so $[c^n]K_q=3\lambda$. Every monomial of $P$ is divisible by $ab$, and
$P$ has total degree $n$. Its $c$-degree is therefore at most $n-2$. The norm
acts only on $a,b$, so the same bound holds for $NP$. Comparing the
$c^n$ coefficient gives $3\lambda=0$ and therefore $\lambda=0$.

A Hecke correspondence consists of restriction, rational matrix transport,
and transfer. Sequence~\eqref{eq:residue-exact} commutes with all three:
rational matrices preserve numerator degree and pole order while permuting
hyperplanes, and the other maps are natural in the coefficients. Thus
liftable residues are Hecke-stable.
Let $T_3(p,1)$ be the double-coset operator of $\operatorname{diag}(1,1,p)$
and $T_2(p)$ the usual binary prime operator.
In the same literal residue coordinate, these operators satisfy
\begin{equation}\label{eq:Hecke-formula}
 T_3(p,1)=pT_2(p)+p^2I.
\end{equation}
Lemma~\ref{lem:literal-residue-comparison} identifies the literal residue
with the binary period coordinate. Appendix~\ref{app:Hecke-coordinate}
then computes the transfer using Hecke-equivariant Shapiro
\cite[Theorem 3.1]{AshDoud}. Naturality
makes $\mathcal L_n$ stable under the left side of \eqref{eq:Hecke-formula}.
Subtracting $p^2I$ and dividing by $p$ gives binary stability.
Only the nonzero coefficient of $T_2(p)$ is needed for stability.
A fixed scalar in the period comparison does not change this relation.
\end{proof}

Put $d=n+1$. By Lemma~\ref{lem:residue-period},
$\mathcal L_n\subseteq\mathcal B_d$ is Hecke-stable and
$[XY^{d-1}]q(X,Y)=0$ for every $q\in\mathcal L_n$.
The coefficient condition alone does not imply vanishing. Hecke
stability is essential.

Suppose that $\mathcal L_n\otimes_\QQ\CC$ is nonzero.
The odd Eichler--Shimura correspondence identifies the Hecke
action on $\mathcal B_d\otimes_\QQ\CC$ with the usual action
on level one cusp forms of weight $k=d+2$
\cite[Theorem~2.1 and Section~5]{PasolPopa}.
Since the Hecke operators commute and the space is finite
dimensional, $\mathcal L_n\otimes_\QQ\CC$ contains a nonzero
simultaneous eigenperiod $q_h$, corresponding to a normalized
cusp eigenform $h$ of weight $k$.
Rescaling $q_h$ within this subspace, we may take it to be the
odd part in $X$ of the period polynomial
\eqref{eq:analytic-period}. Its indicated coefficient is then
\begin{equation}\label{eq:period-extreme-L}
  [XY^{d-1}]q_h
    =-d\int_0^{i\infty}h(z)z^{d-1}\,dz
    =-d\,i^d\frac{\Gamma(d)}{(2\pi)^d}L(h,d).
\end{equation}
A nonzero level one cusp form has weight at least twelve.
Since $d=k-2$, Lemma~\ref{lem:outer-values} gives
$L(h,d)\ne0$. The coefficient in
\eqref{eq:period-extreme-L} is therefore nonzero, contradicting
its vanishing on $\mathcal L_n\otimes_\QQ\CC$.
Thus, we have $\mathcal L_n=0$.

For the odd degree expansion, the three seeds define a rational
$26$-coordinate cycle $F$. Set
\[
  S=\sum_{r\in\mathcal S}F(r)\in\mathcal M_n.
\]
Summing the signed-permutation relations shows that $S$ is
symmetric and odd in each variable. Summing the fiber norm
equations and reindexing by $r\mapsto rA$ gives $NS=0$.
Thus $S$ represents a top cycle for the full group $G$ with
coefficients in $\mathcal M_n\otimes\det$.
By Lemma~\ref{lem:residue-period}, its residue corresponds
to an element of $\mathcal L_n=0$. In particular, its literal
residue along $a=0$ vanishes.

Only $F(1,0,0)$ and $F(-1,0,0)$ contribute to the residue at $a=0$:
$Q,f$ are polynomial, and the other translates of $P$ have poles only
at $b=0$ or $c=0$. Both contributing coordinates equal $P$, since
\[
  F(-1,0,0)=-P(-a,b,c)=P(a,b,c),
\]
because every exponent $e_j=2j-1$ in
\eqref{eq:PQ-towers} is odd. Consequently, we have
\[
  0=\operatorname{res}_{a=0}S=2q_0(b,c),
\]
and hence $q_0=0$. The possible pole in $P$ disappears, so
all three seeds are polynomial.
Proposition~\ref{thm:boundary-evaluation} applies to the
resulting ordinary extension and gives $f\in U_n$.

We have proved symmetric support in both parities. Since $n>0$, we have
\[
  0=(Nf)(0,0,z)=3f(0,0,z).
\]
Symmetry therefore eliminates all pure powers.
Every remaining monomial involves exactly two variables, so
\[
  f=\Phi(q),\qquad q(X,Y)=f(X,Y,0).
\]
The polynomial $q$ is symmetric and vanishes on both axes.

Restricting $Nf=0$ to $z=0$ gives $N_2q=0$.
Evaluating this binary relation at $(X,X)$ and using the axis
conditions gives $q(X,X)=0$.
We may now evaluate the full norm at $(x,0,z)$:
\[
  0=(Nf)(x,0,z)
    =2\{q(x,z)+q(-x,z)\}.
\]
Thus $q$ is odd in its first variable, and symmetry makes it
odd in the second as well.

If $n$ is odd, separate oddness forces $q=0$, and hence $f=0$.
If $n$ is even, these are precisely the conditions on $q$
in Theorem~\ref{thm:symmetric-classification}.
Conversely, let $q$ satisfy those conditions.
In the expansion of $N\Phi(q)$, the three terms involving
only $x,y$ sum to $N_2q(x,y)=0$, while the other six cancel
in pairs by symmetry and separate oddness.
Thus $\Phi(q)$ is a symmetric solution, and its restriction
to $z=0$ is $q$. This completes the symmetric classification.

\section{Alternating solutions}\label{sec:odd-alt}

In even degree we construct the extension and apply
Proposition~\ref{thm:boundary-evaluation}. In odd degree a separate
mod four argument proves vanishing, without the Bernoulli obstruction
or arithmetic cohomology.

\subsection{The even degree extension}

The omitted vertex and edge stabilizers control the extension obstruction.

\begin{proposition}\label{prop:alternating-even}
Every alternating even degree solution belongs to $U_n$.
\end{proposition}
\begin{proof}
\smallskip\noindent\emph{Vertex stabilizers.}
Assume $n>0$, since the constant case is zero. For the stabilizer calculation
use the group $K^\dagger$ and graph $\mathcal T^\dagger$ introduced in
Section~\ref{sec:symmetric}. The relative graph has two vertex
orbits, represented by the bad line $\QQ e_1$ and bad plane
$\QQ e_1+\QQ e_2$, and two edge orbits: bad line in bad plane and bad line
in a good plane. Integral parabolic reduction modulo three proves this list
exhaustive. The vertex stabilizers are
\[
 H_P=\left\{\begin{pmatrix}A&u\\0&1\end{pmatrix}\right\},\qquad
 H_L=\left\{\begin{pmatrix}\eta&u\\0&D\end{pmatrix}\right\},
\]
where $A\in\mathrm{GL}_2(\ZZ)$, $u\in\ZZ^2$, $\eta=\pm1$, and
$D\in\Gamma_2=\{D\in\mathrm{GL}_2(\ZZ):(0,1)D=(0,1)\pmod3\}$.
The edge stabilizers are
\[
 H_{BB}=\left\{\begin{pmatrix}\eta_1&a&b\\0&\eta_2&c\\0&0&1\end{pmatrix}\right\},\qquad
 H_{BG}=\left\{\begin{pmatrix}\eta_1&a&b\\0&\eta_2&0\\0&3k&1\end{pmatrix}\right\},
\]
with integral entries and independent signs $\eta_i$.

Take $E=\Sym^n((\QQ^3)^*)$ and write $\rho$ for its left
inverse substitution representation. The vertex radicals are
$N_P=\langle U_{13},U_{23}\rangle$ and
$N_L=\langle U_{12},U_{13}\rangle$. On dual polynomial coefficients, the generators act with
$\log\rho(U_{ij})=-x_j\partial_{x_i}$. For $N_P$,
the closed pairs are gradients of potentials, as in
Section~\ref{sec:ordinary}. Modulo exact pairs, restriction of the potential
to $z=0$ gives a binary form of degree $n+1$. Under a general diagonal
matrix this identification has an extra factor $d_3$, which is one on
$H_P$. For $N_L$, the contraction $(\alpha,\beta)\mapsto z\alpha-y\beta$ gives the binary form. Inverse conjugation on its two roots makes this
contraction transform by $\eta^{-1}\det D$ times the dual binary action.
In the table, $\Sym^r(x,y)^*$ abbreviates the degree $r$ binary
polynomials with inverse substitution action, and similarly for $y,z$.
These calculations give its first column.
\begin{table}[!htbp]
\centering
$\begin{array}{|c|c|c|}
\hline
&H^1(N,E)&H^2(N,E)\\\hline
 N_P&\Sym^{n+1}(x,y)^*&\det(A)^{-1}\Sym^n(x,y)^*\\
 N_L&\eta^{-1}\det(D)\Sym^{n+1}(y,z)^*&\eta^{-n-2}\det(D)\\
\hline
\end{array}$
\par\vspace{7pt}
\caption{Cohomology of the two vertex radicals with $E=\Sym^n((\QQ^3)^*)$ and Levi actions.}
\label{tab:radical-cohomology}
\end{table}
\FloatBarrier
The $H^2$ column follows from the top cohomology formula for a rank two
lattice,
$H^2(N,E)=E_N\otimes\bigwedge^2(N\otimes\QQ)^*$, using
$E_{N_P}=\QQ[x,y]_n$ and $E_{N_L}=\QQ x^n$.
These are coinvariants. The corresponding invariant spaces are
$E^{N_P}=\QQ z^n$ and $E^{N_L}=\QQ[y,z]_n$. These invariant spaces
enter the edge calculation below.
Both Levis are virtually free. Thus total degree two in Hochschild--Serre
has only $H^1(L,H^1(N,E))$ and $H^0(L,H^2(N,E))$.
At $H_P$, the central $-I_2$ acts by $-1$ on $H^1$ and positive binary
degree kills the $H^2$ invariants. At $H_L$, the independent line sign kills
$H^1$ and the reflection $\operatorname{diag}(-1,1)\in\Gamma_2$ kills $H^2$.
Hence, we have $H^2(H_P,E)=H^2(H_L,E)=0$.

\smallskip\noindent\emph{Edge stabilizers.}
For either edge radical, extend a rank two radical by its remaining cyclic
root. The exact sequence
$0\to H^1(\ZZ,E^N)\to H^1(N_{\rm edge},E)
\to H^1(N,E)^\ZZ\to0$ is equivariant for
$D_0=\{\operatorname{diag}(\eta_1,\eta_2,1)\}$.
A weight $(a_1,a_2,a_3)$ means the character
$\operatorname{diag}(d_1,d_2,d_3)\mapsto d_1^{a_1}d_2^{a_2}d_3^{a_3}$.
Table~\ref{tab:edge-weights} records the weights of the two
one-dimensional subquotients:
\begin{table}[!htbp]
\centering
$\begin{array}{|c|c|c|}
\hline
\text{edge type}&\text{first subquotient}&\text{second subquotient}\\
\hline
BB&(-1,1,-n)&(0,-n-1,1)\\
 BG&(0,1,-n-1)&(-1,-n,1)\\
\hline
\end{array}$
\par\vspace{7pt}
\caption{Weights of the two degree-one cohomology subquotients for each edge stabilizer. For even $n$, each row has a nontrivial sign character.}
\label{tab:edge-weights}
\end{table}
\FloatBarrier
For type $BB$, use $N_P$ and the remaining root $E_{12}$. The left term
of the five-term sequence comes from $E^{N_P}=\QQ z^n$ and the dual of
$E_{12}$, giving $(-1,1,-n)$. The invariant potential in the right term is
$y^{n+1}$, with the extra $d_3$ factor, giving $(0,-n-1,1)$.
For type $BG$, use $N_L$ and $E_{32}$. The coinvariant of
$E^{N_L}=\QQ[y,z]_n$ under $-y\partial_z$ is represented by $z^n$.
Including the root-dual factor gives $(0,1,-n-1)$. The right term is represented
by $y^{n+1}$ in $\eta^{-1}\det(D)\Sym^{n+1}(y,z)^*$, giving $(-1,-n,1)$.
For $H_{BG}$ the actual remaining lattice root is $3E_{32}$.
Multiplication of its logarithmic differential by the nonzero rational
scalar $3$ changes neither its kernel nor its cokernel.
Since $n$ is even, each weight has a nontrivial first or second sign.
Exact averaging gives $H^1(H_{BB},E)=H^1(H_{BG},E)=0$.

\smallskip\noindent\emph{Extension to the full graph.}
Put $\St^\dagger=H_1(\mathcal T^\dagger;\QQ)$ and
$Q_0=\St/\St^\dagger$. The relative chain sequence and connectedness
of the two graphs give
\[
 0\longrightarrow Q_0\longrightarrow
 \QQ[K^\dagger/H_{BB}]\oplus\QQ[K^\dagger/H_{BG}]
 \longrightarrow
 \QQ[K^\dagger/H_L]\oplus\QQ[K^\dagger/H_P]\longrightarrow0.
\]
Applying $\Hom_{K^\dagger}(-,E)$ and Shapiro gives the exact segment
\[
\begin{split}
 H^1(H_{BB},E)\oplus H^1(H_{BG},E)
 &\longrightarrow\Ext_{K^\dagger}^1(Q_0,E)
 \longrightarrow H^2(H_L,E)\oplus H^2(H_P,E).
\end{split}
\]
The four vanishings make $\Ext_{K^\dagger}^1(Q_0,E)=0$. The sequence
$0\to\St^\dagger\to\St\to Q_0\to0$ shows that
$\Hom_{K^\dagger}(\St,E)\to\Hom_{K^\dagger}(\St^\dagger,E)$ is surjective.
Returning to the original coordinates, we apply
Proposition~\ref{thm:boundary-evaluation} on the identity
apartment. This value is the original polynomial, so it belongs to $U_n$.
\end{proof}

\subsection{Odd degree and reduction modulo four}

Recall that $N=I+A+A^2$, with $A$ defined in
\eqref{eq:polynomial-operators}. The following integral statement
implies that every homogeneous alternating solution of odd degree
vanishes over a field of characteristic zero.

\begin{theorem}\label{mf:thm:main}
Let $f\in\Z[x,y,z]$ be alternating, with every monomial of odd total degree.
If $Nf\in4\Z[x,y,z]$, then $f\in2\Z[x,y,z]$.
Consequently, an alternating homogeneous solution of $Nf=0$ in odd degree
over any characteristic zero field is zero.
\end{theorem}

The integral assertion is stronger than the characteristic zero vanishing.
It is also sharp with respect to the modulus: if
$f=(x+y+z)^2(x-y)(x-z)(y-z)$, direct expansion gives
\[
 Nf=6xy(x-y)(x^2-xy+y^2-z^2).
\]
Thus divisibility of $Nf$ by two does not force $f$ to have even coefficients.

The proof uses an explicit reduced norm-kernel calculation.
We prove the kernel description in Appendix~\ref{mf:app:kernel}.
The finite identities are identities of polynomials, not experimental
checks in a range of degrees.

\subsection{The eight coefficient polynomials}

Alternation becomes symmetry over $k=\FF_2$, so we need the symmetric
kernel of the reduced norm
\[
  \bn g=g(x,y,z)+g(y,x+y,z)+g(x+y,x,z).
\]

Put
\[
  a=x+y+z,\qquad b=xy+xz+yz,\qquad c=xyz,\qquad \cp=ab+c.
\]
The symmetric ring is $k[a,b,c]$. Give its elementary symmetric
generators weights $1,2,3$, respectively, so weighted degree agrees
with total degree in $x,y,z$.

Define
\begin{align}\label{mf:eq:invariants}
  d_4&=a^4+b^2+ac,&
  d_6&=a^2b^2+a^3c+abc+c^2,&
  d_7&=ac(ab+c).
\end{align}
These are the Dickson invariants of degrees four, six, and seven
\cite{Dickson}. Their invariance under invertible linear
substitutions over $k$ follows from the product identity
\begin{equation}\label{mf:eq:product}
  \prod_{(v_1,v_2,v_3)\in k^3}
    (T+v_1x+v_2y+v_3z)
    =T^8+d_4T^4+d_6T^2+d_7T.
\end{equation}
An invertible substitution permutes the eight factors, proving invariance.
Multiplication by these invariants therefore preserves the symmetric
kernel of $\bn$.

\begin{lemma}\label{mf:lem:kernel}
The ring $R=k[d_4,d_6,d_7]$ is a polynomial ring. On symmetric polynomials,
$\ker\bn$ is a free $R$-module on $K_j=a\cp B_j$, $0\le j<8$, where
\begin{equation}\label{mf:eq:Bs}
\begin{split}
(B_0,\ldots,B_7)=\bigl(&1,a,a^2+b,ab,a^3+c,a\cp,\\
                     &a^3(a^2+b),a^6+a^3c+b^3+c^2\bigr).
\end{split}
\end{equation}
\end{lemma}

The proof and finite data are in Appendix~\ref{mf:app:kernel}. Put
\[
 p=d_4,\qquad q=d_6,\qquad \rho=d_7^2,\qquad
 R_0=k[p,q,\rho],\qquad R_2=k[p^2,q^2,\rho].
\]
The degrees of $K_j$ are
\[
 (\delta_0,\ldots,\delta_7)=(4,5,6,7,7,8,9,10).
\]
Let $\kappa_j=1$ when $\delta_j$ is even and $\kappa_j=0$ otherwise, and set
\begin{equation}\label{mf:eq:Lbasis}
 L_j=K_jd_7^{\kappa_j}.
\end{equation}
Since $d_4,d_6$ have even degree and $d_7$ has odd degree,
Lemma~\ref{mf:lem:kernel} gives the following unique expression for every
odd degree symmetric polynomial in the reduced norm kernel:
\begin{equation}\label{mf:eq:eight}
 v=\sum_{j=0}^7 r_j(p,q,\rho)L_j,\qquad r_j\in R_0.
\end{equation}
This expression is valid in all degrees. It is the reason a finite
calculation can prove Theorem~\ref{mf:thm:main} uniformly.

\subsection{An explicit mod four obstruction}

An \emph{integral lift} of $v\in k[x,y,z]$ is a polynomial
$\widetilde v\in\ZZ[x,y,z]$ reducing to $v$. If $\bn v=0$, its norm
is even, but need not be zero. Divide the norm by two, reduce modulo
two, and apply the following operator $\Pi$ to obtain an obstruction
independent of the lift within each symmetry type.

We write a linear substitution as a triple of linear forms:
$(\lambda_1,\lambda_2,\lambda_3)$ sends $(x,y,z)$ to
$(\lambda_1,\lambda_2,\lambda_3)$. Let $\mathscr T$ consist of
\begin{equation}\label{mf:eq:Tlist}
\begin{gathered}
 (y,z,x+y),\quad (z,y,x+z),\quad (x,z,x+y),\\
 (x,z,y+z),\quad (x+y,z,x),\quad (x+y,z,y),\\
 (z,x+y+z,x+z),\quad (x+y+z,z,y+z),\quad (x,y,x+z),\\
 (x,y,y+z),\quad (y,x+z,x),\quad (y,x+y+z,x+y),\\
 (x+z,y,z),\quad (x+y+z,y,y+z),\quad (x,y+z,y),\\
 (x,x+y+z,x+y),\quad (x,y+z,z),\quad (x,x+y+z,x+z).
\end{gathered}
\end{equation}
Each substitution is invertible over $k$. For $h\in k[x,y,z]$,
define
\begin{equation}\label{mf:eq:Pi}
 \Pi h(x,y,z)=
 \sum_{(\lambda_1,\lambda_2,\lambda_3)\in\mathscr T}
                   h(\lambda_1,\lambda_2,\lambda_3).
\end{equation}

\begin{lemma}\label{mf:lem:cancellation}
For every symmetric $g\in k[x,y,z]$, one has $\Pi\bn g=0$.
\end{lemma}
\begin{proof}
Expand the left side into $54$ terms. Identify triples of linear forms that differ only by a permutation,
which does not change their value on a symmetric polynomial.
There are exactly $27$ resulting triples,
each occurring twice. They cancel in characteristic two. The lists in
\eqref{mf:eq:Tlist} and the three substitutions defining $\bn$ completely
specify these pairings.
\end{proof}

For $v$ as in \eqref{mf:eq:eight}, choose an alternating integral lift
$\widetilde v^{-}$. Such a lift exists: in the displayed generators replace
the factor $a\cp$ by
\[
 a\Delta,\qquad \Delta=(x-y)(x-z)(y-z),
\]
and lift all remaining symmetric expressions integrally. Indeed, we have
$\overline\Delta=\cp$ in characteristic two. Since $\bn v=0$, the coefficients
of $N\widetilde v^{-}$ are even. Define
\begin{equation}\label{mf:eq:beta}
 \beta_-(v)=\Pi\left(\frac{N\widetilde v^{-}}2\bmod2\right).
\end{equation}
Define $\beta_+(v)$ in the same way with any symmetric integral lift.

\begin{lemma}\label{mf:lem:well-defined}
The maps $\beta_\pm$ are independent of the respective choices of lift and
are $R_2$-linear.
\end{lemma}
\begin{proof}
Two lifts with the same symmetry differ by $2h$ with that symmetry.
The reduction $\overline h$ is symmetric in either case, so the change in
\eqref{mf:eq:beta} is $\Pi\bn\overline h=0$ by
Lemma~\ref{mf:lem:cancellation}.

If an integral polynomial $D$ is invariant modulo two, its square is
invariant modulo four: $(D+2E)^2\equiv D^2\pmod4$.
Consequently, the integral squares of the three polynomials in
\eqref{mf:eq:invariants} are invariant modulo four under the substitutions
in $N$. Their reductions are also invariant under every substitution in
\eqref{mf:eq:Tlist}. Multiplying a lift by any polynomial in those squares
therefore commutes with division of its norm by two followed by reduction
and application of $\Pi$. These squares generate $R_2$.
\end{proof}

\subsection{The triangular calculation and the proof}

Write $v=\sum_{j=0}^7r_j(p,q,\rho)L_j$ as in \eqref{mf:eq:eight}.
Subscripts $p,q$ denote formal derivatives with $\rho$ fixed. Define
\begin{equation}\label{mf:eq:U}
\begin{split}
  U(v)={}&r_{4,p}L_0+\rho r_{5,p}L_1+r_{4,q}L_2\\
         &+(r_4+\rho r_{5,q})L_3
           +\rho(r_{5,q}+r_{7,p})L_4+\rho r_{7,q}L_6.
\end{split}
\end{equation}
Only $r_4,r_5,r_7$ enter $U$.

\begin{lemma}\label{mf:lem:beta}
On the odd reduced norm kernel, we have
\begin{equation}\label{mf:eq:beta-result}
 \beta_+(v)=U(v),\qquad \beta_-(v)=v+U(v).
\end{equation}
Moreover, we have $U^3=0$.
\end{lemma}
\begin{proof}
Splitting powers of $p,q$ into even and odd parts gives the $R_2$-basis
$1,q,p,pq$ of $R_0$. Lemma~\ref{mf:lem:well-defined} therefore reduces
\eqref{mf:eq:beta-result} to the $32$ polynomials
\[
 V_{4j}=L_j,\quad V_{4j+1}=qL_j,\quad
 V_{4j+2}=pL_j,\quad V_{4j+3}=pqL_j.
\]
Their degrees are at most $27$. Use the integral polynomial formulas
above to form each norm, divide its even coefficients by two, and apply
the $18$ substitutions in \eqref{mf:eq:Tlist}. Exact expansion gives
$\beta_-(V_\nu)=V_\nu+\beta_+(V_\nu)$ in every column.
All nonzero $\beta_+$ columns are the following:
\begin{equation}\label{mf:eq:table}
\begin{array}{c|l@{\qquad}c|l}
\nu&\beta_+(V_\nu)&\nu&\beta_+(V_\nu)\\ \hline
16&V_{12}&17&V_8+V_{13}\\
18&V_0+V_{14}&19&V_1+V_{10}+V_{15}\\
21&\rho(V_{12}+V_{16})&22&\rho V_4\\
23&\rho(V_5+V_{14}+V_{18})&29&\rho V_{24}\\
30&\rho V_{16}&31&\rho(V_{17}+V_{26})
\end{array}
\end{equation}
The other $22$ columns are zero. This is exactly \eqref{mf:eq:U}, since
formal differentiation in $p,q$ is $R_2$-linear. Thus the finite check
proves the identity in every degree, not only through degree $27$.

The image of $U$ has no $L_5$ or $L_7$ coefficient. Applying $U$ again
therefore leaves only $L_0,L_2,L_3$ terms, and $U$ kills those terms.
Hence, we have $U^3=0$.
\end{proof}

\begin{proof}[Proof of Theorem~\ref{mf:thm:main}]
Reduce $f$ modulo two, obtaining $v$. Alternation becomes symmetry and
$Nf\equiv0\pmod4$ gives $\bn v=0$. Using $f$ itself as the lift in
\eqref{mf:eq:beta}, we get $\beta_-(v)=0$.
By Lemma~\ref{mf:lem:beta}, this gives $(I+U)v=0$.
But $(I+U)^{-1}=I+U+U^2$, so we have $v=0$.

A nonzero rational alternating solution of odd degree can be scaled to
a primitive integer polynomial. The integral assertion would make all
its coefficients even, a contradiction. In each fixed degree the norm
and symmetry conditions are a finite rational linear system. Its kernel
remains zero after extension to any characteristic zero field.
\end{proof}

\section{Conclusion of the proof}\label{sec:assembly}

We assemble the rational cases and extend coefficients. For a commutative
$\QQ$-algebra $R$, define $W_n(R)$ and $\mathscr P_n^0(R)$ by the same
polynomial identities with coefficients in $R$.

\begin{proposition}\label{prop:coefficient-base-change}
For every commutative $\QQ$-algebra $R$ and every $n\ge0$, the natural maps
\begin{equation}\label{eq:coefficient-base-change}
\begin{split}
 R\otimes_{\QQ}W_n&\xrightarrow{\ \sim\ }W_n(R) \qquad {\text {and}}\qquad
 R\otimes_{\QQ}\mathscr P_n^0\xrightarrow{\ \sim\ }\mathscr P_n^0(R)
\end{split}
\end{equation}
are isomorphisms. The same statement holds for the symmetric and alternating
solution spaces and for the even and odd period spaces. The ambient
coefficient-extension maps commute with restriction to a coordinate plane
and with cyclic summation.
\end{proposition}
\begin{proof}
In monomial bases, the defining identities are finite linear systems over
$\QQ$. Every $\QQ$-algebra is flat over $\QQ$, so extension of scalars
preserves their kernels. This proves the two isomorphisms and applies
unchanged after imposing symmetry or parity. Restriction to a coordinate
plane and cyclic summation have rational matrices, so they commute with
extension of scalars as well.
\end{proof}

\begin{proof}[Proof of Theorems~\ref{thm:main} and \ref{thm:symmetric-classification}]
Lemma~\ref{lem:sector-decomp} reduces to four homogeneous symmetry-parity
sectors. Section~\ref{sec:symmetric} proves the symmetric classification,
Proposition~\ref{prop:alternating-even} gives alternating even degree
support, and Theorem~\ref{mf:thm:main} gives alternating odd degree
vanishing. Thus every solution has two-variable support and every odd
piece vanishes. The symmetric summand is $f(X,Y,0)$, and its converse
follows from the pairwise cancellation at the end of Section~\ref{sec:symmetric}.

The operator $\partial_x\partial_y\partial_z$ kills exactly the monomials
missing a variable. The remaining monomials have distinct nonzero
derivatives in characteristic zero, proving the differential formulation.

Proposition~\ref{prop:coefficient-base-change} writes every solution over
a characteristic zero field as a linear combination of rational solutions.
It therefore preserves support and odd degree vanishing. Its compatibility
with restriction preserves the symmetric classification and the unique
choice $q=f(X,Y,0)$. This proves both theorems.
\end{proof}

\section{Period polynomials, dimensions, and the common divisor}
\label{sec:period-corollaries}

We identify the alternating binary restrictions, prove the introductory
corollaries, and describe the dependence on the coefficient ring,
including the failure of support in positive characteristic.

\subsection{The complete binary restriction}

Theorem~\ref{thm:symmetric-classification} gives the symmetric case.
For alternating solutions, support and odd degree vanishing leave the
following positive even degree classification.

\begin{lemma}\label{lem:alternating-classification}
Let $F$ be a field of characteristic zero and let $n>0$ be even. An
alternating polynomial $f\in F[x,y,z]_n$ is a solution if and only if
$f=\Phi(q)$, where $q\in F[X,Y]_n$ is alternating, even in each variable,
zero on both axes, and satisfies \eqref{eq:binary-period-intro}.
In that case, we have $q(X,Y)=f(X,Y,0)$.
\end{lemma}

\begin{proof}
Suppose that $f$ is a solution. Theorem~\ref{thm:main} gives its two-variable
support. The norm at $(0,0,z)$ gives $3f(0,0,z)=0$, so cyclicity removes all
pure powers. Every remaining monomial is supported on exactly two variables.
Cyclicity therefore gives
$f=\Phi(q),\quad q(X,Y)=f(X,Y,0)$.
The polynomial $q$ is alternating and vanishes on the axes. Restriction of
the norm to $z=0$ gives $N_2q=0$.

It remains to prove separate evenness. Put
$r(u,v)=q(u,v)-q(-u,v)$. Expanding $N\Phi(q)$ and using alternation and
$N_2q=0$ gives
\begin{equation}\label{eq:alternating-additivity}
 r(y,z)-r(x,z)+r(x-y,z)=0.
\end{equation}
Indeed, after removing the three terms of $N_2q(x,y)$, the six remaining
terms are
\[
 q(y,z)-q(x,z)+q(x-y,z)-q(-y,z)+q(-x,z)-q(y-x,z),
\]
which is the left side of \eqref{eq:alternating-additivity}.
Thus, we have $r(u+v,z)=r(u,z)+r(v,z)$. Differentiating in $v$ at $v=0$ shows that
$\partial_ur(u,z)$ is independent of $u$. Since $r(0,z)=0$, we have
$r(u,z)=u h(z)$ for a polynomial $h$.

Since $q$ is alternating and homogeneous of even total degree,
we have $r(v,u)=-r(u,v)$.
On the diagonal, this means that $2u h(u)=0$, so $h=0$ in characteristic zero.
Hence, we have $q(-u,v)=q(u,v)$. Alternation gives evenness in the other variable.
Conversely, for a binary polynomial with these properties, the three
terms of $N_2q$ sum to zero and the six terms displayed above cancel. Hence
$\Phi(q)$ is a solution. The axis normalization recovers $q$ on $z=0$.
\end{proof}

\begin{proof}[Proof of Corollary~\ref{cor:complete-periods}]
The degree zero spaces are zero. The relation $q(X,Y)+q(-Y,X)=0$, applied
twice, gives $q(-X,-Y)=q(X,Y)$. Thus $\mathscr P_n^0=0$ when $n$ is odd,
as is $W_n$ by Theorem~\ref{thm:main}.

Let $n>0$ be even. The involution $q(X,Y)\mapsto q(-X,Y)$ preserves
the binary period relations \cite[Section~2, (2.2)--(2.4)]{PasolPopa}.
To match conventions, homogenize by $P(X)\mapsto Y^nP(X/Y)$.
The matrix $U^2$ in that reference gives our substitution
$(X,Y)\mapsto(-Y,X-Y)$. Since $U^3=-I$ and $-I$ acts trivially in
even degree, the two three-term norms agree. The involution also preserves
the axis conditions, so $(q(X,Y)\pm q(-X,Y))/2$ gives the two normalized
parity projections.

This is a statement about binary periods, not coordinate sign projections
of the ternary norm. The two-term relation makes the odd part symmetric
and separately odd, and the even part alternating and separately even.
In the notation of Section~\ref{sec:polynomial}, we have
$\mathscr P_n^{0,\mathrm{odd}}=\mathcal B_n$ and $\mathscr P_n^{0,\mathrm{ev}}=\mathcal B_n^0$.
Theorem~\ref{thm:symmetric-classification} and
Lemma~\ref{lem:alternating-classification} identify these spaces with the
two sectors exhausting $W_n$ by Lemma~\ref{lem:sector-decomp}. Both inverse
maps are $\Phi$. Since the summands vanish on the axes, its restriction to
$z=0$ is exactly $q$. This proves the rational isomorphism and uniqueness.
Compatibility with restriction and $\Phi$ in
Proposition~\ref{prop:coefficient-base-change} extends the result to $F$.
\end{proof}

\subsection{Cusp form periods and dimensions}

Let $n>0$ be even, $k=n+2$, and $h\in S_k$.
Expanding \eqref{eq:analytic-period} and using the Mellin transform,
with the normalization of $\Lambda(h,s)$ in
Section~\ref{sec:classical-conventions}, gives
\begin{equation}\label{eq:period-critical-coefficients}
  r_h(X,Y)
    =\sum_{j=0}^n(-1)^{n-j}\binom nj
       i^{j+1}\Lambda(h,j+1)X^{n-j}Y^j.
\end{equation}
Thus the coefficients are explicit multiples of the completed
$L$-values at the critical integers $1,\ldots,k-1$
\cite[Section~5.10]{DiamondShurman}.

Define the even and odd parts in $X$ by
\[
\begin{aligned}
  r_h^{\mathrm{ev}}(X,Y)
    &=\frac{r_h(X,Y)+r_h(-X,Y)}2,\\
  r_h^{\mathrm{odd}}(X,Y)
    &=\frac{r_h(X,Y)-r_h(-X,Y)}2.
\end{aligned}
\]
Both satisfy the period relations. The odd part is odd in both
variables and vanishes on the axes. The even part is normalized uniquely by
\begin{equation}\label{eq:even-period-normalization}
  r_h^{\mathrm{ev},0}(X,Y)
    =r_h^{\mathrm{ev}}(X,Y)
       -r_h^{\mathrm{ev}}(1,0)(X^n-Y^n).
\end{equation}
This subtraction preserves the period relations. Hence
$r_h^{\mathrm{odd}}$ and $r_h^{\mathrm{ev},0}$ belong to
$\mathscr P_n^{0,\mathrm{odd}}(\CC)$ and
$\mathscr P_n^{0,\mathrm{ev}}(\CC)$, respectively.

\begin{proof}[Proof of Corollary~\ref{cor:dimensions}]
Eichler--Shimura identifies cusp forms with the two period spaces modulo
coboundaries \cite[Theorem~2.1]{PasolPopa}. At level one the coboundary
space is spanned by $X^n-Y^n$. It lies in the even part, and
\eqref{eq:even-period-normalization} selects its unique axis-normalized
representative. Thus the two period maps give
\[
 \begin{aligned}
 S_{n+2}&\xrightarrow{\sim}\mathscr P_n^{0,\mathrm{odd}}\otimes\CC,
 &h&\mapsto r_h^{\mathrm{odd}},\\
 S_{n+2}&\xrightarrow{\sim}\mathscr P_n^{0,\mathrm{ev}}\otimes\CC,
 &h&\mapsto r_h^{\mathrm{ev},0}.
 \end{aligned}
\]
Composing with $\Phi$ proves the claimed
identifications. The coefficient-extension proposition transfers the
complex dimensions to every characteristic zero field. Odd degrees and
degree zero vanish by the classification. Finally,
the graded space $S_*:=\bigoplus_{k\ge0}S_k$ satisfies
$S_* =\Delta_{\mathrm{mod}}\CC[E_4,E_6]$, where $E_4,E_6$ are the
normalized Eisenstein series of weights four and six and
$\Delta_{\mathrm{mod}}$ is the weight twelve discriminant cusp form. This gives
$\sum_k\dim S_k\,t^k=t^{12}/((1-t^4)(1-t^6))$
\cite[Chapter VII]{SerreArithmetic}. Shift $k$ to $n+2$ and multiply by two
to obtain \eqref{eq:W-Hilbert}.
\end{proof}

\subsection{The exact common divisor}

\begin{proof}[Proof of Corollary~\ref{cor:common-divisor}]
By the
classification, each nonzero sector is $\Phi(q)$ for a normalized binary
period of one of the two parities, and in both cases we have
\[
 q(u,u)=q(u,-u)=0,\qquad q(-u,v)=-q(v,u).
\]
For the symmetric case the first diagonal vanishing follows by evaluating
the binary norm at $(u,u)$ and using the axis values. For the alternating
case it follows directly from alternation. The other identities follow
from the parity conditions. Thus, we have
$\Phi(q)(u,-u,v)=0$.
Cyclicity supplies the factors $x+y$, $y+z$, and $z+x$.
The construction of $L(q)$ in \eqref{eq:polynomial-lift} already proves divisibility
by $x+y+z$ for either symmetry. The four linear factors are pairwise
nonassociate, so their product $D$ divides each sector and every solution.

To prove that $D$ is the entire common divisor, use the degree ten basis
$f_{\mathrm{s}}=\Phi(q_{\mathrm{s}})$ and
$f_{\mathrm{a}}=\Phi(q_{\mathrm{a}})$ from the introduction. Expansion gives
\begin{equation}\label{eq:sharp-divisor-factorization}
\begin{split}
 f_{\mathrm{a}}={}&D(x,y,z)(x-y)(x-z)(y-z)
 \cdot(x-y-z)(x-y+z)(x+y-z).
\end{split}
\end{equation}
All ten linear factors are distinct. On the other hand, we have
\[
 f_{\mathrm{s}}(u,u,v)=2q_{\mathrm{s}}(u,v),\qquad
 f_{\mathrm{s}}(u+v,u,v)=2q_{\mathrm{s}}(u,v).
\]
Neither is identically zero, and symmetry of $f_{\mathrm{s}}$ shows
that none of the six additional factors in
\eqref{eq:sharp-divisor-factorization} divides it. Therefore, we have
$\gcd(f_{\mathrm{s}},f_{\mathrm{a}})=D$ up to a nonzero constant.
Since both are solutions and $D$ divides every solution, the result follows.
\end{proof}

\subsection{Changing the coefficient ring}

For a commutative $\QQ$-algebra $R$,
Proposition~\ref{prop:coefficient-base-change} extends the classification,
odd degree vanishing, and divisibility by $D$ to $R$, including rings
with nonzero nilpotents.

More generally, a commutative ring $R$ with torsion free additive group
embeds in $R_\QQ=R\otimes_\ZZ\QQ$. A solution's restriction
$q(X,Y)=f(X,Y,0)$ already lies in $R[X,Y]$. The identity $f=\Phi(q)$
and the binary period relations descend coefficientwise from $R_\QQ$,
as do support and odd degree vanishing.

Divisibility by $D$ also descends to $R$. Its leading coefficient
in lexicographic order $x>y>z$ is one, so polynomial division by
$D$ can be performed over $R$ without dividing any coefficients.
The remainder becomes zero over $R_\QQ$, and injectivity makes
it zero over $R$ as well. The greatest common divisor assertion
is made only over fields.

In positive characteristic, the support conclusion fails.
Let $p$ be a prime, and let $\overline f_{\mathrm s}$ be the
reduction modulo $p$ of the integral degree ten symmetric
solution from the introduction. This reduction is nonzero,
since the coefficients of $f_{\mathrm s}$ include $4$ and $-25$.
Define
\[
  \Pi_p(x,y,z)
    =\prod_{0\ne\lambda\in\FF_p^3}
       (\lambda_1x+\lambda_2y+\lambda_3z).
\]
Every invertible linear substitution over $\FF_p$ permutes
the factors, so $\Pi_p$ is invariant under $\GL_3(\FF_p)$.
In particular, it is invariant under cyclic permutation and
the substitutions in $N$. Hence
\[
  N(\Pi_p\overline f_{\mathrm s})
    =\Pi_pN\overline f_{\mathrm s}=0,
\]
and $\Pi_p\overline f_{\mathrm s}$ is cyclic.
It is a nonzero polynomial of degree $p^3+9$.
Since the factors of $\Pi_p$ include $x$, $y$, and $z$,
every monomial of this solution involves all three variables.
Thus it cannot be a sum of polynomials in two variables.
The same example works over every field of characteristic $p$.

In characteristic two there is a smaller example of degree
eleven:
\[
  g=xyz\,\overline D^{\,2}\in\FF_2[x,y,z],
\]
where $\overline D$ denotes the reduction of $D$.
For the integral polynomial $q(X,Y)=X^3Y+XY^3$, direct
substitution gives
\[
  N_2q(X,Y)=2(X^2-XY+Y^2)^2.
\]
Modulo two, we have $\overline D=\Phi(\overline q)$ and hence
$N\overline D=0$. Moreover, $xyz\,\overline D$ is the product
of the seven nonzero linear forms in three variables over
$\FF_2$, so it is invariant under $\GL_3(\FF_2)$.
Therefore $g=(xyz\,\overline D)\overline D$ is a nonzero
solution. It violates both two-variable support and
odd degree vanishing.

\appendix
\section{The restricted apartment presentation}
\label{app:topology}

We prove the acyclicity required in Proposition~\ref{thm:good-apartment}.
Let $X=\mathrm{BA}_3(\ZZ)$ be the partial augmented-frame complex.
Its vertices are primitive integral lines $v^\pm$, identifying $v,-v$.
A frame comes from an integral basis. An augmented frame adds the line
of a signed sum of two basis vectors. Their subsets are the simplices
of $X$. Let $Y$ be the full good subcomplex for $\ell=(1,1,1)$ modulo three.

Church--Putman \cite[Theorem~C]{ChurchPutman} prove that $X$ is
$2$-connected. Deleting vertices can create homology, so we must analyze
the deleted faces and fill the resulting obstructions explicitly.

We will use the corresponding rank two fact. The full augmented-frame
complex in rank two is the contractible filled Farey complex
\cite[Remark~1.4]{ChurchPutman}. Two bad vertices cannot be adjacent,
since their reductions lie in the same one-dimensional subspace.
The link of a bad vertex is therefore a bi-infinite path of good
vertices. Its star is a cone on this contractible path, so removing
that vertex and the interior of its star preserves the homotopy type.
For a map of a sphere and a chosen nullhomotopy, only finitely many
simplices are involved. Removing those bad stars successively proves
that the good rank two subcomplex is connected and weakly contractible,
and hence contractible by Whitehead's theorem
\cite[Chapter~4]{Hatcher}.

\subsection{The bad-face filtration}

For an all-bad simplex $\beta$, its good link
$L_\beta=\operatorname{lk}_X(\beta)\cap Y$ consists of good simplices
disjoint from $\beta$ whose join with $\beta$ lies in $X$.
Write $\operatorname{or}(\beta)$ for its rational orientation line.

Every simplex is the join of its bad face and its good face. Filter
the augmented chain complex of $X$ by the dimension $p$ of the bad
face, allowing $p=-1$ for an entirely good simplex. There are no
all-bad basis triangles or tetrahedra: their reductions would have to
span $\FF_3^3$ inside $\ker\ell$. The only all-bad triangles are
augmented rank two frames. The filtration therefore has four columns,
and gives the spectral sequence
\begin{equation}\label{eq:bad-filtration}
 E^1_{p,q}=
 \begin{cases}
  \widetilde H_q(Y;\QQ),&p=-1,\\[3pt]
  \displaystyle\bigoplus_{\substack{\beta\text{ all-bad}\\\dim\beta=p}}
    \operatorname{or}(\beta)\otimes\widetilde H_q(L_\beta;\QQ),
       &0\le p\le2,
 \end{cases}
 \quad\Longrightarrow\quad\widetilde H_{p+q+1}(X;\QQ).
\end{equation}
We use augmented homology, including degree $q=-1$.
The differentials have bidegree
$d^r:(p,q)\mapsto(p-r,q+r-1)$.

A chain in the $\beta$-summand is a join $\beta*\gamma$, with boundary
\[
 \partial(\beta*\gamma)
   =(\partial\beta)*\gamma+(-1)^{p+1}\beta*(\partial\gamma).
\]
The second term computes good link homology at fixed $p$, and the first
removes a bad vertex. This gives the $E^1$-page and its indexing.
Boundedness of the filtration and finite support of chains ensure
convergence for these infinite complexes
\cite[Chapter~VII, Section~2]{Brown}.

Two types of links are immediate. If a bad augmented triangle spans
the saturated plane with basis $b,c$, and $w$ is an integral complement,
its good link is the nonempty discrete set
\[
 \{(w+rb+sc)^\pm:(r,s)\in\ZZ^2\}.
\]
For the bad edge $\{b^\pm,c^\pm\}$, the same vertices form a square-grid
graph, with edges given by differences $\pm b$ and $\pm c$.
It is connected, and its first homology is generated by unit squares.
We next treat the bad-vertex links.

\subsection{The bad-vertex link}\label{app:bad-vertex-link}

Fix a bad primitive vector $b$, and put $L_b=\operatorname{lk}_X(b^\pm)\cap Y$.
The covector $\ell$ descends to a nonzero covector on
$(\ZZ^3/\ZZ b)/3(\ZZ^3/\ZZ b)$. Let $Z$ be the good rank two
augmented-frame complex in this quotient. It is contractible by the
rank two argument above. Projection gives a simplicial map
$p:L_b\to Z$.

Over $u\in Z$, the lifts $u_k=u_0+kb$, $k\in\ZZ$, of a primitive
$u_0$ form a bi-infinite path $P_u$ with edges $[u_k,u_{k+1}]$.
Over $[u,v]$, all cross-edges $[u_k,v_l]$ occur, together with
\[
 [u_k,u_{k+1},v_l],\qquad [u_k,v_l,v_{l+1}].
\]
Thus this part of the link is the $2$-skeleton of the join $P_u*P_v$.
Over an augmented base triangle, choose compatible representatives and
lifts with $u_0=u_1+u_2$. The remaining triangles are
\[
 \theta_{k,l}
  =[u_1+kb,u_2+lb,u_0+(k+l)b],\qquad(k,l)\in\ZZ^2.
\]
These exhaust the possibilities: the additive relation of an augmented
frame containing $b$ either involves $b$, giving a path-edge triangle,
or relates the other three lines, giving one of the $\theta_{k,l}$.

\begin{lemma}\label{lem:bad-link-generators}
The link $L_b$ is connected and $H_1(L_b;\QQ)=0$.
Its second homology is generated by the tetrahedral boundaries
\[
 \partial[u_k,u_{k+1},v_l,v_{l+1}]
\]
over base edges, and by the prism boundaries joining $\theta_{k,l}$
to $\theta_{k+1,l}$ or to $\theta_{k,l+1}$ over augmented base triangles.
\end{lemma}
\begin{proof}
Attach a $3$-cell along each displayed sphere, obtaining $L_b^+$.
Filter this complex by the dimension of its image in $Z$:
\[
 F_{-1}=\varnothing\subset F_0\subset F_1\subset F_2=L_b^+.
\]
We compute the homology of each successive pair.

Over a vertex, the path $P_u$ is contractible. Over an edge $[u,v]$,
the tetrahedral cells complete the join $P_u*P_v$. Relative to the
endpoint fibers, the contribution is
\[
 H_i(P_u*P_v,P_u\sqcup P_v;\QQ)
   =\begin{cases}\QQ,&i=1,\\0,&i\ne1.\end{cases}
\]
This follows from the long exact sequence of the pair: the join is
contractible and the two endpoint fibers are contractible components.
Any cross-edge from $P_u$ to $P_v$ represents the generator.

Over an augmented base triangle, the relative complex has only the
coherent triangles and the prism cells. It is
\begin{equation}\label{eq:link-relative-grid}
 \QQ[\mathcal E(\ZZ^2)]\xrightarrow{\delta}\QQ[\ZZ^2],
 \qquad\delta[\nu,\nu']=\theta_{\nu'}-\theta_\nu,
\end{equation}
in degrees three and two, where $\mathcal E(\ZZ^2)$ consists of the
oriented horizontal and vertical edges of the integer grid.
The side faces disappear in the relative quotient. Connectedness of
the grid makes $\operatorname{coker}\delta=\QQ$, while its kernel
contributes only in degree three.

Consequently, through total degree two the spectral sequence
$E^1_{p,q}=H_{p+q}(F_p,F_{p-1};\QQ)$ is concentrated in the row
$q=0$. This row is the simplicial chain complex of $Z$: cross-edges
have the expected endpoints, and coherent triangles have the expected
three boundary edges. The only additional term, $E^1_{2,1}$, has
zero possible targets $E^1_{1,1}$ and $E^2_{0,2}$. Since $Z$ is
contractible, $L_b^+$ is connected and has zero homology in degrees
one and two.

Only $3$-cells were added to $L_b$. The long exact sequence of
$(L_b^+,L_b)$ therefore preserves connectedness and first homology,
and gives a surjection
\[
 H_3(L_b^+,L_b;\QQ)\longrightarrow H_2(L_b;\QQ).
\]
The generators on the left map to the attaching spheres, which are
precisely the two families in the statement.
\end{proof}

\subsection{The three possible obstructions}

In the good column $p=-1$ of \eqref{eq:bad-filtration}, the connected
bad-vertex and bad-edge links and nonempty bad-triangle links rule out
incoming differentials to $\widetilde H_0(Y)$. For $H_1(Y)$, the possible
sources are
\[
 H_1(L_b),\qquad\widetilde H_0(L_e),\qquad
 \widetilde H_{-1}(L_\tau),
\]
for bad vertices $b$, edges $e$, and triangles $\tau$.
All vanish by the link calculations.

For $H_2(Y)$, the possible sources are $H_2(L_b)$ on the first page,
$H_1(L_e)$ on the second, and $\widetilde H_0(L_\tau)$ on the third.
On later pages one takes their surviving subquotients. We show that
each incoming differential is zero on the generators described above.
Brackets below denote oriented simplices of primitive lines.

\smallskip\noindent\emph{The first differential.}
Let $(b,c,w)$ be an integral basis with $b,c$ bad and $w$ good. Put
\[
 A=w,\quad B=w+b,\quad C=w+c,\quad D=w+b+c,
 \qquad W=A+D=B+C.
\]
All five vectors are good: if $\ell(w)=1$ modulo three, then
$\ell(W)=2$. The all-good augmented-frame chain
\begin{equation}\label{eq:square-cap}
 [W,B,C,D]-[W,A,C,D]+[W,A,B,D]-[W,A,B,C]
\end{equation}
has boundary $\partial[A,B,C,D]$. This fills each tetrahedral
generator from Lemma~\ref{lem:bad-link-generators}. Indeed, normalize
the two lifted quotient vectors to have the same nonzero covector
value. Their difference is a bad vector $c$, giving this configuration.

For a prism generator, write
\[
 Z=X+Y,\qquad X'=X+b,\qquad Z'=Z+b=X'+Y.
\]
Its sphere is the boundary of the two all-good augmented tetrahedra
\begin{equation}\label{eq:prism-cap}
 [X,Y,Z,X']+[Y,Z,X',Z'].
\end{equation}
Interchanging $X,Y$ treats the other parameter direction. Hence every
generator of $H_2(L_b)$ has zero image in the good column.

\smallskip\noindent\emph{The second differential.}
A generator of $H_1(L_e)$ is a unit square in the good link of a bad
edge. In the notation above, its fillings in the two bad-vertex links
are
\[
 [A,B,C]+[B,D,C],\qquad [A,B,D]+[A,D,C].
\]
Their difference is $-\partial[A,B,C,D]$, which bounds
\eqref{eq:square-cap} in $Y$. These fillings therefore give zero
incoming second differential.

\smallskip\noindent\emph{The third differential.}
Let $u,v,s=u+v$ be a bad augmented core, let $w$ be a good complement,
and put $w'=w+u$ and $q=v-w$. Consider
\begin{align}\label{eq:triangle-cycle}
 Z_0={}&-[w,q,v,s]-[w,v,u,s]-[w,u,w',s]+[w,q,v,w']\nonumber\\
      &+[w,q,w',s]+[w,v,u,w']+[q,v,w',s]+[v,u,w',s].
\end{align}
Each tetrahedron is an augmented frame, using one of
\[
 v=w+q,\qquad s=u+v,\qquad w'=w+u,\qquad s=q+w'.
\]
Expansion gives $\partial Z_0=0$. Its terms with three bad vertices are
$-[w,v,u,s]+[v,u,w',s]$, so its highest filtered component represents,
up to orientation, the adjacent complement difference $[w']-[w]$.
Since this component lifts to the absolute cycle $Z_0$, all its
successive differentials are zero. Translating $w$ and interchanging
$u,v$ gives every adjacent difference in the complement lattice.
These generate $\widetilde H_0(L_\tau)$, so the incoming third
differential is zero.

There is no outgoing differential from the good column. All incoming
differentials in total degrees at most two have now been excluded.
The vanishing of the corresponding homology of $X$ therefore gives
\[
 \widetilde H_0(Y;\QQ)=H_1(Y;\QQ)=H_2(Y;\QQ)=0.
\]
This is the restricted acyclicity needed in
Proposition~\ref{thm:good-apartment}.

\section{Coefficient comparisons for the symmetric argument}\label{app:coefficient-comparisons}

We prove the seed finite-difference identity and identify literal residues
with binary periods and their Hecke action. The comparison must preserve
individual coefficients: pole removal uses an extreme period.

\subsection{The finite-difference identity}\label{app:finite-difference}

The left side of \eqref{eq:common-denominator} is
$\sum_{i=1}^{12}\epsilon_i(Nf)(T_i)$ with the substitutions in
Table~\ref{tab:finite-difference}, read down the left column and then
the right. In the thirty-six expanded terms, those from $Af,A^2f$
cancel in opposite-sign pairs with the same unordered arguments.
The twelve terms from $f$ are the three second differences.
This proves the identity for every symmetric $f$ with $Nf=0$,
without a homogeneity or degree assumption.

\begin{table}[!htbp]
\centering
$\begin{array}{|r|l|r|l|}
\hline
\epsilon_i&T_i&\epsilon_i&T_i\\
\hline
+&(b,a+c,-c)&-&(b-c,b,a-b+c)\\
 -&(b,b+c,a-b-c)&+&(b,a-c,c)\\
 +&(c-b,a-b,-c)&-&(a-b,a-b+c,-a-c)\\
 -&(a-b-c,a-b,c-a)&+&(-b-c,a-b,c)\\
 +&(-a,b-a+c,-c)&-&(-a-c,-a,b+c)\\
 -&(-a,-a+c,b-c)&+&(-a,b-a-c,c)\\
\hline
\end{array}$
\par\vspace{7pt}
\caption{The signed substitutions expressing  \eqref{eq:common-denominator} as a sum of three-term norms.}
\label{tab:finite-difference}
\end{table}
\FloatBarrier

\subsection{The literal residue and its Hecke action}
\label{app:Hecke-coordinate}

We compare the literal $1/a$ residue used in Section~\ref{sec:hecke-pole}
with the binary period coordinate, without mixing coefficients.

Let $H=\{a=0\}$, let $P_H$ be its stabilizer in $G=\GL_3(\ZZ)$, and
put $D_d=\QQ[b,c]_d$ with the right substitution action. All rational
hyperplanes form one $G$-orbit, so Shapiro reduces the direct sum residue
module to its fiber at $H$. The top homology of the rank two radical of
$P_H$ gives the Hochschild--Serre edge map:
\[
\begin{aligned}
 \mathrm{Sh}:H_3(G,\mathcal R_n\otimes\det)
 &\xrightarrow{\sim}H_3(P_H,\mathcal R_{n,H}\otimes\det),\\
 \kappa:H_3(P_H,\mathcal R_{n,H}\otimes\det)
 &\xrightarrow{\sim}H_1(\GL_2(\ZZ),D_{n+1}).
\end{aligned}
\]
The map $\operatorname{res}_H$ reads the literal residue, and
$\operatorname{edge}_2$ reads the coefficient of the standard oriented
binary well-rounded edge. The next lemma compares them with compatible
orientation and averaging conventions.

\begin{lemma}\label{lem:literal-residue-comparison}
With compatible orientation and averaging normalizations, the diagram
\begin{equation}\label{eq:residue-comparison-diagram}
\begin{CD}
 H_3(G,\mathcal R_n\otimes\det)
     @>{\mathrm{Sh}}>>H_3(P_H,\mathcal R_{n,H}\otimes\det)\\
 @V{\operatorname{res}_H}VV @VV{\kappa}V\\
 \mathcal B_{n+1}
     @< {\operatorname{edge}_2}<<H_1(\mathrm{GL}_2(\ZZ),D_{n+1})
\end{CD}
\end{equation}
commutes. Changing these normalizations multiplies the comparison by one
fixed nonzero scalar, independent of the polynomial. In particular, it does
not mix its coefficients. The composite identifies the residue with the
usual odd period polynomial in the same variables $b,c$.
\end{lemma}
\begin{proof}
The stabilizer of $a=0$ in $\mathrm{GL}_3(\QQ)$ consists of
$g=\left(\begin{smallmatrix}\lambda&0\\u&C\end{smallmatrix}\right)$.
On the determinant-twisted residue its action is
$q\mapsto\det(C)q(C(b,c))$. Its integral radical $U=\ZZ^2$ acts trivially
on that fiber. Right homology acts on $H_2(U,\QQ)$ by inverse conjugation,
with factor $\lambda^2/\det C$. The combined factor is therefore $\lambda^2$.
On the integral Levi this is one. The only total-degree-three
Hochschild--Serre term is $(1,2)$, giving
\[
 H_3(P_H,\mathcal R_{n,H}\otimes\det)
   \cong H_1(\mathrm{GL}_2(\ZZ),D_{n+1}).
\]
This is $\kappa$. Shapiro is Hecke-equivariant for this
finite-support system \cite[Theorem 3.1]{AshDoud}. The following cell
calculation specifies the coordinate. The standard
binary well-rounded edge has minimum-one Gram matrices
$C(t)=\left(\begin{smallmatrix}1&t\\t&1\end{smallmatrix}\right)$, $|t|\le1/2$.
A coordinate sign flip reverses the edge, whereas a coordinate swap fixes
it. The coefficient is separately odd and symmetric. The three
edges meeting at a hexagonal vertex give the binary norm.
For a minimum-one binary well-rounded form $C$ and $w\in\RR^2$, put
\[
 \delta_C(w)^2=\min_{z\in\ZZ^2}(w-z)^tC(w-z),\qquad
 G(w,C)=\begin{pmatrix}1+w^tCw-\delta_C(w)^2&w^tC\\Cw&C\end{pmatrix}.
\]
For $C(t)$ the covering radius squared is $1/(2(1+|t|))\le1/2$.
For $t\ge0$ this follows from the Voronoi inequalities
$|w_1+tw_2|,|tw_1+w_2|\le1/2$ and $|w_1-w_2|\le1$.
All vertices have that squared length. A sign flip handles $t\le0$.
Thus the Schur complement is positive.

Translate $w$ into the Voronoi cell of zero. The squared length of an
integral vector $(k,z)$ for $G(w,C)$ is
$k^2(1-w^tCw)+(z+kw)^tC(z+kw)$.
For $k=0$ the binary minimum gives the lower bound one. For $k=\pm1$,
the same bound follows from the Voronoi inequality. If $|k|\ge2$, the first
term alone is at least $k^2/2\ge2$.
The vector $(1,0)$ and two independent binary shortest vectors have length
one, so this is a ternary well-rounded form.
For $p=\left(\begin{smallmatrix}\eta&0\\u&D\end{smallmatrix}\right)$, with $\eta=\pm1$,
put $C'=D^tCD$, $w'=D^{-1}(u+\eta w)$. Then a direct computation gives
$p^tG(w,C)p=G(w',C')$.
The map is injective, and its image is the union of integral parabolic
translates of the standard ternary top cell: reduce the binary block and
translate $w$ to its Voronoi polygon. Conversely, the inequalities for
$(1,z)$ in that cell are exactly the same Voronoi inequalities.

Let $W_3$ be the ternary retract and $W_2$ the binary retract. The map
$i:(w,C)\mapsto G(w,C)$ identifies this with $\RR^2\times W_2$ and is
$P_H$-equivariant. Both this core and $W_3$ are contractible. After equivariant
subdivision, their augmented cellular chains are rational projective
resolutions for $P_H$: the cell stabilizers are finite and averaging is
exact. The cellular inclusion lifts the identity of the trivial module,
so the comparison theorem for resolutions identifies its induced map with
the canonical group homology map \cite[Chapter~I, Sections~7--8]{Brown}.

The radical acts trivially on the residue fiber. Quotienting its
$\RR^2$ coordinate by $\ZZ^2$ gives the oriented torus that represents
$H_2(U,\QQ)$. Fiber integration on the cellular cross product is
$[\ZZ^2\backslash\RR^2]\otimes(e\otimes q) \longmapsto e\otimes q$, where $e$ is an oriented binary edge and the torus has lattice volume one.
Since the Levi has rational cohomological dimension one, bidegree $(1,2)$
is the only term in total degree three. This fiber map is therefore
precisely the Hochschild--Serre map $\kappa$.

For the standard binary edge, take $w$ in the Voronoi cell of zero.
Then $\delta_C(w)^2=w^tCw$ and the first diagonal entry of $G(w,C)$ is one.
The image is the standard ternary cell, and its residue coefficient at
$a=0$ is the same polynomial $q(b,c)$ throughout the fiber. No derivative
or change of the variables $b,c$ occurs in integration. The Jacobian from
$(w_1,w_2,t)$ to $(Cw,t)$ is positive. The ternary cell stabilizer has order
$48$ and the stabilizer also preserving $H$ has order $16$. Their three
cosets give the three coordinate residues in $F_q$. Cellular averaging and
orientation can consequently change this calculation only by one common
nonzero scalar. This proves the commutative diagram after the stated
normalization.

Finally, the binary modular symbol from $0$ to $i\infty$ crosses the
standard edge once. Pairing with an arbitrary dual binary coefficient
identifies its edge coefficient with the odd part of
$\int_0^{i\infty}h(z)(zc-b)^{n+1}\,dz$, with the same orientation
normalization. Thus the comparison is the usual period map in the literal
variables \cite[Theorem 2.1 and Section 5]{PasolPopa}.
\end{proof}

\smallskip
\noindent\emph{The Hecke transfer.}
For the double coset of $\operatorname{diag}(1,1,p)$, a complete set
of right coset representatives is
\[
 \begin{pmatrix}
 1&0&0\\
 0&1&0\\
 b&c&p
 \end{pmatrix}
 \quad(0\le b,c<p),
 \qquad
 \begin{pmatrix}
 1&0&0\\
 a&p&0\\
 0&0&1
 \end{pmatrix}
 \quad(0\le a<p),
 \qquad
 \operatorname{diag}(p,1,1).
\]
These $p^2+p+1$ representatives classify index $p$ column lattices by
normalizing the last nonzero coordinate of their annihilating covector
modulo $p$.

Every representative preserves the hyperplane
$H=\{a=0\}$, so the Hecke correspondence acts on its residue fiber. After separating
the normal coordinate from the two coordinates on $H$, the first family
has, for fixed $c$, binary block
\[
 C_c=
 \begin{pmatrix}
 1&0\\
 c&p
 \end{pmatrix},
 \qquad 0\le c<p,
\]
while the parameter $b$ occurs only in the unipotent direction. The second
family has binary block
$ C_\infty=
 \left(\begin{smallmatrix}
 p&0\\
 0&1
 \end{smallmatrix}\right),$
with the parameter $a$ again lying in the unipotent direction.

The first two families give $p$ representatives above each binary Hecke
coset $C_0,\ldots,C_{p-1},C_\infty$. The final representative
$\operatorname{diag}(p,1,1)$ has binary block $I$ and normal scale $p$.

To compute a fiber contribution, write $t=\operatorname{diag}(\lambda,C)$
and set $U_t=\ZZ^2\cap(C/\lambda)\ZZ^2$ and
$e_t=[\ZZ^2:U_t]$. Normalize fundamental classes by covolume in the same
rational plane. Transfer contributes $e_t$, inverse conjugation on the
radical's top homology contributes $\lambda^2/\det C$, and the residue
coefficient contributes $\det C$. Their product is $e_t\lambda^2$.

For $\lambda=1$, the radical quotient has $p$ elements, represented
by $(0,b)$ for $C=C_c$ and by $(a,0)$ for $C=C_\infty$.
Thus $e_t=p$, and the representatives already constitute the transfer sum.
For $\lambda=p$ and $C=I$, one has $U_t=\ZZ^2$ and $e_t=1$,
so the contribution is $p^2$. Therefore, we have
\[
 \begin{aligned}
 T_3(p,1)|_{\mathrm{res}}
 &=
 p\left(
     \sum_{c=0}^{p-1}[C_c]+[C_\infty]
   \right)
 +p^2[I]  
 =pT_2(p)+p^2I.
 \end{aligned}
\]
Here the brackets denote the usual binary double-coset
correspondences. They should not be interpreted as plain substitutions on
an arbitrary polynomial. This proves \eqref{eq:Hecke-formula}.
Multiplying the literal coordinate by a fixed nonzero scalar does not
change this operator relation.

It remains to check that this binary correspondence has
the standard analytic Hecke normalization on period polynomials. Put
$d=n+1$, and let $h$ be a cusp form of weight $d+2$. For a matrix
$g=
 \left(\begin{smallmatrix}
 a&b\\
 c&e
 \end{smallmatrix}\right)$
of positive determinant $m$, the fractional linear action gives
$((gzY-X)\mid g) =\frac{m(zY-X)}{cz+e}, \quad d(gz)=\frac{m\,dz}{(cz+e)^2}$.
Hence, the differential
$h(z)(zY-X)^d\,dz$ transforms with the usual weight $d+2$ Hecke factor.  The
homological correspondence above is the standard Hecke operator on the
period polynomial of $h$. This is the Hecke-equivariance of the
period map in
\cite[Proposition 5.2, Corollary 5.3, and Remark 5.6]{PasolPopa}.

\section{The analytic boundary comparison}\label{app:arithmetic-direct}
The spectral gap criterion gives the boundary surjection in
Proposition~\ref{thm:boundary-evaluation}. Convergent sums of cusp form
differential forms give the plane--line relation used in
Proposition~\ref{thm:J-vanishes}. These are Eisenstein series constructions.
We prove their required constant term identities directly.

Throughout this appendix, calculations with differential forms are made on
a neat principal congruence subgroup $\Gamma\subset\SL_3(\ZZ)$,
of level divisible by three. Such a subgroup can be chosen normal in
both $K$ and $K^\dagger$. Its stabilizer of an integral coordinate flag
is a semidirect product of its unipotent lattice and its integral Levi.
All rational flags are obtained by integral changes of basis, which
preserve the principal congruence subgroup. At the end we take
invariants under the finite covering group. The coefficient pairing and
all restriction maps commute with this averaging. We use
$X_\Gamma=\Gamma\backslash\mathscr X$, where
$\mathscr X=\SL_3(\RR)/\SO(3)$, and write $\overline X_\Gamma$ for its
Borel--Serre compactification. The convention is unchanged for a rational
conjugate of $\Gamma$.

\subsection{A spectral gap and compact support vanishing}\label{ar:sec:coercive}
We need compactly supported closed two-forms to be exact in ordinary
cohomology. The spectral gap criterion of Bergeron--Venkatesh supplies
this vanishing. We explain its application to our noncompact quotient.

Put $G_{\RR}=\SL_3(\RR)$ and $K_c=\SO(3)$, and let $\mathfrak p$ be the
space of symmetric trace zero real $3\times3$ matrices, with complexification
$\mathfrak p_{\CC}$. Use the invariant metric induced by
$\operatorname{tr}(UV)$ on $\mathfrak p$.
For $E=\Sym^n\CC^3$ or its dual, let $\rho$ be its representation and
choose the Hermitian metric induced by the standard metric on $\CC^3$.
The flat-bundle metric at $gK_c$ is $\|v\|_g=\|\rho(g)^{-1}v\|$.
Let $D$ be the flat exterior differential, $D^*$ its $L^2$ adjoint, and
$\Delta_j=DD^*+D^*D$ the Laplacian in exterior degree $j$.

\begin{proposition}\label{ar:prop:coercivity}
For every $n>0$ there is a constant $c_n>0$, independent of $\Gamma$,
such that every compactly supported smooth $E$-valued form on $X_\Gamma$
satisfies
\begin{equation}\label{ar:eq:coercivity}
 \|D\omega\|^2+\|D^*\omega\|^2\ge c_n\|\omega\|^2.
\end{equation}
Consequently the natural map
$H_c^2(X_\Gamma,E)\longrightarrow H^2(X_\Gamma,E)$ is zero.
The conclusion holds for either determinant twist after descent from a
neat normal subgroup.
\end{proposition}
\begin{proof}
The Cartan involution $g\mapsto(g^{-1})^{\mathsf t}$ exchanges
$\Sym^n\CC^3$ with its dual. For $t>1$, the matrix
$\operatorname{diag}(t,t,t^{-2})$ has largest eigenvalues $t^n$ and
$t^{2n}$ on these two spaces, respectively. They are therefore
nonisomorphic when $n>0$.

The proof of \cite[Lemma~4.1]{BergeronVenkatesh} gives a positive constant
$\varepsilon_E$ such that
\[
 \Lambda_{E^*}-\Lambda_\pi\ge\varepsilon_E
 \quad\text{whenever}\quad
 \Hom_{K_c}(\textstyle\bigwedge^j\mathfrak p_{\CC}\otimes E^*,\pi)\ne0,
\]
for any irreducible unitary representation $\pi$ of $G_{\RR}$ and any $j$.
Here $\Lambda$ is the eigenvalue of the Casimir, the invariant
second order operator. We use the estimate for unitary representations,
not the reference's compact quotient conclusion. Kuga's formula
\cite[(3.6.1)]{BergeronVenkatesh} identifies the Laplacian on the indicated
cochains with $\Lambda_{E^*}-\Lambda_\pi$, since $E$ and $E^*$ have the same
Casimir scalar. The reference uses the Killing metric, which is a positive
multiple of ours, so rescaling changes only the positive constant.

Now identify square-integrable forms with
\[
 L^2\Omega^j(X_\Gamma,E)
 \cong\Hom_{K_c}(\textstyle\bigwedge^j\mathfrak p_{\CC}\otimes E^*,
                         L^2(\Gamma\backslash G_{\RR})).
\]
The right regular representation is a direct integral of irreducible
unitary representations, the continuous analogue of a direct sum. The
uniform bound controls every constituent, including the continuous
spectrum, and gives \eqref{ar:eq:coercivity}. Taking the smaller constant
for $E$ and its dual gives $c_n$ for both.

On the complete quotient $X_\Gamma$, choose compactly supported cutoffs
$\chi_R\to1$ with $|d\chi_R|\le C/R$. The commutators with $D,D^*$ are
exterior multiplication and contraction by $d\chi_R$. These bounds and
local smoothing show that closing the differential on
compactly supported forms gives its distributional domain, and likewise
for the adjoint. The estimate extends to the closed quadratic form.
Its self-adjoint Laplacian $\Delta_j$ therefore has a bounded inverse $G_j$.

The identity $D^2=0$ makes $D$ intertwine the adjacent Laplacians and their
resolvents on this closed complex. For compactly supported closed
$\omega$ of degree two, we have $DG_2\omega=G_3D\omega=0$, and hence
\[
 \omega=(DD^*+D^*D)G_2\omega=D(D^*G_2\omega).
\]
Local elliptic regularity makes this primitive smooth. It need not have
compact support, which is exactly the claimed vanishing of the map to
ordinary cohomology. Determinant twists are trivial on $\SL_3$, and
finite group averaging descends a primitive from a neat normal cover.
\end{proof}

Poincar\'e--Lefschetz duality now gives the boundary surjection in
Proposition~\ref{thm:boundary-evaluation}, for either coefficient system
and its dual.

\subsection{Cuspidal face pairings and Stokes's theorem}\label{ar:sec:stokes}
We compare restrictions by pairing global classes with dual coefficients.
The face pairings sum to zero only after accounting for corner classes.
Let $P$ and $Q$ now denote the upper rational parabolic subgroups of block
types $(2,1)$ and $(1,2)$, respectively. Thus $P$ stabilizes the plane
$\langle e_1,e_2\rangle$ and $Q$ the line $\QQ e_1$.
Their intersection is the upper triangular group $B$.
For a face of type $R\in\{P,Q\}$, let $N_R$ be its arithmetic unipotent
radical and $L_R$ its arithmetic Levi quotient. The radical is a lattice
in $\RR^2$, and the Levi has a finite index subgroup acting on a modular curve.
We use the same letters for rational conjugates at the other cusps.

For a coefficient system $E$, define $\cC_R(E)$ to be the direct sum over
faces of type $R$ of the subspaces
\[
 H^1\bigl(L_R,H^1(N_R,E)\bigr)
\]
whose restrictions at all rational cusps of the binary Levi vanish.
These are the mixed cuspidal groups defined in
Section~\ref{sec:symmetric}. They inject into degree two face cohomology
by Hochschild--Serre, since $L_R$ has rational cohomological dimension one.
Write $r_R$ for restriction when all face components are mixed cuspidal.

The evaluation pairing $E\otimes E^*\to\CC$ and the oriented faces give
perfect pairings
\begin{equation}\label{ar:eq:facepair}
 \langle\ ,\ \rangle_R:\cC_R(E)\times\cC_R(E^*)\longrightarrow\CC.
\end{equation}
Here is the coefficient calculation. The two-generator Koszul complex for
the radical pairs degree one with degree one into its top orientation
line. This is Poincar\'e duality on its compact torus fiber with dual local
systems. On the base, the long exact sequence for a compactified modular curve
identifies cuspidal cohomology with the image of compactly supported
cohomology. Poincar\'e--Lefschetz duality identifies the transpose of
this compact-support-to-ordinary map with the corresponding map for the
dual local system. The induced pairing between the two images is therefore
nondegenerate: the annihilator of either image is the kernel of the
transpose map. In the classical period description this is a nonzero
multiple of the Petersson pairing between holomorphic and antiholomorphic
classes \cite[Theorem~3.3]{PasolPopa}. Combining the two pairings proves
\eqref{ar:eq:facepair}. The two internal orientation characters multiply
to the face orientation. All determinant and reflection actions are retained
before taking finite group invariants, so the descended pairing is still
perfect.

\begin{lemma}\label{ar:lem:stokes}
Suppose $a\in H^2(X_\Gamma,E)$ and $b\in H^2(X_\Gamma,E^*)$ have only
mixed cuspidal restrictions on maximal faces and zero restrictions on all
flag strata. Then, we have
\begin{equation}\label{ar:eq:stokes}
 \langle r_Pa,r_Pb\rangle_P+\langle r_Qa,r_Qb\rangle_Q=0.
\end{equation}
The orientations are included in the two pairings.
\end{lemma}
\begin{proof}
Use smooth representatives on $\overline X_\Gamma$, whose inclusion of
the interior is a homotopy equivalence. The cup product has degree four,
and the integral of its restriction over the oriented boundary is zero
by Stokes's theorem.

We identify the integral before discarding any corner coordinate.
Let $F_P,F_Q$ be the two unions of maximal faces and let $F_B$ be the
union of their flag intersections. Use collar thickenings when forming a
cover. A de Rham model for the boundary is the mapping cone
\[
 \Omega^r(F_P;E)\oplus\Omega^r(F_Q;E)\oplus\Omega^{r-1}(F_B;E),\qquad
 d(\alpha,\beta,u)=(d\alpha,d\beta,
             \alpha|_{F_B}-\beta|_{F_B}-du).
\]
Choose closed face representatives. Their restrictions to a flag collar
are exact, since the collar retracts onto the flag and the restricted
cohomology classes are zero. Extend a primitive with a cutoff, and
subtract its differential. The new face representatives vanish on a
smaller collar. After the same change of representative in the mapping
cone, its corner coordinate $u$ is closed.

The corresponding corner class is represented on the boundary by
$d\chi\wedge u$, with $u$ pulled back from the flag. Choose the same
transverse cutoff $\chi$ for the two classes. Face--corner products vanish
because the face representatives vanish on its support. Corner--corner
products vanish because $d\chi\wedge d\chi=0$. Distinct flag collars can
be taken disjoint, since there is no triple intersection of maximal faces
in rank three. The remaining integral is the sum of the two compactly
supported face pairings. Stokes therefore gives \eqref{ar:eq:stokes}.
This argument changes representatives rather than assuming that the
original corner cochain is zero. Finite averaging descends the identity.
\end{proof}

\subsection{Closed cusp form seeds and their sums}\label{ar:sec:series}
We construct global classes by summing closed plane-face forms. The
coefficient weight gives convergence without meromorphic continuation.
For the remainder of the appendix let
\[
 E=E_n\otimes_\QQ\CC=\Sym^n((\CC^3)^*)\otimes\det^{-1},\qquad
 k=n+3>6,
\]
with the left action $\rho(g)f(\mathbf X)=\det(g)^{-1}f(g^{-1}\mathbf X)$,
where $\mathbf X=(X,Y,Z)^{\mathsf t}$.
Plane horospherical coordinates are specified as follows. Write a point of
$\mathscr X$ as $p\SO(3)$ with
\[
 p=\begin{pmatrix}A&tu\\0&t\end{pmatrix},\qquad
 t>0,\quad \det A=t^{-1},\quad u=(u_1,u_2)^{\mathsf t},\quad
 \tau=A i\in\mathfrak H.
\]
Right rotation of $A$ leaves $u,t,\tau$ unchanged. Here $t$ is the
normal height and $(u,\tau)$ are plane-face coordinates.
Let $\Gamma_P=\Gamma\cap P(\QQ)$, and let $\Lambda_P$ be its projection
onto the determinant one binary Levi. If $h\in S_k(\Lambda_P)$, set
\begin{equation}\label{ar:eq:seed}
 \mathcal B=\tau Y-X+(u_1-\tau u_2)Z,\qquad
 \omega_h=h(\tau)\mathcal B^n(du_1-\tau du_2)\wedge d\tau.
\end{equation}
The symbol $\mathcal B$ in this subsection is a linear polynomial, not the
binary period space $\mathcal B_d$. The form is independent of $t$.
Antiholomorphic forms are obtained by complex conjugation, including the
coefficient polynomial and differential.

\begin{lemma}\label{ar:lem:seed}
The form $\omega_h$ is closed and is equivariant for $\Gamma_P$ with
coefficients $E$. Its radical class is the binary degree $n+1$ cusp form
class, up to a nonzero scalar.
\end{lemma}
\begin{proof}
Put $\beta=du_1-\tau du_2$. Then, we have
\[
 d\mathcal B=(Y-u_2Z)d\tau+Z\beta,\qquad
 d\beta=-d\tau\wedge du_2.
\]
Each term in $d(h\mathcal B^n\beta\wedge d\tau)$ has a repeated factor
$\beta$ or $d\tau$, and is zero. A radical translation replaces $u$ by
$u+v$, and its coefficient action replaces $(X,Y)$ by
$(X-v_1Z,Y-v_2Z)$, proving equivariance under the lattice.

For the Levi calculation let
$g=\operatorname{diag}(A,t)$,
$A=\left(\begin{smallmatrix}a&b\\c&d\end{smallmatrix}\right)$, and
$\delta=\det A>0$. We have $u'=Au/t$ and
$\tau'=(a\tau+b)/(c\tau+d)$. Direct substitution gives
\[
 \mathcal B(u',\tau',\mathbf X)
    =\frac{\delta}{c\tau+d}\mathcal B(u,\tau,g^{-1}\mathbf X),\quad
 \beta'=\frac{\delta}{t(c\tau+d)}\beta,\quad
 d\tau'=\frac{\delta}{(c\tau+d)^2}d\tau.
\]
After including the determinant factor in $\rho$, this becomes
\begin{equation}\label{ar:eq:seed-transform}
 g^*\omega_h=\rho(g)\omega_{\delta^{k/2}(h|_k A)}.
\end{equation}
At the neat integral stabilizer, $t=\delta=1$, and the modular transformation
law proves equivariance. The finite reflection projection supplies the
corresponding full-level statement.

At $u=0$, the radical cochain is
$(\tau Y-X)^n(1,-\tau)$. In the potential description of plane-radical
cohomology, it is the gradient of
$-(\tau Y-X)^{n+1}/(n+1)$. This is precisely the required binary
coefficient, with its nonzero scalar specified. Eichler--Shimura theory
\cite[Theorem~2.1]{PasolPopa} therefore shows that the holomorphic and
antiholomorphic seeds span
$\cC_P(E)$ after including all plane cusps and finite reflections.
\end{proof}

For a datum at one plane cusp, define
\begin{equation}\label{ar:eq:Eseries}
 \cE_P(\omega_h)=\sum_{\gamma\in\Gamma_P\backslash\Gamma}
                      \rho(\gamma)^{-1}\gamma^*\omega_h.
\end{equation}
Equivariance removes the choice of coset representative. Rationally
conjugate the finitely many other plane cusps to $P$ and sum the same
construction over them.

\begin{proposition}\label{ar:prop:convergence}
The series \eqref{ar:eq:Eseries} converges normally with every derivative
on compact subsets of $\mathscr X$. It defines a closed global two-form.
\end{proposition}
\begin{proof}
At the base point write $\gamma=pk_0$ with $p$ in plane coordinates and
$k_0\in\SO(3)$. The bottom row of $p$ is $(0,0,t)$, so
$t=\|e_3^*\gamma\|$ and $\det A=t^{-1}$. In the coefficient transported
pullback, the factor $\rho(p)$ in \eqref{ar:eq:seed-transform} cancels.
The remaining compact action has bounded norm. Furthermore,
\[
 |(h|_k A)(i)|
   =\left(\frac{\delta}{|ci+d|^2}\right)^{k/2}|h(Ai)|
   =(\Im Ai)^{k/2}|h(Ai)|
\]
is bounded. Equation~\eqref{ar:eq:seed-transform} therefore bounds the
norm of a summand by $Ct^{-k/2}$.
The cosets inject, with a bounded multiplicity depending only on the
congruence subgroup, into primitive integral bottom rows. Consequently, we have
\[
 \sum_{\gamma}\|\text{summand at }\gamma\|
 \le C'\sum_{0\ne v\in\ZZ^3}\|v\|^{-k/2}<\infty,
\]
since $k/2>3$. On a compact set, the norms of $vg$ and $v$ are uniformly
comparable. Rational conjugation changes the lattice by fixed denominators
and does not change convergence.

For derivatives, write $k_0g=p(k_0,g)k_1(k_0,g)$ for
$k_0\in\SO(3)$ and $g$ in a fixed neighborhood of the identity.
The parabolic and compact factors, and each fixed derivative of them,
are bounded uniformly in $k_0$. The parabolic character
$\delta^{k/2}$ is unchanged. A right derivative therefore produces
finitely many binary invariant derivatives of the cusp vector,
compact coefficient matrices, and constants from the character.
The binary derivatives are compositions of weight raising,
weight lowering, and rotation operators. Each is bounded in its
normalized weight: this follows from smoothness on the compact part of
a fundamental region and from the exponentially convergent Fourier
expansion at every cusp. Thus the same summable row majorant applies
to every fixed smooth seminorm. Termwise differentiation is legitimate,
and Lemma~\ref{ar:lem:seed} proves closedness.
\end{proof}

Choose holomorphic and antiholomorphic seeds whose face classes sum to
$\xi\in\cC_P(E)$, accounting for the factor $-1/(n+1)$ in
Lemma~\ref{ar:lem:seed}. Denote their sum over plane cusps by $\cE_P(\xi)$.
Its identity constant term represents exactly $\xi$.

\subsection{Constant terms as cohomological restrictions}\label{ar:sec:constantterms}
We show that fiber averaging computes actual restriction classes,
then unfold the sums. Neither step takes a limit at the boundary.

\smallskip\noindent\emph{Restriction and averaging on the fiber.}
Write $\mathcal N_R$ for the real unipotent radical of a maximal
parabolic $R$, $N_R=\Gamma\cap\mathcal N_R$ for its lattice, and
$\Gamma_R=\Gamma\cap R$. Its boundary model is
$e(R)=\mathcal N_R\times\mathfrak H$. Fixing the positive normal
coordinate defines a $\Gamma_R$-equivariant map
\[
 i_{R,t}:e(R)\longrightarrow\mathscr X.
\]
Since $e(R)$ is contractible and $\Gamma_R$ acts freely, the quotient
$\Gamma_R\backslash e(R)$ is a classifying space. Thus $i_{R,t}$ realizes
group restriction. Different positive heights give equivariantly
homotopic maps. At a flag use two normal coordinates.

We describe averaging with its coefficient transport. On a unipotent
fiber, a flat section has $s(\gamma u)=\rho(\gamma)s(u)$ for
$\gamma\in N_R$. Its periodic gauge is $\rho(u)^{-1}s(u)$, and its
flat differential is
\[
 D_{\mathcal N}=d+\rho(\theta)\wedge,
 \qquad \theta=u^{-1}du.
\]
Express a form in the left-invariant coframe and average each coefficient
over the compact nilmanifold $N_R\backslash\mathcal N_R$, with total
Haar volume one. This average is a cochain map: each invariant vector
field has integral zero on a derivative, while the matrices $\rho(\theta)$
and the structure constants of the coframe are constant. Its image is
the finite-dimensional Lie algebra cochain complex with coefficients $E$.
The coefficient transport is part of this averaging operation.

The nilmanifold theorem with unipotent coefficients
\cite[Theorem~1.1]{Kasuya} states that the canonical inclusion $\iota$ of
Lie algebra cochains induces an isomorphism on fiber cohomology.
Its hypotheses hold here: $\mathcal N_R$ is simply connected and nilpotent, $N_R$ is a
lattice, and the polynomial coefficient representation is unipotent on
$\mathcal N_R$. The same holds for the dual, and determinant twists are
trivial on the radical. Since the averaging map satisfies
$\operatorname{Av}\circ\iota=I$, it induces the inverse cohomology
isomorphism. Both maps commute with the lattice-preserving Levi action,
which preserves normalized Haar measure and transports the coframe and
coefficients together. Filtering by base degree gives the comparison for
the whole face bundle.

The same theorem applies at a flag, whose radical is the Heisenberg
group. Its Lie algebra differential retains
$d\theta_{13}=-\theta_{12}\wedge\theta_{23}$, so this application does not
replace the flag fiber by a torus.

The averaged pullback along $i_{R,t}$ therefore represents the
restriction class, with coefficient transport included.

\smallskip\noindent\emph{Unfolding in the coefficient trivialization.}
The invariant-frame coefficient of a form $\omega$ is
\[
 \widetilde\omega(g)(Y_1,\ldots,Y_r)
   =\rho(g)^{-1}\omega_{g\SO(3)}
          (dL_gY_1,\ldots,dL_gY_r),\qquad Y_i\in\mathfrak p,
\]
where $L_g$ is left translation. If $\omega$ is $\Gamma$-equivariant,
then $\widetilde\omega(\gamma g)=\widetilde\omega(g)$. Moreover the
invariant-frame coefficient of \eqref{ar:eq:Eseries} is simply
$\sum_\gamma\widetilde\omega_h(\gamma g)$. Constant term integrals below
are taken in this trivialization, which agrees with the periodic fiber
gauge above after restriction to a face.

For the unfolding, we record congruence conditions prime by prime. Let
$\QQ_p$ be the $p$-adic field, $\ZZ_p$ its integer ring, and
$\mathbb A=\RR\times\prod_p'\QQ_p$ the restricted product with respect to
$\ZZ_p$. Use Lebesgue measure at infinity and $\vol(\ZZ_p)=1$ at finite
primes, so $\QQ\backslash\mathbb A$ has volume one. This notation packages
the congruence coset calculation: rational row classes give global cosets,
and the Chinese remainder theorem records their congruence conditions.

At principal level we retain every determinant component. Restriction and
unipotent averaging preserve each component. After averaging to the
original level, the finite compact determinant contains all local units,
and rational lattice reduction gives one component over $\QQ$.

An \emph{induced section} $\phi$ is the coefficient valued cusp datum
transforming under the plane parabolic by \eqref{ar:eq:seed-transform}.
Its real component is the two-form above, and its finite components record
the congruence cosets.

Average \eqref{ar:eq:Eseries} in the invariant-frame trivialization
and unfold. The elementary row reductions use the matrices
\[
 s_{23}=\begin{pmatrix}1&0&0\\0&0&-1\\0&1&0\end{pmatrix},\qquad
 w_0=\begin{pmatrix}0&0&1\\1&0&0\\0&1&0\end{pmatrix},\qquad
 u(s,t)=\begin{pmatrix}1&s&t\\0&1&0\\0&0&1\end{pmatrix}.
\]
The first matrix interchanges the second and third coordinate lines.
Two planes are either equal or distinct. A line is either contained in
or complementary to the source plane. A full flag has three relative
positions. Gaussian elimination respecting these subspaces gives the
following complete list. In a zero row the indicated target root is
conjugated into the source binary root $E_{12}$, possibly with a minus
sign, which does not affect its average.
\[
\begin{array}{c|c|c}
\text{target}&\text{representative}&\text{contribution or vanishing root}\\ \hline
P&I&\xi\\
P&s_{23}&E_{13}\longmapsto -E_{12}\\
Q&I&E_{12}\longmapsto E_{12}\\
Q&w_0^{-1}&\text{two-root integral}\\
B&I,
\ s_{23},\ w_0^{-1}&E_{12},\ E_{13},\ E_{23}
                           \longmapsto \pm E_{12}
\end{array}
\]
The binary root average is zero by cuspidality. This statement includes
the polynomial coefficient: for a binary period form
$h(\tau)(\tau Y-X)^d\,d\tau$, the periodic gauge on a cusp with
$\tau=u+iv$ replaces the coefficient by $(ivY-X)^d$. Its remaining
$u$-dependence is $h(u+iv)$, whose average over the cusp width is zero.
The gradient and contraction representatives of radical cohomology have
the same property. Thus none of the zero entries discards a polynomial
factor before averaging.

The identity plane cell returns $\xi$, since radical volume is one.
For the complementary line cell, the target radical intersects the
conjugated source parabolic radical trivially, leaving the integral
\begin{equation}\label{ar:eq:intertwiner}
 (T\phi)(g)=\int_{\mathbb A^2}
               \phi(w_0^{-1}u(s,t)g)\,ds\,dt.
\end{equation}
Equivalently this is the product of its real integral and its finite
congruence integrals when the datum is factorizable. The unfolding is
absolutely convergent by Proposition~\ref{ar:prop:convergence}. The
integrated radical quotient is compact, so Fubini applies. The preceding
zero mode argument gives the cohomological identities
\begin{equation}\label{ar:eq:constantterms}
 r_P\cE_P(\xi)=\xi,\qquad r_Q\cE_P(\xi)=T\xi,\qquad
 r_B\cE_P(\xi)=0
\end{equation}
for all $\xi\in\cC_P(E)$. Here $r_B$ denotes restriction at each flag.
These are cohomological restriction identities at finite height.

It remains to exclude a nonmixed face component. For a maximal face stabilizer $R=N_R\rtimes L_R$, the only terms in total
degree two of Hochschild--Serre are $(1,1)$ and $(0,2)$, because
$L_R$ has rational cohomological dimension one. Thus there is an exact sequence
\[
 0\longrightarrow H^1(L_R,H^1(N_R,E))
 \longrightarrow H^2(R,E)
 \xrightarrow{\operatorname{res}_{N_R}} H^2(N_R,E)^{L_R}
 \longrightarrow0.
\]
Choose a flag in this face. Its stabilizer contains $N_R$, so zero
restriction at that flag implies zero restriction to $N_R$. Exactness
puts the face class in the mixed term. The restriction of this term
to a flag is its binary cusp restriction: the analogous mixed edge map
for the binary cusp stabilizer is injective for the same
cohomological dimension reason.

\smallskip\noindent\emph{The dual series and the boundary relation.}
We also need the construction for the dual coefficient system. Let $J$ be
the antidiagonal permutation matrix and define
\[
 \jmath(g)=J(g^{-1})^{\mathsf t}J.
\]
This exchanges $P,Q$ and pulls $E$ back to its dual, after the fixed linear
change of coefficient coordinates by $J$. Apply the preceding construction
to $\jmath(\Gamma)$ and pull back. For each $\eta\in\cC_Q(E^*)$, it gives
a closed global class $\cE_Q^\vee(\eta)$ with
\begin{equation}\label{ar:eq:dualconstantterms}
 r_Q\cE_Q^\vee(\eta)=\eta,\qquad
 r_P\cE_Q^\vee(\eta)=T^\vee\eta,\qquad
 r_B\cE_Q^\vee(\eta)=0.
\end{equation}
The normal height exponent and convergence bound are unchanged.
Here $T^\vee$ is the second constant term map. We assume neither
invertibility nor an adjoint normalization.

\begin{proposition}\label{ar:prop:graph}
Let $n\ge4$. If a global degree two class with coefficients $E$ has only
mixed cuspidal maximal face restrictions $(\xi_P,\xi_Q)$ and zero flag
restrictions, then $\xi_Q=T\xi_P$.
\end{proposition}
\begin{proof}
Pair the class with $\cE_Q^\vee(\eta)$ in
Lemma~\ref{ar:lem:stokes}, for arbitrary $\eta\in\cC_Q(E^*)$:
\[
 \langle\xi_P,T^\vee\eta\rangle_P+\langle\xi_Q,\eta\rangle_Q=0.
\]
Pair $\cE_P(\xi_P)$ with the same dual class:
\[
 \langle\xi_P,T^\vee\eta\rangle_P+\langle T\xi_P,\eta\rangle_Q=0.
\]
Subtract and use the perfect line pairing \eqref{ar:eq:facepair}.
This proves graph containment without computing the dimension of the
global image.
\end{proof}

\subsection{The real constant term entry}\label{ar:sec:real}
To compute the required real entry, take $h(\tau)=(\tau+i)^{-k}$.
This is a local holomorphic lowest rotation weight vector, not a modular
form. In the disk coordinate $w=(\tau-i)/(\tau+i)$, write a vector as
$(\tau+i)^{-k}F(w)$. The weight $k$ action differentiates to
$\partial_w$ and $w^2\partial_w+kw$, up to fixed nonzero normalizations.
They send $w^m$ to $m w^{m-1}$ and $(m+k)w^{m+1}$. Thus the constant
vector generates all rotation weight vectors
$(\tau-i)^m/(\tau+i)^{k+m}$, with no vanishing raising factor.

The integral is equivariant for the binary Levis identified by the block
permutation. Its value on the lowest vector therefore determines its
values on these one-dimensional rotation weight spaces. Conjugation gives
the antiholomorphic component.

Set
\[
 L=w_0^{-1}u(s,t)=\begin{pmatrix}0&1&0\\0&0&1\\1&s&t\end{pmatrix},
 \qquad r^2=1+s^2+t^2.
\]
For the Gram matrix $LL^{\mathsf t}$, the plane coordinates and transported
coefficient in \eqref{ar:eq:seed} are
\[
 u=(s,t)/r^2,\qquad \tau=\frac{-st+ir}{1+s^2},\qquad
 aX+bY+(i/r)Z,
\]
where
\[
 a=\frac{s-it/r}{1+s^2},\qquad
 b=-\frac{1+ist/r}{1+s^2}.
\]
Let $v_1=E_{12}$, $v_2=E_{13}$ be line-radical tangent vectors and
$x_0=E_{23}$ the horizontal binary tangent. With
$\beta=du_1-\tau du_2$, differentiation gives
\begin{equation}\label{ar:eq:tangententry}
 (\beta\wedge d\tau)(v_2,x_0)=-i/r^3.
\end{equation}
To verify this differentiation, set $G=LL^{\mathsf t}$,
$R=G_{33}$, $v=G_{\{1,2\},3}$, and
$C=G_{\{1,2\},\{1,2\}}-vv^{\mathsf t}/R$. Then the plane coordinates are
\[
 u=v/R,\qquad \Re\tau=C_{12}/C_{22},\qquad
 \Im\tau=\sqrt{\det C}/C_{22}.
\]
For a tangent matrix $V$, substitute
$\dot G=L(V+V^{\mathsf t})L^{\mathsf t}$ into these formulas.
For $V=v_2,x_0$ the resulting derivatives satisfy
$d\tau(x_0)=s\,d\tau(v_2)+1$ and
$\beta(x_0)=s\beta(v_2)+t/r^2$. Substitution gives
\eqref{ar:eq:tangententry}.

This entry detects the mixed cohomology class. For a line-radical closed
pair $(\alpha,\beta_0)$, its polynomial coordinate is $Z\alpha-Y\beta_0$.
Exact pairs have zero contraction. At
$X=Z=0,Y=1$, this coordinate reads the negative $v_2$ component.
To exclude a contribution in radical degree two, act by
$a_v=\operatorname{diag}(e^{2v},e^{-v},e^{-v})$. Under the transported
action $\rho(a_v)^{-1}a_v^*$, a monomial of $X$-degree $j$ in radical
cohomological degree one has character $e^{(3-n+3j)v}$.
The integral \eqref{ar:eq:intertwiner} has character $e^{(6-k)v}$:
the source plane character contributes $e^{-kv}$, while changing its
two radical integration coordinates contributes $e^{6v}$.
Since $6-k=3-n$, the degree-one component has $j=0$. A radical degree two
term would require $6-n+3j=3-n$, or $j=-1$, and is absent.
The binary rotation weight $k=n+3$ is the extreme polynomial weight
$n+1$ plus the differential weight two. Hence the holomorphic mixed
component is one-dimensional for this local datum, and the coefficient
just computed determines its scalar.

Relative to the line seed
$h(\tau)(\tau Z-Y)^n(\tau\,dv_1+dv_2)\wedge d\tau$, where $dv_i$
are the covectors dual to the radical coordinates, the scalar is obtained
by dividing
\begin{equation}\label{ar:eq:realintegral}
 I_n=-i\int_{\RR^2}r^{-3}h(\tau)b^n\,ds\,dt
\end{equation}
by $(-1)^n(2i)^{-k}$, its corresponding value at $\tau=i$.
Use $s=\sqrt{r^2-1}\sin\theta$, $t=\sqrt{r^2-1}\cos\theta$, and write
$c=\cos\theta$, $v=\sin\theta$. Then, we have
\[
 \tau=\frac{irc+v}{c-irv},\qquad
 b=-\frac{c-iv/r}{c-irv},\qquad
 \tau+i=\frac{i(r+1)e^{-i\theta}}{c-irv},\qquad
 r^{-3}ds\,dt=r^{-2}dr\,d\theta.
\]
The angular integral is
\[
 \int_0^{2\pi}e^{ik\theta}(c-irv)^3(c-iv/r)^n\,d\theta
       =\frac{2\pi(r+1)^k}{2^k r^n}.
\]
Indeed $c-iav=((1-a)e^{i\theta}+(1+a)e^{-i\theta})/2$, and
$k=n+3$ means that only the negative frequency from every factor contributes
to the constant term. The remaining integral is
$\int_1^\infty r^{-n-2}\,dr=1/(n+1)$. Therefore the real entry is
\begin{equation}\label{ar:eq:cinfty}
 c_\infty=-\frac{2\pi i}{n+1}\ne0.
\end{equation}
Changing orientation or period normalization rescales this nonzero entry.
The conjugate component has the conjugate scalar. Reflection projection
uses the full twisted coefficient systems, so equivariance preserves
nonvanishing without identifying the two untwisted parities.

\subsection{The finite entries and injectivity}\label{ar:sec:finite}
We prove that bad-plane data map injectively to a good line. The prime $3$
determines support, and lattice counts at the other primes give the Euler
factors. We then combine these with the real entry and reflection
signs.

\smallskip\noindent\emph{The level three entry.}
At the prime $3$, let
\[
 K_1=\{g\in\GL_3(\ZZ_3):e_3^*g=e_3^*\pmod3\}.
\]
There are two $K_1$-orbits on planes and two on lines, according to whether
their reductions are annihilated by $e_3^*$. A plane section supported on
the bad-plane orbit is supported on the double coset $PK_1$. It is
determined by its value $v_0$ at the identity. This value is fixed by the
$\GL_2(\ZZ_3)$ subgroup of $P\cap K_1$. It is therefore zero unless the
binary datum has such a fixed vector. Normalize a nonzero section $\phi$ by
$\phi(1)=v_0$, and put
\[
 g_0=\begin{pmatrix}0&0&1\\0&1&0\\1&0&0\end{pmatrix}.
\]
This represents a good-line cusp. We have
\[
 w_0^{-1}u(s,t)g_0=\begin{pmatrix}0&1&0\\1&0&0\\t&s&1\end{pmatrix}
       \in PK_1\quad\Longleftrightarrow\quad s,t\in3\ZZ_3.
\]
For necessity, its last row must be a scalar multiple of the last row of a $K_1$ matrix. The last entry is one, so the scalar is a unit and the other
two entries must lie in $3\ZZ_3$. Conversely the displayed matrix itself lies
in $K_1$ when that condition holds. The integral is therefore supported on
$(3\ZZ_3)^2$, its integrand is $v_0$, and
\begin{equation}\label{ar:eq:local3}
 (T_3\phi)(g_0)=\frac19v_0.
\end{equation}
Here $T_3$ denotes the local factor at the place $3$ of $T$, not the
rank three Hecke operator $T_3(p,1)$ used in the residue argument.

For even $n$, the central $-I_2$ kills the odd degree bad-plane binary
cohomology. Assume now that $n$ is odd. The matrix $-I_3$ fixes every
rational line and plane and acts on
$E_n$ by $(-1)^{n+3}=1$. Thus the restricted functional and its connecting
class extend to $\langle K^\dagger,-I_3\rangle$, the stabilizer of the
projective level covector. At the bad plane, combining $-I_3$ with the
binary sign $\operatorname{diag}(-I_2,1)$ supplies the independent normal
sign $\operatorname{diag}(I_2,-1)$.

The normal inducing character is therefore trivial on all units at $3$,
not only on $1+3\ZZ_3$. It is unramified elsewhere and trivial at the real
place. Applying it to a positive rational prime then shows that its value
on that prime's uniformizer is one, since every other local value is one.
Thus the character is globally trivial. The binary datum is fixed by the
full finite maximal compact subgroup, so it has level one and even weight
$k=n+3$.

The bad plane and the good line both have full-level binary Levi.
Indeed the stabilizer of the good line $\QQ e_3$ in $K^\dagger$ is
\[
 \left\{\begin{pmatrix}D&0\\3u&1\end{pmatrix}:
             D\in\GL_2(\ZZ),\ u\in\ZZ^{1\times2}\right\},
\]
whereas the bad-plane stabilizer has upper block form with the same
binary group. Their coefficient systems have different determinant
factors, which we retain below.

\smallskip\noindent\emph{The unramified factors as lattice sums.}
We compute each remaining local factor directly from classical Hecke
operators. Fix a normalized eigenform $h$ of full level and weight $k$.
Its finite-spherical datum is the lift fixed by
$\GL_2(\widehat\ZZ)$, where $\widehat\ZZ=\prod_p\ZZ_p$.
On $\GL_2^+(\RR)$ put
\[
 F_h(g)=(\det g)^{k/2}(ci+d)^{-k}h(gi),\qquad
 g=\begin{pmatrix}a&b\\c&d\end{pmatrix}.
\]
For a finite adelic matrix $g_f$, choose $\gamma\in\GL_2^+(\QQ)$ with
$\gamma g_f\in\GL_2(\widehat\ZZ)$ and set
$v_h(g_\infty,g_f)=F_h(\gamma g_\infty)$.
Such a $\gamma$ is a rational basis change for the lattice specified by
$g_f$. Its determinant sign can be changed by an element of $\GL_2(\ZZ)$.
Two positive determinant choices differ by $\SL_2(\ZZ)$, so this is
well defined. Reflections supply the other real component.
Positive scalar matrices act trivially, and the even weight gives the
same assertion for the negative scalar. Write $v_h|A$ for right
translation by $A$ at one finite place. No irreducible local
representation is needed for this notation.

Let $p\ne3$, and for a lattice $\Lambda\subset\ZZ_p^2$ of index $p^r$
choose a matrix $A_\Lambda$ whose columns are a basis of $\Lambda$.
Right $\GL_2(\ZZ_p)$-invariance makes the basis choice immaterial.
The classical Hecke normalization gives
\begin{equation}\label{ar:eq:lattice-hecke}
 \sum_{[\ZZ_p^2:\Lambda]=p^r}v_h|A_\Lambda
       =p^{r(1-k/2)}a_h(p^r)v_h.
\end{equation}
Here is the left/right comparison in this normalization. Column lattices
have representatives $A=\left(\begin{smallmatrix}a&0\\b&d\end{smallmatrix}\right)$,
with $ad=m=p^r$ and $0\le b<d$. At finite coordinate one, right translation
by $A$ changes the real function to $F_h(A^{-1}g)$. For
$S=\left(\begin{smallmatrix}0&-1\\1&0\end{smallmatrix}\right)$, we have
\[
 m A^{-1}=S A^{\mathsf t}S^{-1}.
\]
The scalar $m$ and left multiplication by $S$ do not change $F_h$.
The transposes $A^{\mathsf t}$ run over $\mathcal H_m$, so their sum is
$m^{1-k/2}T_m$, evaluated at $S^{-1}g$. The eigenvalue equation and
$\SL_2(\ZZ)$-invariance remove that last $S^{-1}$. This proves
\eqref{ar:eq:lattice-hecke} at finite coordinate one, and hence everywhere:
a full-level adelic function is determined there by the same rational
lattice reduction used to define $v_h$.

Call $\Lambda$ \emph{cyclic} if $\ZZ_p^2/\Lambda$ is cyclic. A lattice
of index $p^r$ is noncyclic exactly when it is contained in $p\ZZ_p^2$. In
that case it is $p\Lambda'$ with $[\ZZ_p^2:\Lambda']=p^{r-2}$.
Subtracting these lattices in \eqref{ar:eq:lattice-hecke} gives
\begin{equation}\label{ar:eq:cyclic-hecke}
 \sum_{\substack{[\ZZ_p^2:\Lambda]=p^r\\\Lambda\ \mathrm{cyclic}}}
       v_h|A_\Lambda
  =p^{r(1-k/2)}
       \{a_h(p^r)-p^{k-2}a_h(p^{r-2})\}v_h,
\end{equation}
where $a_h(p^{-1})=a_h(p^{-2})=0$.

We now relate these lattices to the integral. In
$L(s,t)=w_0^{-1}u(s,t)$, integral changes of $s,t$ are right
$\GL_3(\ZZ_p)$-operations. For
$\max(|s|_p,|t|_p)=p^r$, $r\ge1$, the primitive pair
$p^r(s,t)$ modulo $p^r$ defines the cyclic lattice
\[
 \Lambda_{s,t}=\{(u,v)\in\ZZ_p^2:su+tv\in\ZZ_p\}.
\]
Each such lattice has exactly $(1-p^{-1})p^r$
primitive defining pairs modulo $p^r$, differing by units. Each pair
corresponds to a coset of $\ZZ_p^2$ in the integration variables and
therefore has measure one.

Choose a basis $A_\Lambda$ with determinant $p^r$. Lift its columns
$(u,v)$ to $(-su-tv,u,v)$ in $\ZZ_p^3$, and complete them by a vector
on which the row $(1,s,t)$ has value $p^{-r}$. Such a vector exists by
B\'ezout's identity. The resulting matrix $k_0$ is integral and has
determinant one, and
\[
 L(s,t)k_0=\begin{pmatrix}A_\Lambda&*\\0&p^{-r}\end{pmatrix}.
\]
The plane transformation law \eqref{ar:eq:seed-transform}, with the
product formula for rational determinants, gives the finite factor
$|\det A_\Lambda|_p^{k/2}=p^{-rk/2}$ times $v_h|A_\Lambda$.
The normal character is trivial. Thus \eqref{ar:eq:cyclic-hecke} shows
that the shell of valuation $r$ contributes
\[
 (1-p^{-1})p^{-r(k-2)}
        \{a_h(p^r)-p^{k-2}a_h(p^{r-2})\}v_h.
\]
The integral over $\ZZ_p^2$ contributes $v_h$.

Put $t_p=p^{-(k-2)}$ and
$H_p(T)=\sum_{r\ge0}a_h(p^r)T^r
       =(1-a_h(p)T+p^{k-1}T^2)^{-1}$.
Summing the shells gives the scalar
\begin{align*}
 c_p
 &=1+(1-p^{-1})\{(1-p^{k-2}t_p^2)H_p(t_p)-1\}\\
 &=\frac{1-a_h(p)t_p/p+p^{k-1}t_p^2/p^2}
         {1-a_h(p)t_p+p^{k-1}t_p^2}
   =\frac{L_p(h,k-2)}{L_p(h,k-1)}.
\end{align*}
Lemma~\ref{lem:outer-values} gives absolute convergence.
The Chinese remainder theorem separates the lattice sums into commuting
prime power Hecke sums, acting on $h$ by $c_p$. We may therefore take their
product. Combining it with \eqref{ar:eq:local3} and the real integral gives
\begin{equation}\label{ar:eq:globalentry}
 -\frac{2\pi i}{9(n+1)}
       \frac{L^{(3)}(h,k-2)}{L^{(3)}(h,k-1)}\ne0.
\end{equation}
The scalar is expressed in the induced section and differential form
normalizations fixed above. Only its nonvanishing is used.

\smallskip\noindent\emph{Reflection signs and injectivity.}
The reflection conventions are important here. In binary degree
$d=n+1$, the plane coefficient is $\det^{-1}\Sym^d((\CC^2)^*)$,
whereas the line coefficient has no binary determinant twist. Hence a
reflection acts with opposite signs on the two coefficient systems if
that twist is suppressed. All maps above retain these determinant factors and are equivariant
on the full coefficient systems.
At the real place the holomorphic and antiholomorphic scalars are
complex conjugates, $-2\pi i/(n+1)$ and $2\pi i/(n+1)$, as required
by the reflection. Consequently equivariance survives the finite
reflection projection.

The level one cusp space has a simultaneous Hecke eigenbasis
\cite[Section~5.8]{DiamondShurman}. The correspondence preserves each
eigenform and acts on its holomorphic and antiholomorphic components by
the nonzero entries above. Distinct eigenvalue systems determine distinct
Fourier expansions, so images of normalized eigenforms cannot cancel.
Reflection projection commutes with the correspondence. Injectivity thus
descends to the arithmetic component required in Proposition~\ref{thm:J-vanishes}.

\section{The reduced norm kernel}\label{mf:app:kernel}

We prove Lemma~\ref{mf:lem:kernel} over $k=\FF_2$. In $S=k[a,b,c]$, the
generators and invariants of Section~\ref{sec:odd-alt} have weights
$1,2,3$ and $4,6,7$, respectively. The point is to obtain a basis valid
in every degree, not only in the finite quotient used below.

\subsection{A projector defined by linear substitutions}

Let $\mathcal W$ be the six permutation matrices and let
\[
 \mathcal U=
 \left\{\begin{pmatrix}1&s&t\\0&1&u\\0&0&1\end{pmatrix}:
                         s,t,u\in k\right\}.
\]
For a symmetric polynomial $g$, define
\[
 \pi g(\vect)=\sum_{P\in\mathcal W}\sum_{U\in\mathcal U}g(UP\vect),
 \qquad \vect=(x,y,z)^{\mathsf t}.
\]
This is a sum of substitutions in characteristic two, not a normalized
average. Its value is symmetric, since permuting $x,y,z$ permutes
$\mathcal W$.

For the remaining projector identities, let $\mathscr H$ be the following
set of substitutions, in the notation of \eqref{mf:eq:Tlist}:
\[
\begin{gathered}
 (y,z,x),\quad(y,z,x+y),\quad(z,y,x+z),\\
 (x,z,y),\quad(x,z,x+y),\quad(x,z,y+z),\\
 (x+y,z,x),\quad(x+y,z,y),\quad(x,y,z),\\
 (x,y,x+z),\quad(x,y,y+z),\quad(y,x+z,x),\\
 (x+z,y,z),\quad(x,y+z,y),\quad(x,y+z,z).
\end{gathered}
\]
Set $\Theta h=\sum_{\lambda\in\mathscr H}h(\lambda)$.
Expansion and cancellation of identical triples of arguments, using
symmetry of $g$, give
\begin{equation}\label{mf:eq:projector}
 \bn\pi g=0,\qquad \pi g=g+\Theta\bn g.
\end{equation}
These operator identities hold in every degree. They give $\pi^2=\pi$
and $\operatorname{im}\pi=\ker\bn$, with $\pi$ acting identically on
that kernel. The product identity \eqref{mf:eq:product} makes $\pi$
linear over $R=k[d_4,d_6,d_7]$.

In the notation of Lemma~\ref{mf:lem:kernel}, direct substitution gives
\begin{equation}\label{mf:eq:projected-exps}
 \bigl(\pi(a^2b),\pi(ab^2),\pi(b^3),\pi(b^2c),
       \pi(ac^2),\pi(bc^2),\pi(c^3),\pi(b^5)\bigr)
       =(K_0,K_1,K_2,K_3,K_4,K_5,K_6,K_7).
\end{equation}
In particular, we have $\pi K_j=K_j$.

\subsection{A basis of the invariant quotient}

Put $I=(d_4,d_6,d_7)\subset S$. We use lexicographic monomial order
$a>b>c$. Polynomial division and Buchberger's criterion
\cite[Chapter~2]{CoxLittleOShea} give the monic Gr\"obner basis
\begin{align*}
 &a^4+ac+b^2,\quad a^3b^2+ac^2+b^2c,\quad
 a^3c+a^2b^2+abc+c^2,\\
 &a^2b^3+ab^2c+b^4+bc^2,\quad a^2bc+ac^2,\quad
 a^2c^2+b^4,\\
 &ab^4+c^3,\quad ab^3c,\quad abc^2+b^3c+c^3,\quad
 ac^3+b^5,\\
 &b^6,\quad b^4c+bc^3,\quad b^2c^2,\quad c^4.
\end{align*}
Thus a basis of $S/I$ consists of the residue classes of the $28$
standard monomials
\[
\begin{split}
 &1,b,b^2,b^3,b^4,b^5,c,bc,b^2c,b^3c,c^2,bc^2,c^3,bc^3;\\
 &a,ab,ab^2,ab^3,ac,abc,ab^2c,ac^2;\\
 &a^2,a^2b,a^2b^2,a^2c;\qquad a^3,a^3b.
\end{split}
\]
Here a standard monomial is one not divisible by the leading monomial
of any element of the displayed Gr\"obner basis.

We need a basis adapted to the projector. Let
\begin{equation}\label{mf:eq:complement-exps}
\begin{split}
 \mathscr M=\{&1,a,b,a^2,c,ab,a^3,b^2,ac,bc,ab^2,a^2c,\\
              &c^2,b^3,a^2b^2,ac^2,bc^2,b^4,b^3c,bc^3\}.
\end{split}
\end{equation}
The residue classes of
\[
 K_0,\ldots,K_7,\qquad (1+\pi)m\quad(m\in\mathscr M)
\]
form another basis of $S/I$. The following normal forms verify the
change of basis over $\FF_2$. Here $\operatorname{NF}_I$ is the remainder on division
by the displayed Gr\"obner basis:
\[
\begin{array}{c|l@{\qquad}c|l}
 j&\operatorname{NF}_I(K_j)&j&\operatorname{NF}_I(K_j)\\ \hline
 0&a^2b+ac&4&ab^3+b^2c\\
 1&a^3b+a^2c&5&ab^2c\\
 2&abc+b^3+c^2&6&c^3\\
 3&b^2c&7&b^5
\end{array}
\]
For $m\in\mathscr M$, the normal form of $(1+\pi)m$ is $m$,
except in the five cases
\[
\begin{array}{c|l}
 m&\operatorname{NF}_I((1+\pi)m)\\ \hline
 ab^2&a^3b+a^2c+ab^2\\
 a^2c&a^3b\\
 b^3&abc+c^2\\
 ac^2&ab^3+ac^2+b^2c\\
 bc^2&ab^2c+bc^2.
\end{array}
\]
Applying the substitutions and reducing by the Gr\"obner basis gives
these formulas. They span each weighted degree of $S/I$. In degree five,
for example, they successively give $a^3b$, $a^2c$, $ab^2$, and $bc$.
Degree six gives $c^2$, $abc$, $b^3$, and $a^2b^2$, and the other degrees
are equally triangular. The $28$ spanning classes therefore form a basis.

\subsection{Lifting the basis to the polynomial ring}

To avoid assuming algebraic independence of the invariants, begin with
the formal polynomial ring $R'=k[u_4,u_6,u_7]$, acting on $S$ by
$u_i\mapsto d_i$. The $28$ homogeneous lifts above generate $S$ over
$R'$: reduce a homogeneous polynomial modulo $I$, subtract the
corresponding linear combination of the lifts, and write the remainder
as $d_4s_4+d_6s_6+d_7s_7$. Each $s_i$ has smaller weighted degree,
so induction applies.

The degrees of the lifts have generating polynomial
\[
 \frac{(1-t^4)(1-t^6)(1-t^7)}{(1-t)(1-t^2)(1-t^3)}.
\]
Consequently the free graded $R'$-module on these lifts has the same
Hilbert series as $S$. The surjection just constructed is therefore
an isomorphism in each degree. In particular, the map $R'\to S$ is
injective, the invariants are algebraically independent, and these
lifts form an $R$-basis of $S$.

The projector is identity on the $K_j$ and zero on $(1+\pi)m$.
Taking its image in this $R$-basis proves
\[
 \ker\bn=\bigoplus_{j=0}^7 RK_j,
\]
as asserted in Lemma~\ref{mf:lem:kernel}.

\section{AxiomProver Lean certificate}\label{sec:AxiomProver}

This paper was written through human--AI collaboration with AxiomProver,
an AI system for mathematical research through formal proof under
development at Axiom Math. The collaboration included literature search,
numerical experimentation, and the proposal and rejection of candidate
strategies before the proofs took shape. The authors set the direction,
made the mathematical judgments, and take full responsibility for the
correctness of the results.

Theorems~\ref{thm:main} and \ref{thm:symmetric-classification} and the
essential argument leading to them are formalized unconditionally in
Lean~4.35.0. The repository contains the formal statements, proofs, and
instructions for checking the formalization:
\begin{center}
\url{https://github.com/AxiomMath/TernaryPoly}
\end{center}
The formalization was developed collaboratively by humans and AxiomProver
and audited by humans and AI. It depends on Lean~4.35.0 and Mathlib, with
the dependency revisions pinned in the repository's Lake configuration.
The proof development is free of admissions and uses only Lean's standard
logical axioms. Mathematical and source attribution is retained in the
source headers.

Comparator checks the two main theorem statements against a separate
challenge containing independently specified definitions. The solution
imports the proved targets. Instructions for running Comparator and its
independent kernel checks are included in the repository.

\section*{Acknowledgements}
\noindent The authors thank Henrik Bachmann and Jan-Willem van Ittersum for correspondence
related to the conjecture, and A.~Raghuram for correspondence related to
Eisenstein cohomology.

\end{document}